\documentclass{amsart}
\usepackage{amsmath,amssymb,amscd,verbatim,xspace,amsthm,tensor}
\usepackage{graphicx,enumerate,tikz,tikz-cd,appendix}
\usepackage{latexsym, epsfig, color}
\usepackage{fullpage}
\usepackage{pdfsync}
\usepackage{unicode-math}
\usepackage{mathrsfs}
\usepackage{amsfonts}

\numberwithin{equation}{section}

\theoremstyle{plain}
\newtheorem{theorem}{Theorem}[section]
\newtheorem{proposition}[theorem]{Proposition}
\newtheorem{corollary}[theorem]{Corollary}
\newtheorem{lemma}[theorem]{Lemma}

\newtheorem{hypothesis}[theorem]{Hypothesis}

\theoremstyle{definition}
\newtheorem{definition}[theorem]{Definition}

\newtheorem{example}[theorem]{Example}

\theoremstyle{remark}
\newtheorem*{remark}{Remark}

\begin{document}
	\title[Invariant Part of Class Groups and Distribution of Relative Class Group]{Invariant Part of Class Groups and Distribution of Relative Class Group}
	\date{}
	\author{Weitong Wang}
	\address{College of Mathematical Sciences, Shaw Hall of Harbin Engineering University, No. 145, Nantong Street, Nangang District, Harbin, Heilongjiang Province, 150001 China}
	\email{weitongwang@hrbeu.edu.cn}
	\begin{abstract}
		We generalize the work of Roquette and Zassenhaus on the invariant part of the class groups to the relative class groups.
		Thus, we can show some statistical results as follows.
		For abelian extensions over a fixed number field $K$, we show infinite $C_p$-moments for the Sylow $p$-subgroup of the relative class group when $p$ divides the degree of the extension.
		For sextic number fields with $A_4$-closure, we can show infinite $C_2$-moments for the Sylow $2$-subgroup of the relative class group when the extensions run over a fixed Galois cubic field.
	\end{abstract}
	\maketitle
	\tableofcontents
\section{Introduction}\label{section: intro}
In this paper, we are mainly interested in the distribution of (relative) class groups of number fields.
Let us use an example to explain the notion briefly.
Let $\mathcal{E}$ be the set of quadratic number fields ordered by the absolute discriminant $d$.
Define
\begin{equation*}
	N_{d}(\mathcal{E},X)=\#\{K\in\mathcal{E}\mid d_K<X\},
\end{equation*}
which is the function that counts quadratic number fields ordered by discriminant.
Then we can define the notations of probability and moments for class groups.
Let $A$ be a finite abelian group, and $p$ be a rational prime.
Define the $p$-rank of $A$, denoted by $\operatorname{rk}_pA$, as the largest number $r$ so that there exists some injective group homomorphism $C_p^r\to A$.
For each non-negative integer $r$, define
\begin{equation*}
	\begin{aligned}
		\mathbb{P}_{d,\mathcal{E}}(\operatorname{Cl}_K\cong A):=
		&\lim_{X\to\infty}
		\frac{
			\#\{K\in\mathcal{E}\mid d_K<X\text{ and }\operatorname{Cl}_K\cong A\}
		}
		{
			N_{d}(\mathcal{E},X)
		}
		\\
		\mathbb{P}_{d,\mathcal{E}}(\operatorname{rk}_p\operatorname{Cl}_K\leq r):=
		&\lim_{X\to\infty}
		\frac{
			\#\{K\in\mathcal{E}\mid d_K<X\text{ and }\operatorname{rk}_p\operatorname{Cl}_K\leq r\}
		}
		{
			N_{d}(\mathcal{E},X)
		},
	\end{aligned}	
\end{equation*}
and call it the probability of $\operatorname{Cl}_K\cong A$, resp. $\operatorname{rk}_p\operatorname{Cl}_K\leq r$.
Define the $A$-moment of $\operatorname{Cl}_K$ to be
\begin{equation*}
	\mathbb{E}_{d,\mathcal{E}}(\lvert\operatorname{Hom}(\operatorname{Cl}_K,A)\rvert):=
	\lim_{X\to\infty}
	\frac{
		\sum_{
			\substack{
				K\in\mathcal{E}\\
				d_K<X
			}
		}
		\lvert\operatorname{Hom}(\operatorname{Cl}_{K},A)\rvert
	}{
		N_{d}(\mathcal{E},X)
	}.
\end{equation*}
When $p$ is an odd prime, Cohen and Lenstra~\cite{CL84} gives the prediction for the probability distribution of $\operatorname{Cl}_K\otimes\mathbb{Z}_p\cong A$, where $A$ is a finite abelian $p$-group.
When the set of fields $\mathcal{E}$ is generalized to the set of Galois $G$-fields, where $G$ is a finite group, Cohen and Martinet~\cite{CM90} gives the corresponding generalization of Cohen-Lenstra Heuristics.
The method of Cohen and Martinet could be applied to non-Galois cases and obtain the corresponding predictions.
See the author and Wood~\cite{wang2021moments}.
Though this area is widely open, there are some proven results.
A famous one is obtained by Davenport and Heilbronn~\cite{Davenport1971Cubic}.
In the context of distribution of class groups, we can translate their result into the following: 
for quadratic number fields, the $C_3$-moment of class groups is exactly what is predicted by Cohen-Lenstra-Martinet Heuristics.
Another result is the work of Alex Smith~\cite{smith2022distribution}, where he proves that the distribution of $\operatorname{Cl}_K\otimes\mathbb{Z}_2$ satisfies the Gerth's Conjecture~\cite{Gerth1987} when $K$ runs over quadratic number fields.

We discuss in detail the case where $p=2$ and $K$ quadratic.
If we apply Genus Theory (see Ishida~\cite{ishida1976genus} for example) to quadratic number fields, then
\begin{equation*}
	\omega(d_K)-1\leq\operatorname{rk}_2\operatorname{Cl}_K\leq\omega(d_K),
\end{equation*}
where $\omega(n)$ counts all the distinct prime factors of an integer $n$.
This implies that for each non-negative integer $r$, we have that
\begin{equation*}
	\mathbb{P}_{d,\mathcal{E}}(\operatorname{rk}_2\operatorname{Cl}_K\leq r)=0,
\end{equation*}
and
\begin{equation*}
	\mathbb{E}_{d,\mathcal{E}}(\lvert\operatorname{Hom}(\operatorname{Cl}_K,C_2)\rvert)=+\infty.
\end{equation*}
We can call this phenomenon ``zero-probability'', resp. ``infinite moment'' in short.
This means that the distribution of $\operatorname{Cl}_K\otimes\mathbb{Z}_2$ is qualitatively different from other Sylow $p$-subgroups of $\operatorname{Cl}_K$.
In particular, the original Cohen-Lenstra Heuristics \emph{cannot} be applied to this case (there is literally no prediction from the heuristics).
This is a motivation for Gerth's Conjecture.
And we will follow this phenomenon in this paper.

Let us introduce some notations so that we can make statements clearly.
\begin{definition}\label{def:set of primes}
    Fix a number field $K$.
    Denote by $\mathcal{P}_K$ the set of all finite primes of $K$.
    When $K=\mathbb{Q}$, let $\mathcal{P}=\mathcal{P}_{\mathbb{Q}}$ be the set of all finite rational primes.
    Given a set $\mathfrak{R}$ of ideals of $K$, define $\mathcal{P}_{\mathfrak{R}}$ as the set of all primes in $\mathfrak{R}$.
\end{definition} 
An example of $\mathfrak{R}$ is the set given by an ideal class of $\operatorname{Cl}_K$.
Since there are multiple ways to describe field extensions, we give the following two definitions to make terms like ``the set of all non-Galois cubic number fields'' precise. 
\begin{definition}\label{def:Gamma-fields}
    Let $G$ be a finite group.
    Fix a number field $k$.
    Define a \emph{$G$-extension of $k$} to be an isomorphism class of pairs $(K,\psi)$, where $K$ is a Galois extension of $k$, and $\psi:G(K/k)\cong G$ is an isomorphism.  
    An isomorphism of pairs $(\tau,m_\tau):(K,\psi)\to(K',\psi')$ is defined to be an isomorphism $\tau: K\to K'$ such that the map $m_\tau: G(K/k)\to G(K'/k)$ sending $g$ to $\tau\circ g\circ\tau^{-1}$ satisfies $\psi'\circ m_\tau=\psi$.  
    We sometimes leave the $\psi$ implicit, but this is always what we mean by a $G$-extension.  
    The $G$-extensions of $\mathbb{Q}$ are also called \emph{$G$-fields}.
\end{definition}
\begin{definition}\label{def:set of fields}
    A transitive permutation group is defined to be a subgroup $G$ of $S_n$ such that $G$ acts on $\{1,2,\dots,n\}$ transitively.
    Let $G$ be a transitive permutation group.
    Fix a number field $k$.
    Define ${\mathcal{E}}(G;k)$ to be the isomorphic classes of pairs $(K,\psi)$ such that the Galois closure $(\hat{K},\psi)$ is a $G$-field and $K=\hat{K}^{G_1}$, where $G_1=\operatorname{Stab}_G(1)$ is the stabilizer of $1$.
    If the base field $k=\mathbb{Q}$, then we just omit it and write ${\mathcal{E}}(G):={\mathcal{E}}(G,\mathbb{Q})$.
\end{definition}
There are alternative ways to define the set of fields.
See Wood~\cite[p.292]{directions2016} for their differences and connections.
Note that if $G$ is abelian, and we view $G$ as a transitive permutation group by its operation on itself, then $\mathcal{E}(G;k)$ simply means the set of abelian $G$-extensions over $k$.
Let us give the notation of counting number fields.
\begin{definition}\label{def:counting number fields}
    Let $S$ be a countable set with a counting function $C:S\to\mathbb{R}_{\geq0}$ such that for each $X>0$ the set $\{a\in S\mid C(a)<X\}$ is finite.
    Define
	\begin{equation*}
		N_{C}(S,X):=\#\{a\in S\mid C(a)<X\}.
	\end{equation*}
\end{definition}
For now, we have seen that the (absolute) discriminant could work as a counting function.
However, in some cases, ordering fields by discriminant will contradict what is predicted by the heuristics.
See Cohen and Martinet~\cite{Cohen1994HeuristicsOC}, Bartel and Lenstra~\cite{bartel2020class} for example.
See also Wood~\cite{wood2010probabilities} for some discussions on different orderings from a field-counting point of view.
The choice of the counting function may affect the result of field-counting in a nontrivial way.
But we are not going to discuss it in detail here.
For abelian fields, we simply use the product of ramified primes as our counting function.
For more general cases, see Definition~\ref{def: product of ramified primes}.

When $p\mid\lvert G\rvert$, the work of the author~\cite{wang2022distributionOT} shows the following result for abelian fields.
\begin{theorem}\label{thm: S1 abelian case relative class groups}
    Let $G$ be a finite abelian group with a subgroup $H$.
    Define $\mathcal{E}:=\mathcal{E}(G)$.
    For each $K\in\mathcal{E}$, define $C:=\sqrt{(d_K)}$ where $\sqrt{}$ means the radical, i.e., the product of ramified primes.
    Let $q$ be a prime number such that $p\mid\lvert G/H\rvert$, and $\Omega:=\{g\in G\mid r_g\equiv0\bmod{p}\}$, where $r_g$ is the order of $g$.
    We have that
	\begin{equation*}
		\mathbb{P}_{C,\mathcal{E}}(\operatorname{rk}_p\operatorname{Cl}(K/K^H)\leq r)=0
		\quad\text{and}\quad
		\mathbb{E}_{C,\mathcal{E}}(\lvert\operatorname{Hom}(\operatorname{Cl}(K/K^H),C_p)\rvert)=+\infty,
	\end{equation*}
	where the notation $\operatorname{Cl}(K/K^H)$ means the relative class group.
\end{theorem}
And we can see that the zero probability of $\operatorname{rk}_p\operatorname{Cl}(K/K^H)\leq r$ and the infinite $C_p$-moment are qualitatively different from what is predicted from the Cohen-Lenstra-Martinet Heuristics for the case when $p\nmid\lvert G\rvert$, 
which generalizes the quadratic case.
On the other hand, for non-Galois fields, the method of the author~\cite{wang2022distributionOT} could not be applied directly.
Let us use an example to explain this.
\begin{definition}\label{def: e(p)}
	Let $L/K$ be an extension of the number fields.
	For each $\mathfrak{p}\in\mathcal{P}_K$, define
	\begin{equation*}
		e(\mathfrak{p},L/K):=\gcd(e(\mathfrak{P}/\mathfrak{p}))_{\mathfrak{P}\mid\mathfrak{p}},
	\end{equation*}
	where $\mathfrak{P}$ runs over primes of $L$ above $\mathfrak{p}$.
\end{definition}
The result of Roquette and Zassenhaus~\cite{RZ1969ClassRank} says the following.
\begin{theorem}\cite[Theorem 1]{RZ1969ClassRank}\label{thm:S1 invariant part}
    Let \(K\) be a number field of degree \(n\) over \(\mathbb{Q}\), and \(p\) be a finite rational prime.
    We have
	\[\operatorname{rk}_p\operatorname{Cl}_K\geq\#\{q\in\mathcal{P}\mid e(q,K/\mathbb{Q})\equiv0\bmod{p}\}-2(n-1).\]
\end{theorem}
This is the algebraic preliminary for results like Theorem~\ref{thm: S1 abelian case relative class groups}.
Also, this shows the connection between specific ramified primes and the $q$-rank of class groups.
However, let $G$ be the image of $A_4$ in $S_6$, and consider the set of fields $\mathcal{E}:=\mathcal{E}(G)$.
It means that $L\in\mathcal{E}$ if and only if $[L:\mathbb{Q}]=6$ and the Galois closure $\hat{L}$ is an $A_4$-field (or just sextic $A_4$-field for short).
One can check the ramification table of all tamely ramified primes in $L/\mathbb{Q}$ and obtain the following result:
\begin{proposition}
	For each sextic $A_4$-field $L$, we have
	\begin{equation*}
		\#\{p\nmid 6: e(p,L/\mathbb{Q})\equiv0\bmod{2}\}=0.
	\end{equation*}
\end{proposition}
If we apply Theorem~\ref{thm:S1 invariant part} to a sextic $A_4$-field $L$, then it always says that 
\begin{equation*}
	\operatorname{rk}_2\operatorname{Cl}_L\geq-10,
\end{equation*}
which is trivial.
On the other hand, the structure of $G$ (or $A_4$) tells us that it is likely that there are some nontrivial subgroups of $\operatorname{Cl}_L[2^\infty]$ (the Sylow $2$-subgroup) associated to the ramified primes.
So we need some generalization of Theorem~\ref{thm:S1 invariant part} in the sense that the base field is not necessarily $\mathbb{Q}$.
Let us present it here as follows.
\begin{theorem}\label{thm: S1 estimate of class group}
    Fix a rational prime $p$.
    Let $L/K$ be a finite extension of number fields with Galois closure $N/K$.
    Denote $[N:K]$ by $n$.
    We have
	\[\operatorname{rk}_p\operatorname{Cl}_L\geq\#\{\mathfrak{p}\in\mathcal{P}_K\mid e(\mathfrak{p},L/K)\equiv0\bmod{p}\}-n^2.\]
\end{theorem}
Clearly if the base field $K=\mathbb{Q}$, then we go back to the case of Theorem~\ref{thm:S1 invariant part} with a coarser lower bound on the $p$-rank in general, thus it is indeed a version of generalization.
See Section~\ref{section: estimate of class group} for more details.
As for the distribution of class groups, we have the following.
\begin{theorem}\label{thm: S1 dist of rel class group abelian}
	Let $K$ be a fixed number field.
	For each finite Galois extension $L/K$, define $C(L/K):=\mathfrak{N}(\sqrt{\mathfrak{d}_{L/K}})$, 
    where $\mathfrak{d}$ is the relative discriminant, $\sqrt{}$ is the radical, and $\mathfrak{N}$ is the absolute norm.
	Let $G$ be a finite abelian group, and set $\mathcal{E}:=\mathcal{E}(G,K)$.
	For each $p\mid\lvert G\rvert$, and for each non-negative integer $r$, we have 
	\begin{equation*}
		\mathbb{P}_{C,\mathcal{E}}(\operatorname{rk}_p\operatorname{Cl}(L/K)\leq r)=0.
	\end{equation*}
	In addition, we have
	\begin{equation*}
		\mathbb{E}_{C,\mathcal{E}}(\lvert\operatorname{Hom}(\operatorname{Cl}(L/K),C_p)\rvert)=+\infty.
	\end{equation*}
\end{theorem}
See Section~\ref{section: abelian} for the proof.
For our ``prototype'', the sextic $A_4$-fields, we have the following.
\begin{theorem}\label{thm: S1 dist of class group sextic case}
	Let $\mathcal{E}$ be the set of sextic $A_4$-fields, and $d$ the absolute discriminant.
	For each non-negative integer $r$, we have that
	\begin{equation*}
		\mathbb{P}_{d,\mathcal{E}}(\operatorname{rk}_2\operatorname{Cl}_L\leq r)=0.
	\end{equation*}
	Moreover, we have
	\begin{equation*}
		\mathbb{E}_{d,\mathcal{E}}(\lvert\operatorname{Hom}(\operatorname{Cl}_L,C_2)\rvert)=+\infty.
	\end{equation*}
\end{theorem}
Similarly, for cubic fields with $S_3$-closure, we have the following result.
\begin{theorem}\label{thm: S1 dist of class group cubic case}
    Let $\mathcal{E}$ be the set of non-Galois cubic fields.
    For each $K_3\in\mathcal{E}$, let $K_6$ be its Galois closure with a unique quadratic subfield $K_2$.
    Fix an integer $n\geq2$.
    Define the counting function $C$ by
	\begin{equation*}
		C(K_3):=d_{K_3}d_{K_2}^n,
	\end{equation*}
	where $d$ is the absolute discriminant.
	For each non-negative integer $r$, we have
	\begin{equation*}
		\mathbb{P}_{C,\mathcal{E}}(\operatorname{rk}_3\operatorname{Cl}_{K_3}\leq r)=0.
	\end{equation*}
	Moreover, we have
	\begin{equation*}
		\mathbb{E}_{C,\mathcal{E}}(\lvert\operatorname{Hom}(\operatorname{Cl}_{K_3},C_3)\rvert)=+\infty.
	\end{equation*}
\end{theorem}
Note that for absolute discriminant, we have
\begin{equation*}
	d_{K_6}=d_{K_2}d_{K_3}^2=d_{K_2}^3\operatorname{Nm}_{K_2/\mathbb{Q}}(\mathfrak{d}_{K_6/K_2}),
\end{equation*}
where $\mathfrak{d}$ means the relative discriminant.
So the counting function $C$, compared to the discriminant, assigns different ``weights'' to $\mathfrak{d}_{K_6/K_2}$ and $\mathfrak{d}_{K_2/\mathbb{Q}}=d_{K_2}$,
and we could call such counting functions ``generalized discriminant''.
See Section~\ref{sec: general case} for more discussions.

\subsection{Outline}
In Section~\ref{section: estimate of class group}, we first show some nontrivial algebraic structure of the relative class group, including the proof of Theorem~\ref{thm: S1 estimate of class group}.
This is also the algebraic preliminary to prove Theorem~\ref{thm: S1 dist of rel class group abelian} and~\ref{thm: S1 dist of class group sextic case}.
In Section~\ref{sec: specifications} and~\ref{sec: generating series}, we prepare some technical notations and tools for the proof of the distribution of class groups.
They could be used to translate the arithmetic properties of the extensions of number fields into the information of a generating series together with its analytic properties.
Section~\ref{sec: specifications} focuses on defining local specifications and counting functions in an appropriate context,
and Section~\ref{sec: generating series} focuses on studying the analytic properties of certain series.
Then, we prove statistical results for abelian extensions, resp. for some special non-abelian fields, in Section~\ref{section: abelian}, resp. in Section~\ref{sec: general case}.
The key idea of the proof is the following.
We first construct a suitable generating series to estimate the result of counting fields with a special specification ordered by the counting function.
Then study its analytic properties, especially its analytic continuation and the pole behaviour.
By Tauberian Theorem~\cite[Appendix II Theorem I]{narkiewicz2014elementary}, we obtain the asymptotic behaviour.
Finally, algebraic results like Theorem~\ref{thm: S1 estimate of class group} together with the asymptotic behaviour of the field-counting will imply the desired ``zero-probability'' and ``infinite moment'' as stated in Theorem~\ref{thm: S1 dist of rel class group abelian}, \ref{thm: S1 dist of class group sextic case}, and~\ref{thm: S1 dist of class group cubic case}.

\section{Basic notations}\label{section:notation}
In this section we introduce some of the notations that will be used in the paper.
We use some standard notations coming from analytic number theory.
For example, write a complex number as $s=\sigma+it$.
Denote the Euler's phi function by $\phi(n)$.
Let $\omega(n)$ count the number of distinct prime divisors of $n$ and so on.

We also follow the notations of inequalities with unspecified constants from Iwaniec and Kowalski~\cite[Introduction, p.7]{iwaniec2004analytic}.
Let us just write down the ones that are important for us.
Let $X$ be some space (usually some region of $\mathbb{C}$ in our paper), and let $f,g$ be two complex functions defined on $X$.
Then $f(x)\ll g(x)$ for $x\in X$ means that $\lvert f(x)\rvert\leq Cg(x)$ for some constant $C\geq0$.
Any value of $C$ for which this holds is called an implied constant.
We use $f(x)\asymp g(x)$ for $x\in X$ if $f(x)\ll g(x)$ and $g(x)\ll f(x)$ both hold with possibly different implied constants.
We say that $f=o(g)$ as $x\to x_0$ if for any $\epsilon>0$ there exists some (unspecified) neighbourhood $U_\epsilon$ of $x_0$ such that $\lvert f(x)\rvert\leq\epsilon g(x)$ for $x\in U_\epsilon$.
Finally, $f\sim g$ as $x\to x_0$ if we can write $f=g+o(g)$.

\section{Estimate of class group}\label{section: estimate of class group}
In this section we give an estimate of class group following the idea of Roquette and Zassenhaus~\cite{RZ1969ClassRank}.
The most important goal of this section is to prove Theorem~\ref{thm: S1 estimate of class group} stated in Section~\ref{section: intro}.
Here we generalize it a little, so that it could also give an estimate for the higher torsion part of the class group.
Recall that $\mathcal{P}_K$ means the set of finite primes of $K$, and recall the notation $e(\mathfrak{p},L/K)$ from Definition~\ref{def: e(p)}.
\begin{theorem}\label{thm: estimate of class group}
    Fix a number field $K$, and let $L/K$ be a finite extension with its Galois closure $N/K$.
    There exists some constant $c$, depending on $[N:\mathbb{Q}]$, such that for each rational prime $p$, and for each non-negative integer $l$, we have 
	\begin{equation*}
		\operatorname{rk}_p p^{l}\operatorname{Cl}_L\geq\#\{\mathfrak{p}\in\mathcal{P}_K\mid e(\mathfrak{p},L/K)\equiv0\bmod{p^{l+1}}\}-c.	
	\end{equation*}
    One can take $c=[N:\mathbb{Q}]^2$.
\end{theorem}
For our convenience, let us give an equivalent definition of $p$-rank for a finite abelian group, so that it also works for finitely generated abelian groups.
\begin{lemma}
	Let $A$ be a finite abelian group, then
	\begin{equation*}
		\operatorname{rk}_pA=\dim_{\mathbb{F}_p}A/pA.
	\end{equation*}
\end{lemma}
\begin{proof}
	Recall that for a finite abelian group $A$, its $p$-rank $\operatorname{rk}_pA$ is defined as the largest number $r$ such that there exists some injective homomorphism $C_p^r\to A$.
	
	Since the $p$-part $A_{(p)}=\mathbb{Z}_{(p)}\otimes A$ of $A$ is a direct summand, we see that if $C_p^{l}\to A$ is any group homomorphism, then its image is included in $A_{(p)}$.
	Moreover it is clear that
	\begin{equation*}
		\dim_{\mathbb{F}_p}A/pA=\dim_{\mathbb{F}_p}A_{(p)}/pA_{(p)}.
	\end{equation*}
	So we may assume without loss of generality that $A$ is a finite abelian $p$-group.
	According to the structure theorem of finite abelian group, we see that there exists some non-negative integers $l$ and $n_1\leq n_2\leq\cdots\leq n_l$ such that
	\begin{equation*}
		A\cong\prod_{i=1}^l\mathbb{Z}/p^{n_i}\mathbb{Z}.
	\end{equation*}
	Let $x_i$ be a generator of $\mathbb{Z}/p^{n_i}\mathbb{Z}$.
	Then $\{x_1,\dots,x_l\}$ gives a basis of $A/pA$.
	In other words,
	\begin{equation*}
		\dim_{\mathbb{F}_p}A/pA=l.
	\end{equation*}
	And the set $\{p^{n_i-1}x_i\}_{i=1}^l$ generates an elementary $p$-group in $A$ that is isomorphic to $C_p^l$.
	This shows that
	\begin{equation*}
		\operatorname{rk}_pA\geq l.
	\end{equation*}
	Conversely, let $F\cong\mathbb{Z}_{(p)}^l$ be the free $\mathbb{Z}_{(p)}$-module with the set of generators $\{e_1,\dots,e_l\}$.
	There is a surjective map $f:F\to A$ induced by $e_i\mapsto x_i$.
	If $B\cong C_p^{l'}$ is any elementary $p$-group included in $A$, then $f^{-1}(B)$ is a submodule of $F$.
	Since $\mathbb{Z}_{(p)}$ is a principal ideal domain, we know that $f^{-1}(B)$ is also free and its rank as a free module (or the size of a set of basis) is at most $l$.
	Therefore $B$ is generated by at most $l$ elements, i.e., $l'\leq l$.
	This shows that $\operatorname{rk}_pA\leq l$.
	And we are done.
\end{proof}
From now on, given any abelian group $A$, we define
\begin{equation*}
	\operatorname{rk}_p A=\dim_{\mathbb{F}_p}A/pA.
\end{equation*}
The most important property of this definition is the following.
\begin{proposition}\label{prop: inequalities of p-rank}
    Let $A$, $B$, and $C$ be abelian groups.
	\begin{enumerate}
		\item If $A\subseteq B$ then
		\begin{equation*}
			\operatorname{rk}_pA\leq\operatorname{rk}_pB.
		\end{equation*}
		\item If $B\to C$ is a surjective homomorphism, then
		\begin{equation*}
			\operatorname{rk}_pB\geq\operatorname{rk}_pC.
		\end{equation*}
		\item Given a short exact sequence of abelian groups
		\begin{equation*}
			0\to A\to B\to C\to0,
		\end{equation*}
		we have that
		\begin{equation*}
			\operatorname{rk}_p B\leq\operatorname{rk}_p A+\operatorname{rk}_p C.
		\end{equation*}
	\end{enumerate}	
\end{proposition}
The inequalities hold in the following sense.
Take (3) for example.
If $\operatorname{rk}_pA$ and $\operatorname{rk}_pC$ are both finite, then $\operatorname{rk}_pB$ must also be finite and less than $\operatorname{rk}_pA+\operatorname{rk}_pC$.
Else if $\operatorname{rk}_pB$ is infinite, then the rank $\operatorname{rk}_pA$ or $\operatorname{rk}_pC$ must also be infinite.
\begin{proof}
	(1):
	For an abelian group $A$, we have
	\begin{equation*}
		A=A_t\oplus A_f,
	\end{equation*}
	where $A_t$ is the torsion part and $A_f$ is the torsion-free part.
	Clearly if $A\subseteq B$, then $A_t\subseteq B_t$ and $A_f\subseteq B_f$.
	For the torsion-free part, this reduces to the property that $\mathbb{Z}$ is a principal ideal domain and the number of generators of $A_f$ cannot exceed the one of $B_f$.
	For the torsion part, we can use the equivalent definition that $\operatorname{rk}_pA_t$ is the largest number $r$ such that there exists an injective homomorphism $C_p^r\to A_t$.
	Composing with $A_t\to B_t$, we obtain an injective homomorphism $C_p^r\to B_t$, which implies that $\operatorname{rk}_pB_t\geq r$.
	And we are done.
	
	(2) and (3):
	Assume that we have a short exact sequence
	\begin{equation*}
		0\to A\to B\to C\to0.
	\end{equation*}	
	Tensoring with $\mathbb{Z}/p\mathbb{Z}$, we obtain the following exact sequence
	\begin{equation*}
		A/pA\to B/pB\to C/pC\to0.
	\end{equation*}
	Now that the abelian groups in the sequence are also $\mathbb{F}_p$-vector spaces, we see that the map $B/pB\to C/pC$ splits, and
	\begin{equation*}
		B/pB=C/pC\oplus\operatorname{im}(A/pA).
	\end{equation*}
	As vector spaces, we clearly have
	\begin{equation*}
		\dim_{\mathbb{F}_p}B/pB\geq\dim_{\mathbb{F}_p}C/pC\quad\text{and}\quad\dim_{\mathbb{F}_p}\operatorname{im}(A/pA)\leq\dim_{\mathbb{F}_p}A/pA.
	\end{equation*}
	And the result just follows.
\end{proof}
Let us make some notations.
Fix a number field $K$.
Denote its group of fractional ideals by $I_K$ and its group of principal ideals by $P_K$.
Let $N/K$ be a finite Galois extension with Galois group $G=G(N/K)$.
For any subextension $L/K$ of $N/K$, we can embed $I_L$ into $I_N$ via $L\to N$.
The subgroup $I_L$ of $I_N$ in general is not a $G$-submodule of $I_N$.
Still we can consider the subgroup of $I_L$ fixed by $G$, i.e., define
\begin{equation*}
	I_L^G:=I_L\cap I_N^G=\{I\in I_L\mid g\cdot I=I\text{ for all }g\in G\}
	\quad\text{and}\quad
	P_L^G:=P_L\cap I_N^G,
\end{equation*}
Let us call it $G$-\emph{invariant part} of $I_L$, resp. $P_L$.
Note that when $L=N$, or more generally if $L/K$ is Galois, then $I_L$, resp. $P_L$, is a $G$-module, and the so-called $G$-invariant part coincides with the usual notion defined by the Galois action.
Since we have a chain of extensions, let us denote the primes of $K$ by $v$, the primes of $L$ by $w$, and the primes of $N$ by $x$.
For each $v\in\mathcal{P}_K$, and for each $w\in\mathcal{P}_L$ such that $w\mid v$, define
\begin{equation*}
    n_{w/v}:=\frac{e(w/v)}{e(v,L/K)}.
\end{equation*}
Then define
\begin{equation}\label{eqn: a(p)}
    \mathfrak{a}_{v}:=\prod_{w\mid v}w^{n_{w/v}}\in I_L.
\end{equation}
Recall from Definition~\ref{def: e(p)} that $e(v,L/K)\mid e(w/v)$ for all $w\mid v$.
In other words, $\mathfrak{a}_{v}\in I_L$ is the integral ideal such that $\mathfrak{a}_{v}^{e(v,L/K)}=v$ in $I_L$.
\begin{lemma}\label{lemma: I(L)^G}
	The group $I_L^G$ is a free abelian group generated by 
	\begin{equation*}
		\{\mathfrak{a}_v\mid v\in\mathcal{P}_K\}.
	\end{equation*}
\end{lemma}
\begin{proof}
	This is a generalization of Roquette and Zassenhaus~\cite[(6)]{RZ1969ClassRank} to the case when the base field is changed from $\mathbb{Q}$ to $K$.
	First of all, for any (finite) prime $v$ of $K$, the ideal $\mathfrak{a}_{v}$ (see~(\ref{eqn: a(p)})) viewed as an element of $I_N$ could be decomposed by prime ideals of $N$ as follows:
	\begin{equation*}
		\mathfrak{a}_{v}=
		\prod_{
			\substack{
				w\in\mathcal{P}_L\\
				w\mid v
			}
		}
		\prod_{
			\substack{
				x\in\mathcal{P}_N\\
				x\mid w
			}
		}x^{e(x/w)n_{w/v}}.
	\end{equation*}
	Note that $N/K$ is Galois.
        For each $x\mid w$, we have 
	\begin{equation*}
		e(x/w)e(w/v)=e(x/v)=e(v,N/K).
	\end{equation*}
	Therefore $e(x/w)n_{w/v}=e(x/v)e(v,L/K)^{-1}$, and we have
	\begin{equation*}
		\mathfrak{a}_v=\prod_{x\mid v}x^{e(v,N/K)e(v,L/K)^{-1}}.
	\end{equation*} 
	It is fixed by the action of $G$, for all the exponents are the same.
	This shows that $\mathfrak{a}_v\in I_L^G$ for each $v\in\mathcal{P}_K$.
	
	Note that $\mathfrak{a}_v\mid v$.
	If $v\neq v'$, then clearly $\mathfrak{a}_v$ and $\mathfrak{a}_{v'}$ are coprime, hence linearly independent in $I_L$.
	This shows that $\{\mathfrak{a}_v\}_{v\in\mathcal{P}_K}$ is a set of linearly independent ideals in $I_L$.
	
	Finally, if $\mathfrak{b}\in I_L^G$, and $\mathfrak{b}=w_1^{l_1}\dots w_n^{l_n}$ is the decomposition of $\mathfrak{b}$ in $L$, then all the conjugates of $w_1$ must show up in the decomposition of $\mathfrak{b}$.
	So we can assume without loss of generality that $w_1,\dots,w_n$ all lie above a fixed prime $v$ of $K$.
	Again, let us decompose $\mathfrak{b}$ in $N$:
	\begin{equation*}
		\mathfrak{b}=\prod_{i=1}^{n}\prod_{x\mid w_i}x^{e(x/w_i)l_i}.
	\end{equation*}
	Since all the primes $x$ of $N$ above $v$ are conjugate to each other, we see that if $x_1\mid w_i$ and $x_2\mid w_j$, then
	\begin{equation*}
		\begin{aligned}
			e(x_1/w_i)l_i=&e(x_2/w_j)l_j
			\\
			\Rightarrow\frac{e(v,N/K)}{e(w_i/v)}l_i=&\frac{e(v,N/K)}{e(w_j/v)}l_j
			\\
			\Rightarrow l_i=&\frac{e(w_i/v)}{e(w_j/v)}l_j
		\end{aligned}		
	\end{equation*}
	Since $l_1,\dots,l_n$ are integers, we see that for each $1\leq j\leq n$, we have that
	\begin{equation*}
		n_{w_j/v}=\frac{e(w_j/v)}{\gcd(e(w_1/v),\dots,e(w_n/v))}\mid l_j.
	\end{equation*}
	And this implies that $\mathfrak{a}_v\mid\mathfrak{b}$, i.e., $\{\mathfrak{a}_v\}_{v\in\mathcal{P}_K}$ generates $I_L^G$.
\end{proof}
An immediate corollary is the following.
\begin{corollary}\label{cor: I(L)^G/I(K)}
	The quotient group $I_L^G/I_K$ is a finite abelian group, and we have
	\begin{equation*}
		I_L^G/I_K\cong\prod_{v\in\mathcal{P}_K}\mathbb{Z}/e(v,L/K)\mathbb{Z}
	\end{equation*}
	where $v$ runs over all finite primes of $K$.
\end{corollary}
\begin{proof}
	This is a generalization of Roquette and Zassenhaus~\cite[(8)]{RZ1969ClassRank} to the case when the base field is a number field $K$ instead of $\mathbb{Q}$. 
	Note that $\mathfrak{a}_v^{e(\mathfrak{p},L/K)}=v$.
	So $I_K$ is a free abelian subgroup generated by $\{\mathfrak{a}_v^{e(\mathfrak{p},L/K)}\}_{v\in\mathcal{P}_K}$, and 
	\begin{equation*}
		I_L^G/I_K=\langle\mathfrak{a}_v\mid\mathfrak{a}_v^{e(\mathfrak{p},L/K)}=1\rangle\cong\prod_{v\in\mathcal{P}_K}\mathbb{Z}/e(\mathfrak{p},L/K)\mathbb{Z}.
	\end{equation*}
\end{proof}
Let us give a notation for the subgroup of $\operatorname{Cl}_L$ generated by ideals in $I_L^G$.
\begin{definition}\label{def: invariant part of fractional ideals}
    Let $L/K$ be a subextension of $N/K$.
    Define
	\begin{equation*}
		\operatorname{D}_{L}:=I_{L}^G/P_{L}^G.
	\end{equation*}
\end{definition}
We'll give an estimate for $\operatorname{rk}_p\operatorname{D}_L$, which serves as an estimate of $\operatorname{rk}_p\operatorname{Cl}_L$.
\begin{proposition}\label{prop: D(L) subgroup of Cl(L)}
    Let $L/K$ be a subextension of $N/K$.
    The group \(\operatorname{D}_{L}\) is a subgroup of \(\operatorname{Cl}_{L}\).
    Moreover, \(\operatorname{Cl}_N^G/\operatorname{D}_N\) is annihilated by \(\lvert G\rvert\).
\end{proposition}
\begin{proof}
	For $L$, we see that $I_L^G$ is a subgroup of $I_L$.
	And $P_L^G=I_L^G\cap P_L$ which is the definition.
	So $I_L^G/P_L^G=I_L^GP_L/P_L$ is naturally a subgroup of $\operatorname{Cl}_L$.
	For $N$, we first have the short exact sequence
	\begin{equation*}
		1\to P_N\to I_N\to\operatorname{Cl}_N\to1.
	\end{equation*}
	By taking the $G$-invariant part, we obtain the long exact sequence
	\begin{equation*}
		1\to P_N^G\to I_N^G\to\operatorname{Cl}_N^G\to H^1(G,P_N)\to\cdots.
	\end{equation*}
	Since $H^1(G,P_N)$ is annihilated by \(\lvert G\rvert\) (see Neukirch and Schmidt and Wingberg~\cite[1.6.1]{neukirch2013cohomology}), we see that the cokernel of $I_N^G\to\operatorname{Cl}_N^G$ also satisfies \(\lvert G\rvert\cdot\operatorname{Cl}_N^G/\operatorname{D}_N=0\).
\end{proof}
The group $P_K$ stays inside $P_N^G$ as a subgroup, and let us give the quotient subgroup $P_N^G/P_K$ an estimate by the following two results.
\begin{proposition}\label{prop: P(L)^G/P(K)}
    Fix a rational prime $p$.
    We have
	\begin{equation*}
		P_N^G/P_K\cong H^1(G,\mathscr{O}_N^*),
	\end{equation*}
	where $\mathscr{O}_N$ is the ring of integers of $N$.
\end{proposition}
\begin{proof}
	Let us consider the following short exact sequence
	\begin{equation*}
		1\to\mathscr{O}_N^*\to N^*\to P_N\to1,
	\end{equation*}
	which is induced by $x\mapsto (x)$ for all $x\in N^*$.
	Taking the $G$-invariant part, we have the following long exact sequence:
	\begin{equation*}
		1\to H^0(G,\mathscr{O}_N^*)\to H^0(G,N^*)\to H^0(G,P_N)\to H^1(G,\mathscr{O}_N^*)\to H^1(G,N^*)\to\cdots.
	\end{equation*}
	Note that by additive Hilbert's 90, we have $N^G=K$, as additive modules.
	But the actions of \(G\) on $N$ and $N^*$ are the same.
	In other words, as $G$-sets, $N^*$ is a $G$-subset of $N$.
	So, we know that $H^0(G,N^*)=K^*$.
	
	Claim: $H^0(G,\mathscr{O}_N^*)=\mathscr{O}_K^*$.
	It is clear that $\mathscr{O}_K^*\subseteq(\mathscr{O}_N^*)^G$.
	On the other hand, if $x\in\mathscr{O}_N^*$ is fixed by $G$, then as an element of $L^G$, it must be an element of $K$ by the additive Hilbert 90.
	So $x\in K^*$ is an element such that $\operatorname{Nm}_{L/\mathbb{Q}}(x)=\pm1$, which implies that $\operatorname{Nm}_{K/\mathbb{Q}}(x)=\pm1$.
	This shows that $x\in\mathscr{O}_K^*$.
	
	According to Hilbert's 90, we know that $H^1(G,N^*)=1$.
	So $H^0(G,P_N)=P_N^G\to H^1(G,\mathscr{O}_N^*)$ is surjective with kernel equal the image of $K^*$.
	The map $K^*\to P_N^G$ is induced by $N^*\to P_N$. 
	And $\operatorname{im}(K^*)\cong K^*/\mathscr{O}_K^*\cong P_K$.
	So $K^*\to P_N^G$ gives the usual embedding of $P_K$ in $P_N^G$.
	And we obtain the desired result.
\end{proof}
\begin{proposition}\label{prop: H^1(G,O_L^*)}
	The cohomology group $H^1(G,\mathscr{O}_N^*)$ is a finite abelian group that is annihilated by $\lvert G\rvert$.
	Moreover, we have the following estimate
	\begin{equation*}
		\operatorname{rk}_{p}H^1(G,\mathscr{O}_N^*)\leq\operatorname{rk}_{p}\operatorname{Map}(G,\mathscr{O}_N^*),
	\end{equation*}
	for each rational prime $p$.
	Here $\operatorname{Map}(G,\mathscr{O}_N^*)$ is the abelian group consisting of morphisms between sets.
	In particular, $\operatorname{Map}(G,\mathscr{O}_N^*)$ is a finite abelian group that could be generated by $(r_N+1)[N:K]$ elements, where $r_N$ is the rank of $\mathscr{O}_N^*$.
\end{proposition}
\begin{proof}
	According to the interpretation of cohomolology group $H^1(G,\mathscr{O}_N^*)$, we know that the $1$-cocycles could be represented by the group of crossed homomorphisms from $G$ to $\mathscr{O}_N^*$, and we denote it by $Z^1(G,\mathscr{O}_N^*)$.
	On the other hand, as a group $Z^1(G,\mathscr{O}_N^*)$ is a subgroup of $\operatorname{Map}(G,\mathscr{O}_N^*)$.
	Therefore we have the following exact sequences
	\begin{equation*}
		\begin{aligned}
			&0\to Z^1(G,\mathscr{O}_N^*)\to\operatorname{Map}(G,\mathscr{O}_N^*)\\
			&Z^1(G,\mathscr{O}_N^*)\to H^1(G,\mathscr{O}_N^*)\to0.
		\end{aligned}
	\end{equation*}
	By Proposition~\ref{prop: inequalities of p-rank}, we see that 
	\begin{equation*}
		\operatorname{rk}_pH^1(G,\mathscr{O}_N^*)\leq\operatorname{rk}_pZ^1(G,\mathscr{O}_N^*)\leq\operatorname{rk}_p\operatorname{Map}(G,\mathscr{O}_N^*).
	\end{equation*}
	Note that every object here is a finitely generated abelian group.
	Then, let us construct a set that generates $\operatorname{Map}(G,\mathscr{O}_N^*)$.
	If $\{\alpha_0:=\zeta_N,\alpha_1,\dots,\alpha_{r_N}\}$ is a basis of $\mathscr{O}_N^*$, then for each $g\in G$ and for each $i=0,\dots,r_N$, let $\varphi_{g,i}:G\to\mathscr{O}_N^*$ be the map such that $\varphi_{g,i}(g)=\alpha_i$ and $\varphi_{g,i}(h)=1$ for all $h\neq g$.
	Clearly for each $f\in\operatorname{Map}(G,\mathscr{O}_N^*)$, we see that $f$ is a linear combination of $\{\varphi_{g,i}\}_{g,i}$.
	The size of $\{\varphi_{g,i}\}_{g,i}$ is exactly $(r_N+1)\cdot\lvert G\rvert$.
	And we are done for the proof.
\end{proof}
\subsection{Proof of Theorem~\ref{thm: estimate of class group}}
Let us prove the theorem in this section.
\begin{proof}[Proof of Theorem~\ref{thm: estimate of class group}]
	It suffices to give an estimate for the $p$-rank of $\operatorname{D}_L$, for it is a subgroup of $\operatorname{Cl}_L$ by Proposition~\ref{prop: D(L) subgroup of Cl(L)}.
	Using the relation $\operatorname{Cl}_K=I_K/P_K$, we have a short exact sequence
	\begin{equation*}
		1\to\operatorname{Cl}_K\to I_L^G/P_K\to I_L^G/I_K\to1.
	\end{equation*}
	Clearly $p^{l} I_L^G/P_K\to p^{l}I_L^G/I_K$ is surjective, and the kernel is included in $\operatorname{Cl}_K$.
	Using Proposition~\ref{prop: inequalities of p-rank}, we see that
	\begin{equation}\label{eqn: estimate of class group 1}
		\operatorname{rk}_p p^{l}I_L^G/I_K\leq\operatorname{rk}_p p^{l}I_L^G/P_K\leq\operatorname{rk}_p\operatorname{Cl}_K+\operatorname{rk}_p p^{l}I_L^G/I_K.
	\end{equation}
	By Corollary~\ref{cor: I(L)^G/I(K)}, we see that
	\begin{equation}\label{eqn: estimate of class group 2}
		\begin{aligned}
			\operatorname{rk}_p p^{l}I_L^G/I_K=&\operatorname{rk}_p p^{l} \prod_{v\in\mathcal{P}_K}\mathbb{Z}/e(v,L/K)\mathbb{Z}\\
			=&\#\{v\in\mathcal{P}_K\mid e(v,L/K)\equiv0\bmod{p^{l+1}}\}.
		\end{aligned}		
	\end{equation}
	On the other hand, we have the following commutative diagram
	\begin{equation*}
		\begin{tikzcd}
			&1\arrow[r]&P_L^G/P_K\arrow[r]\arrow[d]&I_L^G/P_K\arrow[r]\arrow[d]&\operatorname{D}_L\arrow[r]\arrow[d]&1\\
			&1\arrow[r]&P_N^G/P_K\arrow[r]&I_N^G/P_K\arrow[r]&\operatorname{D}_N\arrow[r]&1
		\end{tikzcd}
	\end{equation*}
	Note that by Proposition~\ref{prop: H^1(G,O_L^*)}, there exists some constant $c>0$ that only depends on $[N:K]$ such that
	\begin{equation*}
		\operatorname{rk}_p P_N^G/P_K\leq c.
	\end{equation*}
        Here, $c$ could be taken to be $[N:K]^2$.
	Since $P_L^G/P_K$ is a subgroup of $P_N^G/P_K$, by Proposition~\ref{prop: inequalities of p-rank}, we have that
	\begin{equation}\label{eqn: estimate of class group 3}
		\operatorname{rk}_p P_L^G/P_K\leq\operatorname{rk}_p P_N^G/P_K\leq c.
	\end{equation}
	Note also that $p^{l}I_L^G/P_K\to p^{l}\operatorname{D}_L$ is surjective, and the kernel is included in $P_L^G/P_K$.
	By Proposition~\ref{prop: inequalities of p-rank} again, we have that
	\begin{equation*}
		\operatorname{rk}_p p^{l}I_L^G/P_K-\operatorname{rk}_p P_L^G/P_K
		\leq\operatorname{rk}_p p^{l}\operatorname{D}_L
		\leq\operatorname{rk}_p p^{l}I_L^G/P_K	
	\end{equation*}
	Together with (\ref{eqn: estimate of class group 1}), (\ref{eqn: estimate of class group 2}) and (\ref{eqn: estimate of class group 3}), we see that
	\begin{equation*}
		\begin{aligned}
			&\#\{v\in\mathcal{P}_K\mid e(v,L/K)\equiv0\bmod{p^{l+1}}\}-c
			\\
			\leq&\operatorname{rk}_p p^{l}\operatorname{D}_L
			\leq\#\{v\in\mathcal{P}_K\mid e(v,L/K)\equiv0\bmod{p^{l+1}}\}+\operatorname{rk}_p\operatorname{Cl}_K.
		\end{aligned}		
	\end{equation*}
	And we are done for the proof.
\end{proof}
Note that for the subgroup $\operatorname{D}_L$, not only we have a lower bound for $\operatorname{rk}_p\operatorname{D}_L$, but also an upper bound that is related to $\operatorname{rk}_p\operatorname{Cl}_{K}$.
For any algebraic extension of number fields $L/K$, there is an induced map $\operatorname{Nm}_{L/K}:\operatorname{Cl}_L\to\operatorname{Cl}_K$ by the norm map on the group of fractional ideals $\operatorname{Nm}_{L/K}:I_L\to I_K$.
The kernel of $\operatorname{Nm}_{L/K}$ is the relative class group, i.e., we have the exact sequence:
\begin{equation*}
	1\to\operatorname{Cl}(L/K)\to\operatorname{Cl}_L\stackrel{\operatorname{Nm}}{\longrightarrow}\operatorname{Cl}_K\to1.
\end{equation*}
Let us give an estimate for $\operatorname{rk}_p\operatorname{Cl}(L/K)$.
\begin{corollary}\label{cor: estimate of rel class group}
    Let $L/K$ be a finite extension of number fields.
    Fix an integer $l\geq1$ and a finite rational prime $p$.
    There exists a constant $c$, depending on the degree $[N:\mathbb{Q}]$, such that
	\begin{equation*}
		\operatorname{rk}_p p^{l}\operatorname{Cl}(L/K)\geq\#\{v\in P_K\mid e(v,L/K)\equiv0\bmod{p^{l+1}}\}-c.
	\end{equation*}
\end{corollary}
\begin{proof}
	Recall that $\mathfrak{a}_v:=\prod_{w\mid v}w^{n_{w/v}}$.
        See also~(\ref{eqn: a(p)}).
	Since $\mathfrak{a}_v$ is fixed by all Galois actions and $a_v^{e(v,L/K)}=v$, we see that
	\begin{equation*}
		\operatorname{Nm}_{L/K}(\mathfrak{a}_v)=\mathfrak{a}_v^{[L:K]}=v^{[L:K]/e(v,L/K)}.
	\end{equation*}
	By the fundamental identity 
	\begin{equation*}
		[L:K]=\sum_{w\mid v}e(w/v)f(w/v),
	\end{equation*}
	we know that $e(v,L/K)\mid[L:K]$, i.e., the norm of $\mathfrak{a}_v$ is a power of $v$.
	Clearly this implies that $\mathfrak{a}_v$ is contained in the relative class group $\operatorname{Cl}(L/K)$ if the prime ideal $v$ is principal.
	
	Define $I_{L/K}^G$ to be the subgroup of $I_L$ generated by $\{\mathfrak{a}_v\mid v\in P_K\}$.
	Let $\operatorname{D}(L/K)$ be the subgroup of $\operatorname{D}_L$ that is generated by the image of $I_{L/K}^G$, i.e.,
	\begin{equation*}
		\operatorname{D}(L/K)=I_{L/K}^G/(P_L^G\cap I_{L/K}^G).
	\end{equation*}
	We see that $I_{L/K}^G$ is a subgroup of $I_L^G$, and
	\begin{equation*}
		I_{L/K}^G I_K/I_K=I_{L/K}^G/(I_K\cap I_{L/K}^G)\cong\prod_{v\in P_K}\mathbb{Z}/e(v,L/K)\mathbb{Z},
	\end{equation*} 
	which is a subgroup of $I_L^G/I_K$.
	Using Proposition~\ref{prop: inequalities of p-rank}, the surjective homomorphism
	\begin{equation*}
		I_{L/K}^G P_K/P_K\to I_{L/K}^G I_K/I_K\to1
	\end{equation*}
	shows that 
	\begin{equation*}
		\begin{aligned}
			\operatorname{rk}_{p}p^{l}I_{L/K}^GP_K/P_K
			\geq&\operatorname{rk}_{p}p^{l}I_{L/K}^G I_K/I_K
			\\
			=&\#\{v\in P_K\mid e(v,L/K)\equiv0\bmod{p^{l+1}}\}.
		\end{aligned}		
	\end{equation*}
	On the other hand, we have that
	\begin{equation*}
		1\to\frac{P_L^G\cap I_{L/K}^G}{P_K\cap I_{L/K}^G}\to\frac{I_{L/K}^G}{P_K\cap I_{L/K}^G}\to \operatorname{D}(L/K)\to1.
	\end{equation*}
	By (\ref{eqn: estimate of class group 3}), we know that there exists some constant $c$ that only depends on $[N:K]$ such that
	\begin{equation*}
		\operatorname{rk}_{p}P_L^G/P_K\leq c.
	\end{equation*}
	Together with the above short exact sequence, we have that
	\begin{equation*}
		\begin{aligned}
			\operatorname{rk}_pp^{l}\operatorname{D}(L/K)
			\geq&\operatorname{rk}_pp^{l}\frac{I_{L/K}^G}{P_K\cap I_{L/K}^G}-\operatorname{rk}_p\frac{P_L^G\cap I_{L/K}^G}{P_K\cap I_{L/K}^G}
			\\
			\geq&\operatorname{rk}_pp^{l}I_{L/K}^G P_K/P_K-\operatorname{rk}_{p}P_L^G/P_K
			\\
			\geq&\#\{v\in P_K\mid e(v,L/K)\equiv0\bmod{p^{l+1}}\}-c.
		\end{aligned}
	\end{equation*}
\end{proof}
\subsection{Connections with distribution of class groups}
We start with a generalization of the hypothesis in the author~\cite{wang2022distributionOT}.

In this section, let $G$ be a transitive permutation group in $S_n$.
Fix a number field $K$.
Define $\mathcal{E}$ as a subset of $\mathcal{E}(G;K)$.
Let $C$ be a counting function on $\mathcal{E}$ such that $\mathcal{E}$ is ordered by $C$.

In concrete examples, we may exclude the fields $L$ that contain a strictly larger group of roots of unity than $K$.
In other words, a field $L\in\mathcal{E}(G;K)$ is contained in $\mathcal{E}$ if and only if
\begin{equation*}
	\mu(K)=\mu(L),
\end{equation*}
where $\mu(K)$ means the group of roots of unity in $K$.
In other cases, one may consider the fields that satisfy some local conditions.
For example, define $\mathcal{E}$ to be the subset of $\mathcal{E}(C_2;\mathbb{Q})$ such that a quadratic $L$ is contained in $\mathcal{E}$ if and only if $G_{\infty}(L/\mathbb{Q})=C_2$, where $G_{\infty}$ means the decomposition group at infinity.
In other words, $\mathcal{E}$ is the set of all the imaginary quadratic number fields.

Let $\Omega$ be a subset of $G$.
We say that $\Omega$ is closed under conjugation if for each $x\in\Omega$, its orbit under the action of $G$ is included in $\Omega$,
i.e., $x\in\Omega\Rightarrow G\cdot x\subseteq\Omega$.
We say that $\Omega$ is closed under invertible powering if for each $x,y\in G$ such that $\langle x\rangle=\langle y\rangle$ we have $x\in\Omega$ if and only if $y\in\Omega$.
Here, $\langle x\rangle$ is the subgroup of $G$ generated by $x$.
\begin{definition}\label{def: specifications in general}
    Fix a number field $K$ with a nonempty set $T$ of prime ideals.
    Let $G$ be a transitive permutation group,
    and $\Omega$ a nonempty subset of $G$ that does not contain the identity and is closed under invertible powering and conjugation.
    For each $L\in\mathcal{E}$, define
    \begin{equation*}
        \mathcal{P}_{T}(\Omega,L):=\{\mathfrak{p}\in T: \mathfrak{p}\nmid\lvert G\rvert\infty\text{ and } y_{\mathfrak{p}}\in\Omega\},
    \end{equation*}
    where $y_{\mathfrak{p}}$ is a generator of the inertia subgroup $I_{\mathfrak{P}}$ for some $\mathfrak{P}/\mathfrak{p}$.
    If $T=\mathcal{P}_K$, then we just write it as $\mathcal{P}_{K}(\Omega,L)$.
    For each non-negative integer $\gamma$,
    define $\mathcal{E}_{\Omega,\gamma}^{T}$ to be the subset of $\mathcal{E}$ by
    \begin{equation*}
        \mathcal{E}_{\Omega,\gamma}:=\{L\in\mathcal{E}\mid\#\mathcal{P}_{\Omega,L}(T)=\gamma\}.
    \end{equation*}
\end{definition}
Note that when $\mathfrak{p}\nmid\lvert G\rvert\infty$, it is unramified or tamely ramified.
So its inertia subgroup (up to conjugate) is cyclic and admits a generator.
Due to our conditions on $\Omega$, it makes no difference which prime $\mathfrak{P}/\mathfrak{p}$ or which one of the generators of the inertia subgroup $I_{\mathfrak{P}}$ is picked up to do the computation.
In other words, the notation $\mathcal{E}_{\Omega,\gamma}^{T}$ is well-defined.
And for each $L\in\mathcal{E}$, it must be contained in one of the $\mathcal{E}_{\Omega,\gamma}^{T}$ when $\gamma$ runs over all non-negative integers.
Therefore, we have that
\begin{equation*}
	\mathcal{E}=\bigsqcup_{\gamma=0}^{\infty}\mathcal{E}_{\Omega,\gamma}^{T}.
\end{equation*}
Clearly, $\mathcal{E}_{\Omega,\gamma}^{T}$ is equipped with the counting function $C$ coming from $\mathcal{E}$.
In particular, we require that for any $X>0$, the set $\{L\in\mathcal{E}\mid C(L)<X\}$ be finite, so that we can avoid strange counting functions.
Then we see that the functions $N_{C}(\mathcal{E},X)$ and $N_{C}(\mathcal{E}_{\Omega,\gamma}^{T},X)$ are both well-defined.
And their arithmetic meaning is simply counting fields.
\begin{hypothesis}\label{hypothesis: comparison between field counting results}
	\begin{enumerate}
		\item For each non-negative integer $\gamma$, there exists some non-negative integer $\gamma'$ such that
		\begin{equation*}
			N_{C}(\mathcal{E}_{\Omega,\gamma}^{T},X)=o(N_{C}(\mathcal{E}_{\Omega,\gamma'}^{T},X)),
		\end{equation*}
		as $X\to\infty$.
		In this case, we say that Hypothesis 1 holds for $((\mathcal{E},C),\Omega,T)$.
		\item For each non-negative integer $\gamma$, we have that
		\begin{equation*}
			N_{C}(\mathcal{E}_{\Omega,\gamma}^{T},X)=o(N_{C}(\mathcal{E},X)),
		\end{equation*}
		as $X\to\infty$.
		In this case, we say that Hypothesis 2 holds for $((\mathcal{E},C),\Omega,T)$.
	\end{enumerate}
\end{hypothesis}
It is not hard to show that Hypothesis 1 implies 2, and we just leave it to the reader.
\begin{example}
	Though in this paper, we will study the questions under a suitable set-up, so that we could prove that the hypothesis is true.
	But there are ``counter-examples''.
	Let $G:=S_3$.
	Then $\mathcal{E}=\mathcal{E}(G;\mathbb{Q})$ means the non-Galois cubic fields.
	Let $\Omega:=\{(123),(132)\}\subseteq G$, and $T$ be the set of all rational primes.
	The notation $\mathcal{E}_{\Omega,\gamma}^T$ means all non-Galois cubic fields $L$ such that there are exactly $\gamma$ totally ramified primes $p\nmid6$.
	According to Davenport-Heilbronn~\cite{Davenport1971Cubic}, if we order cubic fields by (absolute) discriminant, denoted by $d$, then there exists some constant $c>0$ such that
	\begin{equation*}
		\lim_{X\to\infty}\frac{N_{d}(\mathcal{E}_{\Omega,0}^T,X)}{N_d(\mathcal{E},X)}=c.
	\end{equation*}
	In other words, the hypothesis fails if we order the cubic fields by the discriminant.
	To explain the above phenomena, note that the discriminant of a cubic field could be written as $df^2$ where $d$ is the discriminant of the associated quadratic number field, while $f$ means the product of all totally ramified primes.
	Therefore, if we order fields by the discriminant, then the properties of quadratic number fields may affect the statistical results in a nontrivial way.
	Of course, this is just a philosophical reasoning, not a proof.
	See also Bartel and Lenstra~\cite{bartel2020class}.
\end{example}
Then we show the connection between the estimate of class groups and the distribution of class groups.
\begin{theorem}\label{thm: 0-prob and infty-moment}
    Fix a rational prime $p$ and a non-negative integer $l$.
    \begin{enumerate}
        \item Assume $e_G\notin\Omega$ to be a nonempty subset of $G$ closed under invertible powering and conjugation such that there exists some constant $c\geq0$ for each $L\in\mathcal{E}$, we have
        \begin{equation*}
            \operatorname{rk}_p p^{l}\operatorname{Cl}(L/K)\geq\#\mathcal{P}_{T}(\Omega,L)-c.
	\end{equation*}
	If Hypothesis~\ref{hypothesis: comparison between field counting results}(2) holds for $((\mathcal{E},C),\Omega,T)$, then for each non-negative integer $r$, we have that
	\begin{equation*}
		\mathbb{P}_{C,\mathcal{E}}(\operatorname{rk}_pp^{l}\operatorname{Cl}(L/K)\leq r)=0.
	\end{equation*}
	\item If there exists a constant $0\leq a<1$ such that for each non-negative integer $r$, we have $\mathbb{P}_{C,\mathcal{E}}(\operatorname{rk}_p p^{l}\operatorname{Cl}(L/K)\leq r)\leq a$, then
	\begin{equation*}
		\mathbb{E}_{C,\mathcal{E}}(\lvert\operatorname{Hom}(p^{l}\operatorname{Cl}(L/K),C_{p})\rvert)=+\infty.
	\end{equation*}
    \end{enumerate}
\end{theorem}
\begin{proof}
	We first show (1).
	By the condition, for each non-negative integer $\gamma$ and for each $m>c$ if $L\in\mathcal{E}^{T}_{\Omega,\gamma+m}$ then
	\begin{equation*}
		\operatorname{rk}_p p^{l}\operatorname{Cl}(L/K)>\gamma.
	\end{equation*}
	So by Hypothesis~\ref{hypothesis: comparison between field counting results}(2), for each non-negative integer $r$, we have
	\begin{equation*}
		\begin{aligned}
			\mathbb{P}_{C,\mathcal{E}}(\operatorname{rk}_p p^l\operatorname{Cl}(L/K)\leq r)
			\leq&\lim_{X\to\infty}\frac{\sum_{\gamma=0}^{r+c}N_{C}(\mathcal{E}^{T}_{\Omega,\gamma},X)}{N_{C}(\mathcal{E},X)}\\
			=&\sum_{\gamma=0}^{r+c}\lim_{X\to\infty}\frac{N_{C}(\mathcal{E}^{T}_{\Omega,\gamma},X)}{N_{C}(\mathcal{E},X)}=0.
		\end{aligned}
	\end{equation*}	
	Then we show (2).
	For each small real number $0<\epsilon<1-a$, and for each real number $M>0$, there exists some integer $r_M\geq1$ such that for all $r\geq r_M$ we have 
	\begin{equation*}
		p^r(1-a-\epsilon)>M.
	\end{equation*}
	Recall that
	\begin{equation*}
		\mathbb{P}_{C,\mathcal{E}}(\operatorname{rk}_p p^{l}\operatorname{Cl}(L/K)\leq r)=
		\lim_{X\to\infty}
		\frac{\#\{L\in\mathcal{E}\mid C(L)<X\text{ and }\operatorname{rk}_p p^{l}\operatorname{Cl}(L/K)\leq r\}}
		{N_{C}(\mathcal{E},X)}.
	\end{equation*}
	So for the same $\epsilon$, there exists $N=N(\epsilon,r_M)>0$ such that for all $X>N$ we have
	\begin{equation*}
		\frac{\#\{L\in\mathcal{E}\mid C(L)<X\text{ and }\operatorname{rk}_p p^{l}\operatorname{Cl}(L/K)\leq r_M\}}
		{N_{C}(\mathcal{E},X)}<a+\epsilon.
	\end{equation*}
	Then for each $X>N$ we see that 
	\begin{equation*}
		\begin{aligned}
			&\frac{\sum_{C(L)<X}\lvert\operatorname{Hom}(p^l\operatorname{Cl}(L/K),C_p)\rvert}
			{N_{C}(\mathcal{E},X)}
			\\
			\geq&\frac{p^{r_M}\cdot\#\{C(L)<X\mid\operatorname{rk}_p p^l\operatorname{Cl}(L/K)\geq r_M\}}
			{N_{C}(\mathcal{E},X)}
			\\
			>&p^{r_M}(1-a-\epsilon)>M.
		\end{aligned}
	\end{equation*}
	This shows that the infinite moment holds.	
	And we are done for the whole proof.
\end{proof}
Finally we explain the condition in the theorem by the following result.
\begin{lemma}\label{lemma: ramification of primes in general}
    Let $G\subseteq S_n$ be a transitive permutation group with $G_1=\operatorname{Stab}_G(1)$.
    Let $N/K$ be a Galois $G$-extension with $L:=N^{G_1}$.
    Fix a finite prime $v$ of $K$, and let $x$ be a prime of $N$ above $v$ with the decomposition group $G_{x}\subseteq G$.
	\begin{enumerate}
		\item Define $\mathcal{P}_{v}(L):=\{w\in\mathcal{P}_L:w\mid v\}$.
		The morphism of sets
		\begin{equation*}
			G_1\backslash G/G_{x}\to\mathcal{P}_{v}(L),\quad G_1\sigma G_{x}\mapsto\sigma x\vert_L
		\end{equation*}
		is bijective.
		\item If $w=x\vert_L$, then
		\begin{equation*}
			[L_{w}:K_{v}]=\# G_{x}\cdot1
			\quad\text{and}\quad
			e(w/v)=\# I_{x}\cdot1,
		\end{equation*}
		where $I_{x}$ is the inertia subgroup, and $G_{x}\cdot1$, resp. $I_{x}\cdot1$, is the orbit.
		\item We have
		\begin{equation*}
			e(v,L/K)=\gcd(\# I_{v}\cdot i)_{i=1}^n,
		\end{equation*}
		where $I_{v}\subseteq G$ is the inertia subgroup of $v$ defined up to conjugacy.
	\end{enumerate}
\end{lemma}
\begin{proof}[Proof of the lemma]
    The first statement is a standard result from algebraic number theory.
    See Neukirch~\cite[p.54-55]{neukirch2013algebraic} for an explanation.

    \emph{Proof of (2)}.    
        We move on to the second one.
        We have the following chain of extensions of local fields:
        \begin{equation*}
            K_{v}\subseteq L_{w}\subseteq N_{x}.
	\end{equation*}
	Since $G_{x}\cong G(N_{x}/K_{v})$, we have that
	\begin{equation*}
		L_{w}=N_{x}^{G_{x}\cap G_1}.
	\end{equation*}
	To see why this is true, it is clear that the fixed subgroup $G(N_{x}/L_{w})$ is a subgroup of $G_{x}\cap G_1$, for $G(N_{x}/L_{w})$ must fix $L$.
	For the other direction, note that $L_{w}$ could be defined by the completion of $L$ with the valuation $w$.
	So for each $g\in G_{x}\cap G_1$, it fixes all the $(L,w)$-Cauchy sequences, hence fixing the field $L_{w}$.
	This first implies that
	\begin{equation*}
		[L_{w}:K_{v}]=[G_{x}:G_{x}\cap G_1].
	\end{equation*}
	By definition of the inertia subgroup, the short exact sequence
	\begin{equation*}
		1\to I_{x}\to G_{x}\to G(\kappa_{x}/\kappa_{v})\to 1
	\end{equation*}
	induces the short exact sequence
	\begin{equation*}
		1\to  I_{x}\cap G_1\to G_{x}\cap G_1\to G(\kappa_{x}/\kappa_{w})\to 1.
	\end{equation*}
	In other words, we have
	\begin{equation*}
		e(w/v)=[I_{x}:I_{x}\cap G_1].
	\end{equation*}
	Then let us translate it into the context of permutation groups.
	The right cosets $G_1\backslash G$ form a right $G$-set, i.e., $(G_1\sigma)\cdot g\mapsto G_1(\sigma g)$.
	And the $G$-action induces an injective homomorphism $\varphi:G\to S_n$ which preserves $G_1$, i.e., $\varphi(G_1)=G_1$.
	So we see that $G$ and $\varphi(G)$ are both transitive and their actions on $\{1,\dots,n\}$ are determined by $\operatorname{Stab}_G(1)=G_1=\operatorname{Stab}_{\varphi(G)}(1)$, 
    which implies that $G$ and $\varphi(G)$ together with their actions on $\{1,\dots,n\}$ are permutation isomorphic.
		
	So the group index $[G_{x}:G_{x}\cap G_1]$ is the same thing as $\#  G_{x}\cdot1$,
	i.e., this number is independent of such a map $\varphi:G\to S_n$ that preserves $G_1$.
	And similarly we have that $e(w/v)=\# I_{x}\cdot 1$.
	This means that the size of the orbit of $1$ under the action of $G_{x}$ is the degree $[L_{w}:K_{v}]$, and the ramification index is given by the size of the orbit induced by $I_{x}$.
    
    \emph{Proof of (3)}.    
    Finally, let us prove the identity
	\begin{equation*}
		e(v,L/K)=\gcd(\# I_{v}\cdot i)_{i=1}^n.
	\end{equation*}
	Note that the double cosets are in one-to-one correspondence with the orbits $\{(G_1\sigma)\cdot G_{x}\mid\sigma\in G\}$.
	So we also have a bijective map $\psi:\{G_{x}\cdot i\mid1\leq i\leq n\}\to\mathcal{P}_{v}(L)$.
	Note that the number of \emph{different} orbits may be less than $n$ in general.
	Fix some $1\leq i\leq n$.
        Define $w':=\psi(G_{x}\cdot i)$, and let $x'$ be a prime of $N$ that lies above it.
	By applying the method in the proof of (2) to $K_{v}\subseteq L_{w'}\subseteq N_{x'}$, we see that 
	\begin{equation*}
		[L_{w'}:K_{v}]=\# G_{x'}\cdot 1\quad\text{and}\quad e(w'/v)=\# I_{x'}\cdot 1.
	\end{equation*}	
	Clearly $G_{x}$ and $G_{x'}$ are conjugate in $G$.
	So there exists some $\sigma\in G$ such that $G_{x'}=\sigma G_{x}\sigma^{-1}$.
	Note that our $G$-action on $G_1\backslash G$ is on the right.
        By identifying the stabilizers $\{G_1,\dots,G_n\}$ and the right cosets $G_1\backslash G$ and $n$ objects in general $\{1,\dots,n\}$ as $G$-sets, we have 
	\begin{equation*}
		\sigma^{-1}G_1\sigma=G_1\cdot\sigma=\sigma\cdot 1.
	\end{equation*}
	Since $G_1\sigma G_{x}\mapsto\sigma x\vert_L=w'$, we see that $\sigma^{-1}G_1\sigma=G_i$, the stabilizer of $i$.
	So we have that
	\begin{equation*}
		\begin{aligned}
			[G_{x'}:G_{x'}\cap G_1]
			=&[\sigma G_{x}\sigma^{-1}:\sigma G_{x}\sigma^{-1}\cap G_1]
			\\
			=&[G_{x}:G_{x}\cap G_i]
			=\# G_{x}\cdot i.
		\end{aligned}		
	\end{equation*}
	Similar method shows that $e(w'/v)=\# I_{x}\cdot i$.
	The choice of $x\mid v$ only affects the labeling.
	So we finally reach the identity
	\begin{equation*}
		e(v,L/K)=\gcd(e(w/v))_{w\mid v}=\gcd(\# I_{v}\cdot i)_{i=1}^n.
	\end{equation*}
\end{proof}
Lemma~\ref{lemma: ramification of primes in general} provides a method to pick up a suitable subset $\Omega$ of $G$ that satisfies the condition of Theorem~\ref{thm: 0-prob and infty-moment}.
Or we could say that the lemma allows us to prescribe the ramification of primes in an extension by the language of group theory.
Let us use an example to show it.
\begin{example}
	Let $p$ be a finite rational prime and let $l$ be a non-negative integer as in Theorem~\ref{thm: 0-prob and infty-moment}.
	Consider the set
	\begin{equation*}
		\Omega_{p^{l+1}}:=\{g\in G\mid\gcd(\#\langle g\rangle\cdot i)_{i=1}^n\equiv0\bmod{p^{l+1}}\}.
	\end{equation*}
	Clearly it is closed under invertible powering and conjugation, and $e_G\notin\Omega_{p^{l+1}}$.
	Assume that $\Omega_{p^{l+1}}$ is nonempty in the rest of this example.
	If $\mathfrak{p}$ is tamely ramified in $L/K$ such that its tame inertia generator $y_{\mathfrak{p}}$ is contained in $\Omega_{p^{l+1}}$, then we have that
	\begin{equation*}
		e(\mathfrak{p},L/K)=\gcd(\#\langle y_{\mathfrak{p}}\rangle\cdot i)_{i=1}^n\equiv0\bmod{p^{l+1}}.
	\end{equation*}
	So our construction $\Omega_{p^{l+1}}$ with the nonempty condition is an example that satisfies the condition of Theorem~\ref{thm: 0-prob and infty-moment}.
	To be precise, let $T$ be the set of principal prime ideals $\mathfrak{p}$ of $K$ such that $\mathfrak{p}\nmid\lvert G\rvert$.
	By Corollary~\ref{cor: estimate of rel class group}, there exists some constant $c>0$ such that for all $L\in\mathcal{E}(G)$, we have
	\begin{equation*}
		\begin{aligned}
			\operatorname{rk}_p p^l\operatorname{Cl}(L/K)
			\geq&\#\{\mathfrak{p}\in T\mid e(\mathfrak{p},L/K)\equiv0\bmod{p^{l+1}}\}-c
			\\
			\geq&\#\mathcal{P}_{T}(\Omega_{p^{l+1}},L)-c.
		\end{aligned}		
	\end{equation*}	
	For a more concrete one, let $G:=S_3$, and $L/\mathbb{Q}$ be a non-Galois cubic field.
	Then for any rational prime $p\nmid6\infty$, we have
	\begin{equation*}
		e(p,L/\mathbb{Q})\equiv0\bmod{3}\iff I_p=\langle(123)\rangle.
	\end{equation*}
	Let $\Omega:=\{(123),(132)\}$, which is just $\Omega_3$ in this case.
	By Theorem~\ref{thm: estimate of class group}, there exists some constant $c>0$ such that for all non-Galois cubic field $L$, we have
	\begin{equation*}
		\operatorname{rk}_3\operatorname{Cl}_L\geq\#\mathcal{P}_{\mathbb{Q}}(\Omega,L)-c.
	\end{equation*}
\end{example}
The main results of this section are related to the concept called ``ambiguous ideal classes''.
The study of genus theory (e.g. Ishida~\cite{ishida1976genus}) could also be viewed as part of this topic.
For us, these results provide an algebraic preliminary for the statistical results, and we will start with the fundamental set-ups in the next section.

\section{Local specifications}\label{sec: specifications}
In this section, we establish the fundamental notions for our proof of the main results.
Fix a number field $K$, i.e., the base field of this section.
For each prime $\mathfrak{p}$ of $K$, let $K_{\mathfrak{p}}$ be the completion of $K$ at $\mathfrak{p}$.
Define
\begin{equation*}
	\begin{aligned}
		U_{\mathfrak{p}}:=\left\{
		\begin{aligned}
			&\{x\in K_{\mathfrak{p}}:\lvert x\rvert_{\mathfrak{p}}=1\}\quad&\text{if }\mathfrak{p}\text{ is finite};
			\\
			&\mathbb{R}_+^*\quad&\text{if }\mathfrak{p}\text{ is infinite real};
			\\
			&\mathbb{C}^*\quad&\text{if }\mathfrak{p}\text{ is infinite complex}.
		\end{aligned}\right.
	\end{aligned}
\end{equation*}
In other words, $U_{\mathfrak{p}}$ is the unit group of $K_{\mathfrak{p}}$.
Let $J_K$ be the group of id{\`e}les of $K$, and $\operatorname{C}_K$ the id{\`e}les class group.
Let $S$ be a finite set of primes containing the ones at infinity.
Define
\begin{equation*}
	J_K^{S}:=\prod_{\mathfrak{p}\in S}K_{\mathfrak{p}}^*\times\prod_{\mathfrak{p}\notin S}U_{\mathfrak{p}}
\end{equation*} 
be the $S$-id{\`e}les.
Define $K^S:=J_K^S\cap K^*$ to be the $S$-units of $K$.

Let $G$ be a finite group, and $F$ be a fixed field.
An {\'e}tale $F$-algebra $A$ of degree $\lvert G\rvert$ is called a $G$-structured $F$-algebra $A$ if there is an inclusion $\varphi_A:G\hookrightarrow\operatorname{Aut}_F(A)$ such that $G$ acts transitively on the idempotents of $A$.
If $A$ and $B$ are two $G$-structured $F$-algebras, then an isomorphism $f:A\to B$ is an isomorphism of $F$-algebras such that the induced map $f_*:\operatorname{Aut}_F(A)\to\operatorname{Aut}_F(B)$ restricted to $G$ is the identity map, i.e.,
for each $g\in G$, we have 
\begin{equation*}
	f_*(\varphi_A(g))=\varphi_B(g).
\end{equation*}
The most important example to keep in mind is the following.
If $M/K_{\mathfrak{p}}$ is an extension of local fields with an injective map $G(M/K_{\mathfrak{p}})\hookrightarrow G$, then $\Sigma_{\mathfrak{p}}:=\operatorname{Ind}^G_H M$ is a $G$-structured $K_{\mathfrak{p}}$-algebra, where $H$ is the image of $G(M/K_{\mathfrak{p}})$ in $G$.
Roughly speaking, this is another way to state the local-global principle.

If $F$ is a local field, then we define the decomposition group and the inertia subgroup of $A$ over $F$ (up to conjugate) as follows.
For a fixed centrally primitive idempotent $e$ of $A$, the field $eA$ corresponds to a field extension $eA/F$ with the corresponding extension of valuations $w/v$.
So, if $g\in G$ fixes the component $eA$, then it is clearly an element in the decomposition group $G_{w}\subseteq G$.
It also admits the inertia group $I_{w}$ which comes from the exact sequence
\begin{equation*}
	1\to I_{w}\to G_{w}\to G(\lambda/\kappa)\to1,
\end{equation*}
where $\lambda$, resp. $\kappa$, is the residue field of $eA$, resp. $F$.
\begin{definition}
	Let $K$ be a number field, and $G$ be a finite group.
	\begin{enumerate}
		\item Fix a prime $\mathfrak{p}$ of $K$.
        Define a specification at $\mathfrak{p}$ (with respect to $G$) to be a $G$-structured $K_{\mathfrak{p}}$-algebra $\Sigma_{\mathfrak{p}}$ which is considered up to isomorphism, such that there exists a subgroup $H\subseteq G$ and an $H$-extension $F/K_{\mathfrak{p}}$ with $\operatorname{Ind}^G_H F\cong\Sigma_{\mathfrak{p}}$.
		We say that a $G$-extension $L/K$ satisfies the specification $\Sigma_{\mathfrak{p}}$ at $\mathfrak{p}$ if $L_{\mathfrak{p}}=L\otimes K_{\mathfrak{p}}\cong\Sigma_{\mathfrak{p}}$.
		\item Define a local specification of $K$ (with respect to $G$) to be a collection of specifications $\Sigma=(\Sigma_{\mathfrak{p}})_{\mathfrak{p}}$ such that for all but finitely many primes the specification $\Sigma_{\mathfrak{p}}$ is unramified.
		We say that a $G$-extension $L/K$ satisfies the local specification $\Sigma$ if for each prime $\mathfrak{p}$ the extension $L/K$ satisfies the specification $\Sigma_{\mathfrak{p}}$.
	\end{enumerate}
\end{definition}
\begin{remark}
	Note that a $G$-structured algebra $\Sigma_{\mathfrak{p}}$ over $K_{\mathfrak{p}}$ is an {\'e}tale algebra.
	Its ramification is well-defined via its components.
	In particular, since it admits a transitive $G$-action, the ramification could be defined using a single component.
	For example, if $F$ is one of the component and $F/K_{\mathfrak{p}}$ is tamely ramified, then every component is tamely ramified over $K_{\mathfrak{p}}$, and we can talk about the inertia generator (up to conjugate).
\end{remark}
When $G$ is abelian, specifications at a prime $\mathfrak{p}$ are literally \emph{all} the $G$-structured $K_{\mathfrak{p}}$-algebras by the following result.
\begin{lemma}\label{lemma: Galois representation}
	If $G$ is a finite abelian group, then we have the one-to-one correspondence between the following two sets
	\begin{equation*}
		\{\text{isomorphism classes of }G\text{-structured }F\text{-algebra}\}\leftrightarrow\operatorname{Hom}(G_F,G).
	\end{equation*}
\end{lemma}
\begin{proof}		
	See Wood~\cite[Lemma 2.6]{wood2010probabilities} for example.
\end{proof}
Note that when $G$ is finite abelian, locally we have the one-to-one correspondence between $\operatorname{Hom}(G_{K_{\mathfrak{p}}},G)$ and $\operatorname{Hom}(K_{\mathfrak{p}}^*,G)$.
This means that every specification at $\mathfrak{p}$ corresponds to a continuous map $\chi:K_{\mathfrak{p}}^*\to G$ and vice versa.
Globally we can replace $K_{\mathfrak{p}}^*$ by the id{\`e}les class group $\operatorname{C}_K$, and the surjective maps $\operatorname{C}_K\to G$ correspond exactly to the $G$-extensions $L/K$.

Let us give the definition of a counting function based on local specifications.
\begin{definition}\label{def: product of ramified primes}
	Let $G$ be a finite group.
	Define $c_G:G\to\mathbb{R}_{\geq0}$ to be a function as follows.
	\begin{enumerate}
		\item $c_G(g)=0$ if and only if $g$ is the identity of $G$;
		\item for each $g,h\in G$, for each integer $a$ that is relatively prime to the order of $g$, we have that $c_G(g)=c_G(hg^ah^{-1})$.
	\end{enumerate} 
	For a prime $\mathfrak{p}\mid\lvert G\rvert\infty$ of $K$, let $c_{\mathfrak{p}}:\{$ specifications at $\mathfrak{p}\}\to\mathbb{R}_{\geq0}$ be any function.
	For each $\mathfrak{p}$ of $K$, define the function $c:\{$ specifications at $\mathfrak{p}\}\to\mathbb{R}_{\geq0}$ with respect to $c_G$ and $\{c_{\mathfrak{p}}\}_{\mathfrak{p}\mid\lvert G\rvert\infty}$ as follows:
	\begin{equation*}
		c(\Sigma_{\mathfrak{p}})=\left\{
		\begin{aligned}
			&c_G(y_{\mathfrak{p}})\quad&\text{if }\mathfrak{p}\nmid\lvert G\rvert\infty
			\\
			&&\text{and }y_{\mathfrak{p}}\text{ is a inertia generator}.
			\\
			&c_{\mathfrak{p}}(\Sigma_{\mathfrak{p}})\quad&\text{otherwise.}
		\end{aligned}\right.
	\end{equation*}
	If $L/K$ is a $G$-extension, then define the counting function $C$ with respect to $c_G$ and $\{c_{\mathfrak{p}}\}_{\mathfrak{p}\mid\lvert G\rvert\infty}$ by the expression
	\begin{equation*}
		C(L):=\prod_{\mathfrak{p}}\mathfrak{Np}^{c(L\otimes K_{\mathfrak{p}})}.
	\end{equation*}
	Here $\mathfrak{Np}$ is the absolute norm for finite prime $\mathfrak{p}$, and set $\mathfrak{Np}:=1$ if $\mathfrak{p}\mid\infty$.
\end{definition}
See also Wood~\cite[Section 2]{wood2010probabilities}.
We can generalize the counting function $C$ to continuous homomorphisms $C:J_K^S\to G$ when $G$ is finite abelian by the following steps.
\begin{lemma}\label{lemma: extension of local map}
    Let $G$ be a finite group.
    For each prime $\mathfrak{p}$ of $K$, and for each continuous homomorphism $\chi_{\mathfrak{p}}:U_{\mathfrak{p}}\to G$, there exists some continuous homomorphism $\tilde{\chi}_{\mathfrak{p}}:K_{\mathfrak{p}}^*\to G$ that extends $\chi_{\mathfrak{p}}$.
\end{lemma}
\begin{proof}
	Using the fact that
	\begin{equation*}
		K_{\mathfrak{p}}^*\cong U_{\mathfrak{p}}\times\mathbb{Z}
	\end{equation*}
	both algebraically and topologically, we see that the required map $\tilde{\chi}_{\mathfrak{p}}$ simply comes from a choice of $\mathbb{Z}\to G$.
\end{proof}
When $G$ is also abelian, another way to state this lemma is that for each continuous map $\chi_{\mathfrak{p}}:U_{\mathfrak{p}}\to G$, there exists an algebra $\Sigma_{\mathfrak{p}}$ such that the corresponding map $K_{\mathfrak{p}}^*\to G$ extends $\chi_{\mathfrak{p}}$.
Of course, such an algebra may not be unique.
\begin{lemma}\label{lemma: product of ramified primes}
    Let $G$ be a finite abelian group.
    Fix a number field $K$ with a finite set $S$ of primes containing the ones at $\lvert G\rvert\infty$.
    Given functions $c_G$ and $\{c_{\mathfrak{p}}\}_{\mathfrak{p}\mid\lvert G\rvert\infty}$ as in Definition~\ref{def: product of ramified primes}, there exists a counting function $C$ defined for continuous homomorphisms $\chi:J_K^S\to G$, such that if $L/K$ is an abelian $G$-extension corresponding to the Artin reciprocity map $\tilde{\chi}:\operatorname{C}_K\to G$, then
	\begin{equation*}
		C(L)=C(\chi),
	\end{equation*}
	where $\chi:J_K^S\to G$ is induced by $\tilde{\chi}$.
\end{lemma}
\begin{proof}
	For each $\mathfrak{p}\mid\lvert G\rvert\infty$, and for each local map $\chi_{\mathfrak{p}}:K_{\mathfrak{p}}^*\to G$, we have the function
	\begin{equation*}
		c(\chi_{\mathfrak{p}}):=c(\Sigma_{\mathfrak{p}})
	\end{equation*}
	where $\Sigma_{\mathfrak{p}}$ is the corresponding $G$-structured $K_{\mathfrak{p}}$-algebra.
	
	On the other hand, the inertia subgroup of $\Sigma_{\mathfrak{p}}/K_{\mathfrak{p}}$ is determined by the restriction $\chi_{\mathfrak{p}}:U_{\mathfrak{p}}\to G$.
        Together with Lemma~\ref{lemma: extension of local map}, we see that the function $c$ is actually well-defined for $\operatorname{Hom}(U_{\mathfrak{p}},G)$ for all $\mathfrak{p}\nmid\lvert G\rvert\infty$.
	And we can define
	\begin{equation*}
		C(\chi):=\prod_{\mathfrak{p}}\mathfrak{Np}^{c(\chi_{\mathfrak{p}})}.
	\end{equation*}
	It coincides with $C(L)$ if $\chi$ is induced by $L$, for they give the same local factors.
\end{proof}
In our paper, we mainly consider ``specifications defined globally''.
For example, we can consider quadratic number fields with exactly $\gamma$ ramified primes, where $\gamma$ is a non-negative integer.
\begin{definition}\label{def: specification}
	Let $G$ be a finite group.
	Let $T$ be a (possibly infinite) set of primes of the number field $K$.
	\begin{enumerate}
		\item A specification at $T$ is defined to be a collection of specifications $\Sigma=(\Sigma_{\mathfrak{p}})_{\mathfrak{p}\in T}$, such that for all but finitely many primes $\mathfrak{p}\in T$ the specification $\Sigma_{\mathfrak{p}}$ is unramified.
		\item Let $L/K$ be a $G$-extension.
		We say that $L$ satisfies the specification $\Sigma$ at $T$, denoted by $L\sim\Sigma$, 
		if $L_{\mathfrak{p}}\cong\Sigma_{\mathfrak{p}}$ for all $\mathfrak{p}\in T$.	
		\item Let $\Phi$ be a set of specifications at $T$.
		We write $L\in\Phi$ if there exists some $\Sigma\in\Phi$ such that $L\sim\Sigma$.
	\end{enumerate}	
\end{definition}
The definition may not work well for number fields, especially when $T$ is infinite.
For example, for each $p\in\mathfrak{P}$, let $\Sigma_p$ be some unramified local specification, and define $\Sigma:=\prod_{p\in\mathcal{P}}\Sigma_p$.
Clearly when $G$ is nontrivial, there is no $G$-extension $L/\mathbb{Q}$ that satisfies $\Sigma$.

Then we generalize it to the ``specifications of morphisms''.
Let $\chi:J_K^S\to G$ be a homomorphism.
For each $\mathfrak{p}\notin S$, the information of $\chi_{\mathfrak{p}}:U_{\mathfrak{p}}\to G$ is not enough to give a fixed $G$-structured $K_{\mathfrak{p}}$-algebra.
So, if we want a notation of ``specifications for morphisms'', then the idea is to consider only the set of specifications ``defined by the inertia generators''.
\begin{definition}\label{def: specification of a morphism}
    Let $G$ be a finite abelian group.
    Fix a number field $K$ with a finite set $S$ of primes containing the ones at $\lvert G\rvert\infty$.
	\begin{enumerate}
		\item For each prime $\mathfrak{p}\nmid\lvert G\rvert\infty$, let $\Phi(\mathfrak{p})$ be a set of specifications at $\mathfrak{p}$.
		We say that $\Phi(\mathfrak{p})$ is defined by inertia generators if for each specifications $\Sigma_{\mathfrak{p},1}$ and $\Sigma_{\mathfrak{p},2}$ at $\mathfrak{p}$ with the same inertia subgroup we have $\Sigma_{\mathfrak{p},1}\in\Phi(\mathfrak{p})$ if and only if $\Sigma_{\mathfrak{p},2}\in\Phi(\mathfrak{p})$.
		\item Let $T$ be a set of primes of $K$.
		For each prime $\mathfrak{p}\in T$, let $\Phi(\mathfrak{p})$ be a set of specifications at $\mathfrak{p}$ and let
		\begin{equation*}
			\begin{aligned}
				\Phi:=&\{\Sigma=(\Sigma_{\mathfrak{p}})_{\mathfrak{p}\in T}\mid\Sigma_{\mathfrak{p}}\in\Phi(\mathfrak{p})
				\text{ and }\Sigma\text{ is a specification at }T\}
				\\
				=&\sideset{}{'}{\prod}_{\mathfrak{p}\in T}\Phi(\mathfrak{p}),
			\end{aligned}			
		\end{equation*}
		where $\prod'$ means the restricted product whose restriction is given by the set of unramified specifications at $\mathfrak{p}$.
        We say that $\Phi$ is defined by inertia generators if for each $\mathfrak{p}\in T\backslash S$ the set $\Phi(\mathfrak{p})$ is defined by inertia generators.
		\item Let $\Phi=\sideset{}{'}{\prod}_{\mathfrak{p}\in T}\Phi(\mathfrak{p})$ be the set of specifications at $T$ that is defined by inertia generators.
		For each map $\chi:J_K^S\to G$, we say that $\chi$ satisfies $\Phi$, denoted by $\chi\in\Phi$, if for each $\mathfrak{p}$ and for each algebra $\tilde{\Sigma}_{\mathfrak{p}}$ whose local map extends $\chi_{\mathfrak{p}}$, we have that $\tilde{\Sigma}_{\mathfrak{p}}\in\Phi(\mathfrak{p})$.
	\end{enumerate}	
\end{definition}
It is clear that if $L/K$ is an abelian $G$-extension corresponding to the map $\chi:\operatorname{C}_K\to G$, then the induced map $J_K^S\to G$ satisfies $\Phi$ if and only if $L\in\Phi$, where $\Phi$ is a set of specifications defined by inertia generators.
So, this is indeed a generalization of the specifications for abelian extensions of $K$ to the morphisms $J_K^S\to G$.

We give a description of the ``specifications defined globally'' showing up in Hypothesis~\ref{hypothesis: comparison between field counting results} in terms of the definitions above.
Let $G$ be a transitive permutation group, and $\Omega$ be a nonempty subset of $G$ that is closed under conjugation and invertible powering.
Fix a number field $K$, define $\mathcal{E}$ as a subset of $\mathcal{E}(G;K)$.
Let $S$ be a finite set of primes of $K$ containing the ones at $\lvert G\rvert\infty$,
and $T\subseteq\mathcal{P}_K$ be a set of prime ideals.
Recall from Definition~\ref{def: specifications in general} that for each non-negative integer $\gamma$, we have defined the notation $\mathcal{E}_{\Omega,\gamma}^{T}$.
\begin{definition}\label{def: square-free ideals}
	Let $\mathfrak{a}$ be an integral ideal of $K$, i.e., $\mathfrak{a}\in I_K^+$.
	\begin{enumerate}
		\item We say that $\mathfrak{a}$ is square-free if $\mathfrak{p}^2\nmid\mathfrak{a}$ for each finite prime $\mathfrak{p}$, i.e., $\mu(\mathfrak{a})\neq0$ where $\mu$ is the M{\"o}bius function.
		\item Define $\omega(\mathfrak{a})$ to be the number of distinct prime factors of $\mathfrak{a}$.
		\item For a set $T$ of prime ideals, define
		\begin{equation*}
			I^{+,\mu}_T:=\{\mathfrak{a}\in I_K^+:\mu(\mathfrak{a})\neq0\text{ and }\mathfrak{p}\mid\mathfrak{a}\Rightarrow\mathfrak{p}\in T\}.
		\end{equation*}   
	\end{enumerate}
\end{definition}
Using these notations, we can express $\mathcal{E}_{\Omega,\gamma}^{T}$ in terms of local specifications.
\begin{proposition}\label{prop: translation between local and global specifications}
	Fix a non-negative integer $\gamma$.
	For each prime $\mathfrak{p}\in T$, define
	\begin{equation*}
		\Omega^{+}_{\mathfrak{p}}:=\{\Sigma_{\mathfrak{p}}\mid y_{\mathfrak{p}}\in\Omega\}
		\quad\text{and}\quad
		\Omega^{-}_{\mathfrak{p}}:=\{\Sigma_{\mathfrak{p}}\mid y_{\mathfrak{p}}\notin\Omega\},
	\end{equation*}
	where $y_{\mathfrak{p}}$ is a generator of the inertia subgroup of $\Sigma_{\mathfrak{p}}$.
	For each $\mathfrak{a}\in I^{+,\mu}_T$, define
	\begin{equation*}
		\Phi(\mathfrak{a}):=\prod_{\mathfrak{p}\mid\mathfrak{a}}\Omega^{+}_{\mathfrak{p}}\times\sideset{}{'}{\prod}_{
			\substack{
				\mathfrak{p}\in T\\
				\mathfrak{p}\nmid\mathfrak{a}
			}
		}\Omega^{-}_{\mathfrak{p}}.
	\end{equation*}
	We have
	\begin{equation}\label{eqn: translation between local and global specifications}
		\mathcal{E}_{\Omega,\gamma}^{T}=
		\bigsqcup_{
			\substack{
				\mathfrak{a}\in I^{+,\mu}_T\\
				\omega(\mathfrak{a})=\gamma
			}
		}\{L\in\mathcal{E}\mid L\in\Phi(\mathfrak{a})\}.
	\end{equation}	
\end{proposition}
\begin{proof}
	By definition, for any $G$-extension $L/K$, the inertia generator of $\mathfrak{p}$ is contained in $\Omega$ if and only if $L_{\mathfrak{p}}\in\Omega^{+}_{\mathfrak{p}}$.
	This implies that if $L$ is contained in one of the sets on the right-hand side of (\ref{eqn: translation between local and global specifications}), then it must be contained in $\mathcal{E}_{\Omega,\gamma}^{T}$.
	
	By definition, if $L\in\mathcal{E}_{\Omega,\gamma}^{T}$, then there are exactly $\gamma$ primes $\mathfrak{p}_1,\dots,\mathfrak{p}_{\gamma}$ such that $\mathfrak{p}_i\in T$ and that the inertia generator of $L_{\mathfrak{p}_i}/K_{\mathfrak{p}_i}$ (up to conjugate) is contained in $\Omega$.
	Define $\mathfrak{a}:=\mathfrak{p}_1\cdots\mathfrak{p}_{\gamma}$.
	And we see that $\mathfrak{a}\in I^{+,\mu}_T$, and $\omega(\mathfrak{a})=\gamma$, and $L\in\Phi(\mathfrak{a})$.
	
	It suffices to show that the union is disjoint.
        When $\gamma=0$, we see that the right-hand side becomes a single set.
	So we assume without loss of generality that $\gamma>0$ and let $\mathfrak{a}$ and $\mathfrak{b}$ be two different integrals of $ I^{+,\mu}_T$ with $\omega(\mathfrak{a})=\gamma=\omega(\mathfrak{b})$ such that
	\begin{equation*}
		\{L\in\mathcal{E}\mid L\in\Phi(\mathfrak{a})\}
		\quad\text{and}\quad
		\{L\in\mathcal{E}\mid L\in\Phi(\mathfrak{b})\}
	\end{equation*}
	are nonempty.
	Then first of all, there exists some $\mathfrak{p}\in T$ such that $\mathfrak{p}\mid\mathfrak{a}$ but $\mathfrak{p}\nmid\mathfrak{b}$.
	Second, if $L_1\in\Phi(\mathfrak{a})$ and $L_2\in\Phi(\mathfrak{b})$, then $L_{1,\mathfrak{p}}$ has inertia generator contained in $\Omega$ while the inertia generator of $L_{2,\mathfrak{p}}$ is contained in $G\backslash\Omega$.
	So $L_{1}\neq L_2$ belong to different isomorphism classes of field extensions over $K$.
	This shows that the union on the right-hand side of (\ref{eqn: translation between local and global specifications}) must be disjoint.
\end{proof}
Let us conclude this section by some applications of these notations.
\begin{lemma}\label{lemma: Hom(J_K,G)}
    Let $G$ be a finite abelian group.
    Fix a number field $K$ with a finite set $S$ of primes containing the ones at $\lvert G\rvert\infty$.
    We have that
	\begin{equation*}
		\operatorname{Hom}(J_K^S,G)\cong\bigoplus_{\mathfrak{p}\in S}\operatorname{Hom}(K_{\mathfrak{p}}^*,G)+\bigoplus_{\mathfrak{p}\notin S}\operatorname{Hom}(U_{\mathfrak{p}},G).
	\end{equation*}
\end{lemma}
\begin{proof}
	For topological groups $A$ and $B$, the notation $\operatorname{Hom}(A,B)$ always means the group of continuous homomorphisms.
	So for any (continuous) homomorphism $\chi:J_K^S\to G$, we see that the kernel $\ker\chi$ must be both open and closed.
	
	Claim: given local continuous homomorphisms $\chi_{\mathfrak{p}}:U_{\mathfrak{p}}\to G$ for all $\mathfrak{p}\notin S$ such that $\chi_{\mathfrak{p}}$ is trivial for all but finitely many primes, and continuous homomorphisms $\chi:K_{\mathfrak{p}}^*\to G$ for all $\mathfrak{p}\in S$, the product
	\begin{equation*}
		\chi:=\prod_{\mathfrak{p}\in S}\chi_{\mathfrak{p}}\times\prod_{\mathfrak{p}\notin S}\chi_{\mathfrak{p}}
	\end{equation*}
	is a morphism of topological groups.
	It suffices to show that $\ker\chi$ is open.
	For a fixed $\mathfrak{p}$, the kernel $\ker\chi_{\mathfrak{p}}$ is an open and closed subgroup of $U_{\mathfrak{p}}$ or $K_{\mathfrak{p}}^*$.
	And $\ker\chi_{\mathfrak{p}}=U_{\mathfrak{p}}$ for all but finitely many $\mathfrak{p}$.
	So if $x=(x_{\mathfrak{p}})_{\mathfrak{p}}\in J_K^S$ is contained in the kernel of $\chi$, then
	\begin{equation*}
		\prod_{\mathfrak{p}}x_{\mathfrak{p}}\ker\chi_{\mathfrak{p}}
	\end{equation*}
	is an open neighbourhood of $x$ in $J_K^S$ and is included in $\ker\chi$.
	This shows that $\ker\chi$ is open.
	
	Conversely, if $\chi:J_K^S\to G$ is continuous, then there exists a finite set of primes $S_0\subseteq\mathcal{P}(K)\backslash S$ such that
	\begin{equation*}
		\prod_{\mathfrak{p}\in S}W_{\mathfrak{p}}\times\prod_{\mathfrak{p}\in S_0}V_{\mathfrak{p}}\times\prod_{\mathfrak{p}\notin S\cup S_0}U_{\mathfrak{p}}\subseteq\ker\chi,
	\end{equation*}
	where $W_{\mathfrak{p}}$ is an open and closed neighbourhood of $1$ in $K_{\mathfrak{p}}^*$ and $V_{\mathfrak{p}}$ is an open and closed neighbourhood of $1$ in $U_{\mathfrak{p}}$.
	This implies that $\chi$ factors through the quotient $\prod_{\mathfrak{p}\in S}K_{\mathfrak{p}}^*\times\prod_{\mathfrak{p}\in T}U_{\mathfrak{p}}$, which reduces to the finite product of the morphisms $\chi_{\mathfrak{p}}$ where $\mathfrak{p}$ runs over primes in $S\cup S_0$.
	So $\chi$ is actually an element of the restricted product of $\operatorname{Hom}(K_{\mathfrak{p}}^*,G)$ and $\operatorname{Hom}(U_{\mathfrak{p}},G)$ with respect to the trivial map, where $\mathfrak{p}$ runs over all primes of $K$,
	i.e., the direct sum.
	And we are done for the proof.
\end{proof}
Based on the structure of $\operatorname{Hom}(J_K^S,G)$, we abstractly obtain a generating series as follows.
\begin{proposition}\label{prop: unique factorization of ideals}
    Let $C$ be a counting function defined as in Definition~\ref{def: product of ramified primes}, and $S$ a finite set of primes containing the ones at $\lvert G\rvert\infty$.
    Define $f$ to be a complex-valued function defined on morphisms $\chi:J_K^S\to G$ such that $f(\chi)=1$ if $\chi$ is trivial and
	\begin{equation*}
		f(\chi)=\prod_{\mathfrak{p}}f(\chi_{\mathfrak{p}}).
	\end{equation*}
    Let $T$ be a set of primes.
    For each $\mathfrak{p}\in T$, let $\Phi(\mathfrak{p})$ be a set of local specifications at $\mathfrak{p}$.
    Define $\Phi:=\prod'_{\mathfrak{p}\in T}\Phi(\mathfrak{p})$ as the restricted product with respect to unramified specifications.
    If $\Phi$ is defined by inertia generators, then we have
	\begin{equation*}
		\sum_{
			\substack{
				\chi:J_K^S\to G\\
				\chi\in\Phi
			}
		}f(\chi)C(\chi)^{-s}
		=\prod_{\mathfrak{p}\in T}
		\Bigl(
		\sum_{
			\substack{
				\chi_{\mathfrak{p}}:U_{\mathfrak{p}}\to G\\
				\chi\in\Phi(\mathfrak{p})
			}
		}
		f(\chi_{\mathfrak{p}})(\mathfrak{Np})^{-c(\chi_{\mathfrak{p}})s}
		\Bigr)
		\prod_{\mathfrak{p}\notin T}
		\Bigl(
		\sum_{
			\substack{
				\chi_{\mathfrak{p}}:W_{\mathfrak{p}}\to G\\
				\chi\notin\Phi(\mathfrak{p})
			}
		}f(\chi_{\mathfrak{p}})\mathfrak{Np}^{-c(\chi_{\mathfrak{p}})s}
		\Bigr)
		,
	\end{equation*}
	where $W_{\mathfrak{p}}=U_{\mathfrak{p}}$ if $\mathfrak{p}\notin S$ and $W_{\mathfrak{p}}=K_{\mathfrak{p}}^*$ otherwise.
\end{proposition}
\begin{proof}
	Since the identity holds in the sense of the Riemann sum, by Lemma~\ref{lemma: Hom(J_K,G)}, this reduces to the fact that $C$ and $f$ are multiplicative with respect to the group of ideals of $K$.
\end{proof}
This statement also explains why we prefer the morphisms $\chi:J_K^S\to G$.
In summary, we give a definition of local specification in terms of $G$-structured algebra and apply it to the morphisms $\chi:J_K^S\to G$ so that we can construct generating series for them.
In the next section, we are going to study the analytic properties of such series.

\section{Analytic continuation for generating series}\label{sec: generating series}
In this section we present some analytic results related to the generating series used in this paper.
Throughout the section fix a number field $K$.
Let us introduce the notion called ray class group.
See also Neukirch~\cite[VI.1.7]{neukirch2013algebraic}.
Let $\mathfrak{m}$ be a module (an integral ideal) of $K$.
Define $I_K^{\mathfrak{m}}$ as the fractional ideals that are relatively prime to $\mathfrak{m}$, and $P_K^{\mathfrak{m}}$ as the group of all principal ideals $(a)\in P_K$ such that
\begin{equation*}
	a\text{ totally positive}\quad\text{and}\quad a\equiv1\bmod{\mathfrak{m}}.
\end{equation*}
The latter condition means that there exist integers $b,c\in\mathscr{O}_K^*$ that are relatively prime to $\mathfrak{m}$ such that $a=b/c$ and $b\equiv c\bmod{\mathfrak{m}}$.
The ray class group mod $\mathfrak{m}$, denoted by $\operatorname{Cl}_K^{\mathfrak{m}}$, is the quotient $I_K^{\mathfrak{m}}/P_K^{\mathfrak{m}}$.
Let $h_{\mathfrak{m}}:=\lvert\operatorname{Cl}_K^{\mathfrak{m}}\rvert$ be the size of the ray class group.

Using the language of id{\`e}les, we can give an equivalent definition for the ray class group $\operatorname{Cl}_K^{\mathfrak{m}}$.
For each prime $\mathfrak{p}$, we define the higher unit groups as follows.
First, set $U_{\mathfrak{p}}^{(0)}:=U_{\mathfrak{p}}$ (see also the beginning of Section~\ref{sec: specifications}).
Then, for each positive integer $n$, define
\begin{equation*}
	U_{\mathfrak{p}}^{(n)}:=\left\{\begin{aligned}
		&1+\mathfrak{p}^n\quad&\text{if }\mathfrak{p}\text{ is finite};
		\\
		&\mathbb{R}_+^*\quad&\text{if }\mathfrak{p}\text{ is infinite real};
		\\
		&\mathbb{C}^*\quad&\text{if }\mathfrak{p}\text{ is infinite complex}.
	\end{aligned}\right.
\end{equation*}
Write $\mathfrak{m}=\prod_{\mathfrak{p}}\mathfrak{p}^{n_{\mathfrak{p}}}$ as a replete ideal, where $\mathfrak{p}$ runs over all primes of $K$.
In particular, $n_{\mathfrak{p}}=0$ for all $\mathfrak{p}\mid\infty$.
Define $J_K^{\mathfrak{m}}:=\prod_{\mathfrak{p}}U_{\mathfrak{p}}^{(n_{\mathfrak{p}})}$.
The subgroup $\operatorname{C}_K^{\mathfrak{m}}:=J_K^{\mathfrak{m}}K^*/K^*$ of the id{\`e}les class group $\operatorname{C}_K$ is called the \emph{congruence subgroup} $\bmod{\mathfrak{m}}$.
The map
\begin{equation*}
	J_K\to I_K,\quad\alpha\mapsto(\alpha)=\prod_{\mathfrak{p}\nmid\infty}\mathfrak{p}^{v_{\mathfrak{p}}(\alpha_{\mathfrak{p}})}
\end{equation*}
induces an isomorphism $\operatorname{C}_K/\operatorname{C}_K^{\mathfrak{m}}\to\operatorname{Cl}_K^{\mathfrak{m}}$.
See Neukirch~\cite[VI.1.9]{neukirch2013algebraic} for details.
\begin{lemma}\label{lemma: l(1,s)}
    Let $\mathfrak{m}$ be a module of $K$, and $\mathfrak{R}$ be a set of ideals representing an element of the ray class group $\operatorname{Cl}_K^{\mathfrak{m}}$.
    Define
	\begin{equation*}
		l_1(\mathfrak{R},s):=\sum_{\mathfrak{p}\in\mathcal{P}_{\mathfrak{R}}}\frac{1}{(\mathfrak{N}\mathfrak{p})^s},
	\end{equation*}
	where $\mathcal{P}_{\mathfrak{R}}$ means the primes in the class $\mathfrak{R}$.
	The series $l_1(\mathfrak{R},s)$ defines a holomorphic function in the open half-plane $\sigma>1$.
	It could be extended analytically to the line $\sigma=1$ except at $s=1$.
	Moreover, there exists a holomorphic function $f_{\mathfrak{R}}(s)$ at $s=1$ such that
	\begin{equation*}
		l_1(\mathfrak{R},s)=f_{\mathfrak{R}}(s)+\frac{1}{h_{\mathfrak{m}}}\log\frac{1}{s-1}
	\end{equation*}
	near $s=1$.
	Here $h_{\mathfrak{m}}$ is the size of the ray class group $\operatorname{Cl}_K^{\mathfrak{m}}$.
\end{lemma}
\begin{proof}
	Let $\chi:\operatorname{Cl}_K^{\mathfrak{m}}\to S^1$ be a character.
	Define
	\begin{equation*}
		L(\chi,s):=\sum_{\mathfrak{a}\in I^{\mathfrak{m}}_K}\frac{\chi(\mathfrak{a})}{(\mathfrak{N}\mathfrak{a})^s}
	\end{equation*}
	as the corresponding Hecke $L$-series.
	When $\chi$ has conductor other than $1$, it actually admits a holomorphic continuation to the whole plane.
	Moreover, the Hecke $L$-series $L(\chi,s)$ admits a zero-free region in a neighbourhood $U$ of $\sigma\geq1$ with at most one real and simple zero inside $U$ (the so-called exceptional zero).
	See Coleman~\cite{Coleman_1990} for example.
	On the other hand, Neukirch~\cite[VII.13.3]{neukirch2013algebraic} shows that if $\chi$ is a nontrivial (irreducible) character of $\operatorname{Cl}_K^{\mathfrak{m}}=I_K^{\mathfrak{m}}/P_K^{\mathfrak{m}}$, then the $L$-series $L(\chi,s)$ satisfies
	\begin{equation*}
		L(\chi,1)\neq0.
	\end{equation*}
	So we could simply say that there exists a neighbourhood $V$ of $\sigma\geq1$ (dependent on $K$ and $\chi$) such that $L(\chi,s)$ is nonzero in $V$.
	Since $\operatorname{Cl}_K^{\mathfrak{m}}$ is a finite abelian group, there are only finitely many characters, so we can take such a neighbourhood $V$ that works for all $\chi$.
    By Definition~\ref{def:set of primes}, let $\mathcal{P}_{I_K^{\mathfrak{m}}}$ be the set of primes in $I_K^{\mathfrak{m}}$.
	For each ideal class $\mathfrak{R}\in\operatorname{Cl}_K^{\mathfrak{m}}$, denote the set of primes in $\mathfrak{R}$ by $\mathcal{P}_{\mathfrak{R}}$.
	In the open half-plane, we have
	\begin{equation*}
		\log L(\chi,s)=\sum_{\mathfrak{p}\in\mathcal{P}_{I_K^{\mathfrak{m}}}}\sum_{n=1}^\infty\frac{1}{n}\chi(\mathfrak{p})^n(\mathfrak{N}\mathfrak{p})^{-ns}.
	\end{equation*}
	Note that for each $\delta>0$, it is absolutely and uniformly convergent in the half-plane $\sigma>1+\delta$.
	So given any class $\mathfrak{R}\in\operatorname{Cl}_K^{\mathfrak{m}}$, we have
	\begin{equation*}
		\begin{aligned}
			\sum_{\chi}\bar{\chi}(\mathfrak{R})\log L(\chi,s)
			=&\sum_{\mathfrak{p}}\sum_{n=1}^\infty\frac{1}{n}\sum_{\chi}\bar{\chi}(\mathfrak{R})\chi(\mathfrak{p})^n(\mathfrak{N}
			\mathfrak{p})^{-ns}
			\\
			=&h_{\mathfrak{m}}\sum_{\mathfrak{p}\in\mathcal{P}_{\mathfrak{R}}}\frac{1}{(\mathfrak{N}\mathfrak{p})^s}
			+\sum_{n=2}^\infty\frac{h_{\mathfrak{m}}}{n}
			\sum_{
				\substack{
					\mathfrak{p}\in\mathcal{P}_{I_K^{\mathfrak{m}}}\\
					\chi(\mathfrak{p}^n)=\chi(\mathfrak{R})
				}
			}\frac{1}{(\mathfrak{N}\mathfrak{p})^{ns}}
			\\
			=:&h_{\mathfrak{m}}(l_1(\mathfrak{R},s)+f(s)).
		\end{aligned}
	\end{equation*}
	By taking the absolute value, we see that $f(s)$ is holomorphic in the closed half-plane $\sigma\geq1$.
	The identity already gives the holomorphic continuation of $l_1(\mathfrak{R},s)$ to the line $\sigma=1$ except at $s=1$, because $L(\chi,s)$ are all holomorphic functions in the open half-plane $\sigma>1$ that could be extended analytically to the line $\sigma=1$ except at $s=1$.
	And they are nonzero in some neighbourhood $V$ of the closed half-plane $\sigma\geq1$, as mentioned above.
        Since $V$ and $V\backslash\{s=1\}$ could be chosen to be simply connected, we see that the notation $\log L(\chi,s)$ also makes sense in the corresponding region.
	
	The behaviour of $l_1(\mathfrak{R},s)$ at $s=1$ is given by the pole behaviour of $L(\chi,s)$.
	To be precise, if $\chi$ has conductor other than $1$, then $L(\chi,s)$ is holomorphic at $s=1$ with $L(\chi,1)\neq0$.
	Else if $\chi_0$ is the principal character, then $L(\chi_0,s)$ differs $\zeta_K(s)$ by a finite number of factors (in the sense of Euler form).
	So there exists some holomorphic function $g(s)$ at $s=1$ such that
	\begin{equation*}
		\log L(\chi_0,s)=\log\frac{1}{s-1}+g(s).
	\end{equation*}
	And this implies the identity near $s=1$:
	\begin{equation*}
		\begin{aligned}
			l_1(\mathfrak{R},s)
			=&\frac{1}{h_{\mathfrak{m}}}\log\frac{1}{s-1}+\frac{1}{h_{\mathfrak{m}}}g(s)-f(s)+\frac{1}{h_{\mathfrak{m}}}\sum_{\chi\neq\chi_0}L(\chi,s)
			\\
			=:&\frac{1}{h_{\mathfrak{m}}}\log\frac{1}{s-1}+f_{\mathfrak{R}}(s).
		\end{aligned}
	\end{equation*}
\end{proof}
Then we generalize the lemma in the following way.
Recall that for each integral ideal $\mathfrak{a}=\mathfrak{p}_1^{r_1}\cdots\mathfrak{p}_l^{r_l}$ of $K$, where $r_i$ is a positive integer for $i=1,\dots,l$, denote the number of distinct prime factors by $\omega(\mathfrak{a}):=l$.
We call $\mathfrak{a}$ square-free if $r_1=r_2=\cdots=r_l=1$.
Recall that for a class of ideals $\mathfrak{R}$ (e.g., a subset of the ray class group), denote the prime ideals in $\mathfrak{R}$ by $\mathcal{P}_{\mathfrak{R}}$.
\begin{lemma}\label{lemma: l(r,s)}
	Let $k$ be a fixed number field, and $\mathfrak{R}$ be a set of ideals.
	Write $I^{+,\mu}_{\mathfrak{R}}:= I^{+,\mu}_{\mathcal{P}_{\mathfrak{R}}}$ (see Definition~\ref{def: square-free ideals}).
	If the series
	\begin{equation*}
		l_{1}(\mathfrak{R},s):=\sum_{\mathfrak{p}\in\mathcal{P}_{\mathfrak{R}}}(\mathfrak{N}\mathfrak{p})^{-s}
	\end{equation*}
	defines a holomorphic function in the open half-plane $\sigma>\sigma_0$, where $\sigma_0$ is a fixed positive real number,
	and it could be extended analytically to the line $\sigma=\sigma_0$ except at $s=\sigma_0$ with an identity
	\begin{equation*}
		l_1(\mathfrak{R},s)=f(s)+f_0(s)\log\frac{1}{s-\sigma_0},
	\end{equation*}
	near $s=\sigma_0$, where $f(s),f_0(s)$ are holomorphic at $s=\sigma_0$ with $f_0(\sigma_0)\neq0$,
	then for each non-negative integer $\gamma$, the series
	\begin{equation*}
		l_{\gamma}(\mathfrak{R},s):=
		\sum_{
			\substack{
				\mathfrak{a}\in I^{+,\mu}_{\mathfrak{R}}\\
				\omega(\mathfrak{a})=\gamma
			}
		}\prod_{\mathfrak{p}\mid\mathfrak{a}}(\mathfrak{N}\mathfrak{p})^{-s}
	\end{equation*}
	also defines a holomorphic function in the open half-plane $\sigma>\sigma_0$ that could be extended to the line $\sigma=\sigma_0$ except at $s=\sigma_0$.
	And there exists holomorphic functions $g_0(s)$, $\dots$, $g_{\gamma}(s)$ at $s=\sigma_0$ such that
	\begin{equation*}
		l_{\gamma}(\mathfrak{R},s)=\sum_{i=0}^{\gamma}g_i(s)(\log\frac{1}{s-\sigma_0})^i
	\end{equation*}
	near $s=\sigma_0$.
\end{lemma}
\begin{proof}
	We prove it by induction on $\gamma$.
	The statement is true when $\gamma=1$ by the condition.
	Now assume that it is true for the series $l_{1}(\mathfrak{R},s),\dots,l_{\gamma-1}(\mathfrak{R},s)$, where $\gamma\geq2$.
	One can check the following identity:
	\begin{equation*}
		l_1(\mathfrak{R},s)\cdot l_{\gamma-1}(\mathfrak{R},s)=\gamma l_{\gamma}(\mathfrak{R},s)+\sum_{j=2}^{\gamma}(-1)^{j}l_1(\mathfrak{R},js)l_{\gamma-j}(\mathfrak{R},s),
	\end{equation*}
	where $l_0(s):=1$ as a notation.
	Then the analytic properties of $l_{\gamma}(\mathfrak{R},s)$ follows from the induction assumption.
\end{proof}
Based on these two lemmas, we can further give a result on the analytic property of Euler product
\begin{equation}\label{eqn: F(R,s)}
	F(\mathfrak{R},s):=\prod_{\mathfrak{p}\in\mathfrak{R}}(1-(\mathfrak{N}\mathfrak{p})^{-s})^{-1}.
\end{equation}
\begin{proposition}\label{prop: F(R,s)}
	Let $\mathfrak{m}$ be a module of $K$, and $\mathfrak{R}$ be a set of ideals representing an ideal class in $\operatorname{Cl}_K^{\mathfrak{m}}$.
	Then the Euler product $F(\mathfrak{R},s)$ defined as in (\ref{eqn: F(R,s)}) gives a holomorphic function in the open half-plane $\sigma>1$ that could be extended analytically to the line $\sigma=1$ except at $s=1$.
	And there exists some holomorphic function $f_{\mathfrak{R}}(s)$ at $s=1$ such that
	\begin{equation*}
		F(\mathfrak{R},s)=f_{\mathfrak{R}}(s)(s-1)^{-1/h_{\mathfrak{m}}}
	\end{equation*}
	near $s=1$.
\end{proposition}
\begin{proof}
	First of all, by comparing with $\zeta_K(s)$, we see that $F(\mathfrak{R},s)$ is absolutely and uniformly convergent in $\sigma>1+\delta$ for any $\delta>0$.
	
	It is clear that $F(\mathfrak{R},s)$ is nonzero in the open half-plane $\sigma>1$.
	So let us consider the difference
	\begin{equation*}
		l_1(\mathfrak{R},s)-\log F(\mathfrak{R},s)=\sum_{\mathfrak{p}\in\mathfrak{R}}\sum_{n=2}^\infty\frac{1}{n}(\mathfrak{N}\mathfrak{p})^{-ns}.
	\end{equation*}
	This defines a holomorphic function in the half-plane $\sigma>1/2$.
	So we see that $F(\mathfrak{R},s)$ admits a holomorphic continuation to $\sigma\geq1$ except at $s=1$.
	In addition, we see that $F(\mathfrak{R},s)^{h_{\mathfrak{m}}}$ has a simple pole at $s=1$.
	Since $F(\mathfrak{R},s)$ is nonzero in a neighbourhood of $s=1$, the existence of $f_{\mathfrak{R}}(s)$ is guaranteed.
\end{proof}
A corollary is the following.
\begin{corollary}\label{cor: F(R,s)}
	Let $\mathfrak{m}$ be a module of $K$, and let $\mathfrak{R}$ be a set of ideals representing an ideal class in $\operatorname{Cl}_K^{\mathfrak{m}}$.
	Define
	\begin{equation*}
		G(\mathfrak{R},s):=\prod_{\mathfrak{p}\in\mathcal{P}_{\mathfrak{R}}}(1+n\mathfrak{Np}^{-s}),
	\end{equation*}
	where $n$ is a fixed positive integer.
	Then $G(\mathfrak{R},s)$ defines a holomorphic function in the open half-plane $\sigma>1$ that could be extended analytically to the line $\sigma=1$ except at $s=1$.
	And there exists a holomorphic function $g(s)$ at $s=1$ such that
	\begin{equation*}
		G(\mathfrak{R},s)=g(s)(s-1)^{-n/h_{\mathfrak{m}}}
	\end{equation*}
	near $s=1$.
\end{corollary}
\begin{proof}
	By comparing $G(\mathfrak{R},s)$ with $F(\mathfrak{R},s)^n$, we see that $G(\mathfrak{R},s)$ is holomorphic in the open half-plane.
	Since $G$ is non-zero in this region, we can take the logarithm and the difference
	\begin{equation*}
		E(s):=\log G(\mathfrak{R},s)-n\log F(\mathfrak{R},s)
	\end{equation*}
	is a holomorphic function in the closed half-plane $\sigma\geq1$.
	Therefore we have the identity
	\begin{equation*}
		G(\mathfrak{R},s)=\exp(E(s))\cdot F(\mathfrak{R},s)^n
	\end{equation*}
	which gives the analytic continuation of $G(\mathfrak{R},s)$ described in the statement.
\end{proof}

\section{Relative class group of abelian extensions}\label{section: abelian}
In this section, fix a number field $K$ and a finite abelian group $G$.
Let $\mathcal{E}=\mathcal{E}(G;K)$ be the set of abelian $G$-extensions over $K$.
Recall from Definition~\ref{def: product of ramified primes} that we have defined a counting function $C$ for $G$-extensions $L/K$.
In this section, we simply set $c_G(g)=1$ for all nonzero $g\in G$ and set $c_{\mathfrak{p}}(\Sigma_{\mathfrak{p}})=1$ if $\Sigma_{\mathfrak{p}}$ has nontrivial inertia for each $\mathfrak{p}\mid\lvert G\rvert\infty$.
In other words, the counting function $C$ is exactly the same as the product of ramified primes defined in Theorem~\ref{thm: S1 dist of rel class group abelian}.
The main result of this section is the following.
\begin{theorem}\label{thm: dist of rel class group abelian}
	Let $p$ be a rational prime and $l$ be a non-negative integer.
	If $p^{l+1}\mid\lvert G\rvert$, then for each non-negative integer $r$, we have that
	\begin{equation*}
		\mathbb{P}_{\mathcal{E},C}(\operatorname{rk}_p p^{l}\operatorname{Cl}(L/K)\leq r)=0.
	\end{equation*}
	In addition,
	\begin{equation*}
		\mathbb{E}_{\mathcal{E},C}(\lvert\operatorname{Hom}(p^{l}\operatorname{Cl}(L/K),C_p)\rvert)=+\infty.
	\end{equation*}
\end{theorem}
We view $G$ as a transitive permutation group by its multiplication on itself.
For a positive integer $n$, define
\begin{equation*}
	\Omega_{n}=\{g\in G\mid\gcd(\# g\cdot i)_{i=1}^{\lvert G\rvert}\equiv0\bmod{n}\}.
\end{equation*}
Since $G$ is finite abelian, we see that $\Omega_n=\{g\in G\mid r_g\equiv0\bmod{n}\}$, where $r_g$ is the order of $g$.

Let $S$ be a finite set of primes of $K$ containing all the ones at $\lvert G\rvert\infty$ such that $\operatorname{Cl}_K^S=1$.
Note that
\begin{equation*}
	\operatorname{Cl}_K^S\cong J_K/J_K^{S}K^*.
\end{equation*}
This is equivalent to saying that
\begin{equation*}
	J_K^S/K^S=J_K^SK^*/K^*\cong\operatorname{C}_K.
\end{equation*}
For the existence of such a set $S$, see Neukirch~\cite[VI.1.4]{neukirch2013algebraic}.
Recall the notion ray class group from Section~\ref{sec: generating series} (see also Neukirch~\cite[VI.1.7]{neukirch2013algebraic}).
Let $\mathfrak{m}$ be a module (integral ideal) of $K$, $\mathfrak{R}$ be a set of ideals that representing an ideal class in $\operatorname{Cl}_K^{\mathfrak{m}}$, and $\mathcal{P}_{\mathfrak{R}}$ be the primes contained in $\mathfrak{R}$.
For each non-negative integer $\gamma$, defome $\mathcal{E}_{n,\gamma}^{\mathfrak{R}}:=\mathcal{E}_{\Omega_n,\gamma}^{\mathcal{P}_{\mathfrak{R}}\backslash S}$.
See also Definition~\ref{def: specifications in general}.
In particular, if $\mathfrak{R}\mapsto P_K^{\mathfrak{m}}$ corresponds to the trivial ideal class in $\operatorname{Cl}_K^{\mathfrak{m}}$, we further abbreviate the notation $\mathcal{E}_{n,\gamma}^{P_K^{\mathfrak{m}}}$ as $\mathcal{E}_{n,\gamma}^{\mathfrak{m}}$.
The key step of the proof is the following.
\begin{proposition}\label{prop: hypothesis abelian}
	Write $m:=\lvert G\rvert$.
	There exists a module $\mathfrak{m}=\prod_{\mathfrak{p}}\mathfrak{p}^{n_{\mathfrak{p}}}$ of $K$ such that for each $(a)\in P_K^{\mathfrak{m}}$ with the generator $a$ being totally positive, we have $\operatorname{Nm}_{K/\mathbb{Q}}(a)\equiv1\bmod{m}$.
	For each such module $\mathfrak{m}$, and for each rational prime $p$ such that $p^{l+1}\mid m$ with $l$ a non-negative integer, the Hypothesis~\ref{hypothesis: comparison between field counting results}(2) holds for $(\mathcal{E},\Omega_{p^{l+1}},P_K^{\mathfrak{m}})$.
	To be precise, for each non-negative integer $\gamma$, we have
	\begin{equation*}
		N_{C}(\mathcal{E}_{p^{l+1},\gamma}^{\mathfrak{m}},X)=o(N_{C}(\mathcal{E},X)),
	\end{equation*}
	as $X\to\infty$.
\end{proposition}
We'll prove it later.
Provided that the proposition is true, we first prove the theorem.
\begin{proof}[Proof of Theorem~\ref{thm: dist of rel class group abelian}]
	Recall that we write $\lvert G\rvert=m$ and $p^{l+1}\mid m$.
	So $\Omega_{p^{l+1}}$ is nonempty.
	And $P_K^{\mathfrak{m}}$ is included in $P_K$ as a subset, which is just the definition.
	And the Hypothesis~\ref{hypothesis: comparison between field counting results}(2) holds for $(\mathcal{E},\Omega_{p^{l+1}},P_K^{\mathfrak{m}})$.
	So all the conditions of Theorem~\ref{thm: 0-prob and infty-moment} are satisfied.
	And we get the result directly.
\end{proof}
In the rest of this section, let us prove Proposition~\ref{prop: hypothesis abelian}.
We first cite a result on field-counting.
\begin{lemma}\label{lemma: field-counting of abelian extensions in general}
	For $g\in G$, let $r_g$ be its order.
	There exists some positive constant $a_{\mathcal{E}}>0$ such that
	\begin{equation*}
		N_{C}(\mathcal{E},X)\sim a_{\mathcal{E}}X(\log X)^{\beta-1},
	\end{equation*}
	where
	\begin{equation*}
		\beta=\sum_{1\neq g\in G}[K(\zeta_{r_g}):K].
	\end{equation*}
\end{lemma}
\begin{proof}
	This is the special case of Wood~\cite[Lemma 2.9]{wood2010probabilities} when $C$ is the product of ramified primes.
	One can apply Tauberian Theorem~\cite[Appendix II Theorem I]{narkiewicz2014elementary} to obtain the result on field-counting directly.
	See also Delange~\cite{Delange54} for more details on Delange's Tauberian Theorem.
\end{proof}
Then we show that the module $\mathfrak{m}$ exists as claimed in the proposition.
\begin{lemma}\label{lemma: a suitable module of K}
	Let $M/L$ be a finite field extension of number fields, and $\mathfrak{l}$ be a module of $L$.
	There exists some module $\mathfrak{m}$ of $M$ such that the norm map induces a well-defined homomorphism
	\begin{equation*}
		\operatorname{Nm}_{M/L}:\operatorname{Cl}_M^{\mathfrak{m}}\to\operatorname{Cl}_{L}^{\mathfrak{l}}.
	\end{equation*}
\end{lemma}
\begin{proof}
	For any finite field extension $M/L$ of number fields, the norm map
	\begin{equation*}
		\operatorname{Nm}_{M/L}:\operatorname{C}_M\to\operatorname{C}_L
	\end{equation*}
	is continuous.
	The congruence subgroup $\operatorname{C}_{L}^{\mathfrak{l}}$ is a closed subgroup of finite index of $\operatorname{C}_{L}$, so the preimage $\operatorname{Nm}_{M/L}^{-1}(\operatorname{C}_{L}^{\mathfrak{l}})$ is a closed subgroup of finite index of $\operatorname{C}_M$.
	Since closed subgroups of finite index of $\operatorname{C}_M$ are precisely those subgroups that contain a congruence subgroup (see Neukirch~\cite[VI.1.8]{neukirch2013algebraic}), we see that there must be a module $\mathfrak{m}$ of $M$ such that
	\begin{equation*}
		\operatorname{C}_M^{\mathfrak{m}}\subseteq\operatorname{Nm}_{M/L}^{-1}(\operatorname{C}_{L}^{\mathfrak{l}})\subseteq\operatorname{C}_M.
	\end{equation*}
	And this shows that $\operatorname{Nm}_{M/L}:\operatorname{C}_M\to\operatorname{C}_{L}$ induces a map
	\begin{equation*}
		\operatorname{Nm}_{M/L}:\operatorname{Cl}_M^{\mathfrak{m}}=\operatorname{C}_M/\operatorname{C}_M^{\mathfrak{m}}\to\operatorname{Cl}_{L}^{\mathfrak{l}}=\operatorname{C}_{L}/\operatorname{C}_{L}^{\mathfrak{l}},
	\end{equation*}
	and we are done.
\end{proof}
\begin{remark}
	\begin{enumerate}
		\item Recall that $\operatorname{Cl}_{\mathbb{Q}}^m\cong(\mathbb{Z}/m\mathbb{Z})^*$ via the map $(a)\mapsto \bar{a}$, where $a$ is the positive generator of the principal ideal $(a)$, and $\bar{a}$ is the class modulo $m$.
		Therefore, if the lemma is true, then for any $m>0$, we can find a module $\mathfrak{m}$ of $K$ such that
		\begin{equation*}
			\operatorname{Nm}_{K/\mathbb{Q}}P_{K}^{\mathfrak{m}}\subseteq P_{\mathbb{Q}}^m.
		\end{equation*}
		\item From the proof we see that if $\mathfrak{m}$ satisfies the condition that $\operatorname{Nm}_{K/\mathbb{Q}}P_{K}^{\mathfrak{m}}\subseteq P_{\mathbb{Q}}^m$, then any module $\mathfrak{n}$ such that $\mathfrak{m}\mid\mathfrak{n}$ will do.
	\end{enumerate}	
\end{remark}
Finally we prove an estimate for $N_{C}(\mathcal{E}_{p^{l+1},\gamma}^{\mathfrak{m}},X)$.
\begin{proposition}\label{prop: field-counting of abelian extensions in general}
	Write $m:=\lvert G\rvert$ and $m':=\lvert G\backslash\Omega_{p^{l+1}}\rvert$.
	Let $\mathfrak{m}$ be a module of $K$ such that $\operatorname{Nm}_{K/\mathbb{Q}}P_K^{\mathfrak{m}}\subseteq P_{\mathbb{Q}}^m$.
	Then
	\begin{equation*}
		N_{C}(\mathcal{E}_{p^{l+1},\gamma}^{\mathfrak{m}},X)\ll X(\log X)^{\beta-1-\frac{m-m'}{h_{\mathfrak{m}}}}(\log\log X)^{\gamma},
	\end{equation*}
	where
	\begin{equation*}
		\beta=\sum_{1\neq g\in G}[K(\zeta_{r_g}):K]^{-1}.
	\end{equation*}
\end{proposition}
\begin{proof}
    Let $\mathcal{P}_{\mathfrak{m}}$ be the set of primes in $P_K^{\mathfrak{m}}$.
    To simplify the notation, define
    \begin{equation*}
        \mathcal{P}_{\mathfrak{m}}^+:=\mathcal{P}_{\mathfrak{m}}\backslash S\quad\text{ and }\mathcal{P}_{\mathfrak{m}}^-:=(\mathcal{P}_{\mathfrak{m}}\backslash S)^c=S\cup(\mathcal{P}_K\backslash\mathcal{P}_{\mathfrak{m}}).
    \end{equation*}
    And write $I^{+,\mu}_{\mathcal{P}_{\mathfrak{m}}^+}:=I^{+,\mu}_{\mathfrak{m}}$.
    See Definition~\ref{def: square-free ideals} for $I^{+,\mu}_{\mathcal{P}_{\mathfrak{m}}^+}$.
    
    \emph{Step 1}:	
	For each $\mathfrak{p}\in\mathcal{P}_{\mathfrak{m}}^+$, define
	\begin{equation*}
		\Omega_{\mathfrak{p}}^+:=\{\Sigma_{\mathfrak{p}}\mid y_{\mathfrak{p}}\in\Omega_q\}
		\quad\text{and}\quad
		\Omega_{\mathfrak{p}}^-:=\{\Sigma_{\mathfrak{p}}\mid y_{\mathfrak{p}}\notin\Omega_q\},
	\end{equation*}
	where $y_{\mathfrak{p}}$ is the inertia generator of $\Sigma_{\mathfrak{p}}$.
	Then Proposition~\ref{prop: translation between local and global specifications} says that
	\begin{equation*}
		\mathcal{E}^{\mathfrak{m}}_{q,\gamma}=
		\bigsqcup_{
			\substack{
				\mathfrak{a}\in I^{+,\mu}_{\mathfrak{m}}\\
				\omega(\mathfrak{a})=\gamma
			}
		}\{L\in\mathcal{E}\mid L\in\Phi(\mathfrak{a})\},
	\end{equation*}
	where
	\begin{equation*}
		\Phi(\mathfrak{a})=\prod_{\mathfrak{p}\mid\mathfrak{a}}\Omega_{\mathfrak{p}}^+
		\times
		\sideset{}{'}\prod_{
			\substack{
				\mathfrak{p}\in\mathcal{P}_{\mathfrak{m}}^+\\
				\mathfrak{p}\nmid\mathfrak{a}
			}
		}\Omega_{\mathfrak{p}}^-.
	\end{equation*}    
	For each $\mathfrak{p}\in\mathcal{P}_{\mathfrak{m}}^+$, we see that $\Omega_{\mathfrak{p}}^{\pm}$ are defined by inertia generators.
	So, for each $\mathfrak{a}\in I^{+,\mu}_{\mathfrak{m}}$, the set $\Phi(\mathfrak{a})$ is also defined by inertia generators,
	and for each $\chi:J_K^S\to G$, the notation $\chi\in\Phi(\mathfrak{a})$ is well-defined.
	Note also that $J_K^S/K^S\cong\operatorname{C}_K$ by our condition, so a surjective map $\operatorname{C}_K\to G$ is just a surjective map $\chi:J_K^S\to G$ that is trivial on $K^S$.
	In summary, for each integer $n>0$ and for each square-free integral ideal $\mathfrak{a}$, we have that
	\begin{equation*}
        \lvert\{L\in\mathcal{E}\mid L\in\Phi(\mathfrak{a})\text{ and }C(L)=n\}\rvert
		\leq
        \lvert\{\chi:J_K^S\to G\mid\chi\in\Phi(\mathfrak{a})\text{ and }C(\chi)=n\}\rvert.
	\end{equation*}
	So, by writing
	\begin{equation*}
        \begin{aligned}
            a_n:=&\lvert\{L\in\mathcal{E}_{p^{l+1},\gamma}^{\mathfrak{m}}\mid C(L)=n\}\rvert
            \\
            b_n:=&\Biggl\lvert
        \bigsqcup_{
			\substack{
				\mathfrak{a}\in I^{+,\mu}_{\mathfrak{m}}\\
				\omega(\mathfrak{a})=\gamma
			}
		}\{\chi:J_K^S\to G\mid \chi\in\Phi(\mathfrak{a})\text{ and }C(\chi)=n\}
        \Biggr\rvert,	
        \end{aligned}
	\end{equation*}
	we see that $a_n\leq b_n$ for each $n$.
	And it suffices to show that
	\begin{equation*}
		\sum_{n<X}b_n\ll X(\log X)^{\beta-1-\frac{m-m'}{h_{\mathfrak{m}}}}(\log\log X)^{\gamma}
	\end{equation*}
	as $X\to\infty$.
	
	\emph{Step 2}:
	Define
	\begin{equation*}
		\mathcal{F}_{q,\gamma}^{\mathfrak{m}}:=\{\chi:J_K^S\to G\mid\exists\mathfrak{a}\in I^{+,\mu}_{\mathfrak{m}}[\omega(\mathfrak{a})=\gamma\text{ and }\chi\in\Phi(\mathfrak{a})]\},
	\end{equation*}
	and define its generating series as
	\begin{equation*}
		F_{C,\mathcal{F}_{q,\gamma}^{\mathfrak{m}}}(s):=\sum_{n=1}^\infty b_n n^{-s}.
	\end{equation*}
	Note that
	\begin{equation*}
		\mathcal{F}_{q,\gamma}^{\mathfrak{m}}=
		\bigsqcup_{
			\substack{
				\mathfrak{a}\in I^{+,\mu}_{\mathfrak{m}}\\
				\omega(\mathfrak{a})=\gamma
			}
		}
		\{\chi:J_K^S\to G\mid \chi\in\Phi(\mathfrak{a})\}.
	\end{equation*}
	Using Proposition~\ref{prop: unique factorization of ideals}, we have the following computation:
	\begin{equation*}
		\begin{aligned}
			F_{C,\mathcal{F}_{q,\gamma}^{\mathfrak{m}}}(s)=
			&\sum_{n=1}^{\infty}
			b_n n^{-s}
			\\
			=&\sum_{n=1}^{\infty}
			n^{-s}\cdot
			\sum_{
				\substack{
					\mathfrak{a}\in I^{+,\mu}_{\mathfrak{m}}\\
					\omega(\mathfrak{a})=\gamma
				}
			}
			\sum_{
				\substack{
					\chi:J_K^S\to G\\
					\chi\in\Phi(\mathfrak{a})\\
					C(\chi)=n
				}
			}1
			\\
			=&\sum_{
				\substack{
					\mathfrak{a}\in I^{+,\mu}_{\mathfrak{m}}\\
					\omega(\mathfrak{a})=\gamma
				}
			}
			\sum_{
				\substack{
					\chi:J_K^S\to G\\
					\chi\in\Phi(\mathfrak{a})
				}
			}C(\chi)^{-s}
			\\
			=&\sum_{
				\substack{
					\mathfrak{a}\in I^{+,\mu}_{\mathfrak{m}}\\
					\omega(\mathfrak{a})=\gamma
				}
			}
			\prod_{\mathfrak{p}\mid\mathfrak{a}}\Bigl(
			\sum_{n=0}^\infty
			\sum_{
				\substack{
					\chi:U_{\mathfrak{p}}\to G\\
					\chi\in\Omega_{\mathfrak{p}}^{+}\\
					c(\chi)=n
				}
			}
			\mathfrak{Np}^{-ns}
			\Bigr)
			\prod_{
				\substack{
					\mathfrak{p}\in\mathcal{P}_{\mathfrak{m}}^+\\
					\mathfrak{p}\nmid\mathfrak{a}
				}
			}\Bigl(
			\sum_{n=0}^{\infty}
			\sum_{
				\substack{
					\chi:U_{\mathfrak{p}}\to G\\
					\chi\in\Omega^{-}_{\mathfrak{p}},c(\chi)=n					
				}
			}
			\mathfrak{Np}^{-ns}
			\Bigr)
			\\
			\cdot&\prod_{\mathfrak{p}\in\mathcal{P}_{\mathfrak{m}}^-}
			\Bigl(
			\sum_{n=0}^\infty
			\sum_{
				\substack{
					\chi:W_{\mathfrak{p}}\to G\\
					c(\chi)=n
				}
			}\mathfrak{Np}^{-ns}
			\Bigr)
			,
		\end{aligned}		
	\end{equation*}
	where $W_{\mathfrak{p}}=U_{\mathfrak{p}}$ if $\mathfrak{p}\notin S$ and $W_{\mathfrak{p}}=K_{\mathfrak{p}}^*$ otherwise.
	Recall that $\operatorname{Nm}_{K/\mathbb{Q}}P_K^{\mathfrak{m}}\subseteq P_{\mathbb{Q}}^{(m)}$.
	So for each $\mathfrak{p}\in\mathcal{P}_{\mathfrak{m}}^+$, the group $\operatorname{Hom}(U_{\mathfrak{p}},G)$ has size $m=\lvert G\rvert$.
	Also $C$ is just the product of ramified primes, so we can express $F_{C,\mathcal{F}_{q,\gamma}^{\mathfrak{m}}}(s)$ in a much simpler form:
	\begin{equation}\label{eqn: field-counting of abelian extensions in general 1}
		\begin{aligned}
			F_{C,\mathcal{F}_{q,\gamma}^{\mathfrak{m}}}(s)=
			&\sum_{\mathfrak{a}\in I^{+,\mu}_{\mathfrak{m}}}
			\prod_{\mathfrak{p}\mid\mathfrak{a}}\Bigl(
			(m-m')\mathfrak{Np}^{-s}			
			\Bigr)
			\prod_{
				\substack{
					\mathfrak{p}\in\mathcal{P}_{\mathfrak{m}}^+\\
					\mathfrak{p}\nmid\mathfrak{a}
				}
			}\Bigl(
			1+(m'-1)\mathfrak{Np}^{-s}
			\Bigr)
			\\
			\cdot&\prod_{\mathfrak{p}\in\mathcal{P}_{\mathfrak{m}}^-}
			\Bigl(
			\sum_{\chi:W_{\mathfrak{p}}\to G}\mathfrak{Np}^{-c(\chi)s}
			\Bigr).
		\end{aligned}		
	\end{equation}
	
	\emph{Step 3}:
	The expression above (\ref{eqn: field-counting of abelian extensions in general 1}) is still an infinite sum of Euler products.
	Let us modify it so that it becomes more manageable.
	Define
	\begin{equation}\label{eqn: field-counting of abelian extensions in general 2}
		\begin{aligned}
			\tilde{F}_{C,\mathcal{F}_{q,\gamma}^{\mathfrak{m}}}(s):=
			&
			\prod_{\mathfrak{p}\in\mathcal{P}_{\mathfrak{m}}^+}
			\Bigl(
			1+(m'-1)\mathfrak{Np}^{-s}
			\Bigr)
			\prod_{\mathfrak{p}\in\mathcal{P}_{\mathfrak{m}}^-}
			\Bigl(
			\sum_{\chi:W_{\mathfrak{p}}\to G}\mathfrak{Np}^{-c(\chi)s}
			\Bigr)
			\\
			\cdot&
			\Bigl(
			\sum_{
				\substack{
					\mathfrak{a}\in I^{+,\mu}_{\mathfrak{m}}\\
					\omega(\mathfrak{a})=\gamma
				}
			}
			\prod_{\mathfrak{p}\mid\mathfrak{a}}\bigl(
			(m-m')\mathfrak{Np}^{-s}			
			\bigr)
			\Bigr).
		\end{aligned}
	\end{equation}
	By expanding it in the sense of the Riemann sum, that is, write $\tilde{F}_{C,\mathcal{F}_{q,\gamma}^{\mathfrak{m}}}(s)=\sum_{n=1}^\infty c_nn^{-s}$, for any $n>0$ we see that
	\begin{equation*}
		b_n\leq c_n.
	\end{equation*}
	Because if we compare (\ref{eqn: field-counting of abelian extensions in general 2}) with (\ref{eqn: field-counting of abelian extensions in general 1}), each Euler product in the infinite sum has more factors.
	
	\emph{Step 4:}
	Since $\tilde{F}_{C,\mathcal{F}_{q,\gamma}^{\mathfrak{m}}}(s)$ is absolutely convergent in the half-plane $\sigma>1$ by comparing with a large enough power of Riemann zeta function,
	we see that in the same open half-plane, the expression~(\ref{eqn: field-counting of abelian extensions in general 2}) equals to the Dirichlet series $\sum_{n}c_nn^{-s}$ as holomorphic functions.
	Note that
	\begin{equation*}
		\prod_{\mathfrak{p}\in\mathcal{P}_{\mathfrak{m}}^-}\Bigl(
		\sum_{\chi:W_{\mathfrak{p}}\to G}\mathfrak{Np}^{-c(\chi)s}
		\Bigr)
		=\prod_{\mathfrak{p}\in\mathcal{P}_K}
		\Bigl(
		\sum_{\chi:W_{\mathfrak{p}}\to S}\mathfrak{Np}^{-c(\chi)s}
		\Bigr)
		\prod_{\mathfrak{p}\in\mathcal{P}_{\mathfrak{m}}^+}
		\Bigl(
		1+(m-1)\mathfrak{Np}^{-s}
		\Bigr)^{-1}.
	\end{equation*} 
	We can rewrite $\tilde{F}_{C,\mathcal{F}_{q,\gamma}^{\mathfrak{m}}}(s)$ as follows:
	\begin{equation*}
		\begin{aligned}
			\tilde{F}_{C,\mathcal{F}_{q,\gamma}^{\mathfrak{m}}}(s)
			=&\prod_{\mathfrak{p}\in\mathcal{P}_K}
			\Bigl(
			\sum_{\chi:W_{\mathfrak{p}}\to S}\mathfrak{Np}^{-c(\chi)s}
			\Bigr)
			\\
			&\prod_{\mathfrak{p}\in\mathcal{P}_{\mathfrak{m}}^+}
			\Bigl(
			1+(m-1)\mathfrak{Np}^{-s}
			\Bigr)^{-1}
			\prod_{\mathfrak{p}\in\mathcal{P}_{\mathfrak{m}}^+}
			\Bigl(
			1+(m'-1)\mathfrak{Np}^{-s}
			\Bigr)
			\\
			&\Bigl(\sum_{\mathfrak{a}\in I^{+,\mu}_{\mathfrak{m}}}
			\prod_{\mathfrak{p}\mid\mathfrak{a}}\bigl(
			(m-m')\mathfrak{Np}^{-s}			
			\bigr)\Bigr).
		\end{aligned}		
	\end{equation*}
	By Wood~\cite[Lemma 2.10]{wood2010probabilities}, we know that
	\begin{equation*}
		\prod_{\mathfrak{p}}
		\Bigl(
		\sum_{\chi:W_{\mathfrak{p}}\to S}\mathfrak{Np}^{-c(\chi)s}
		\Bigr)
	\end{equation*}
	admits an analytic continuation to the line $\sigma=1$ except at $s=1$, and it has a pole at $s=1$ of order
	\begin{equation*}
		\beta=\sum_{0\neq g\in G}[K(\zeta_{r_g}):K]^{-1},
	\end{equation*}
	where $r_g$ is the order of $g$.
	Using Corollary~\ref{cor: F(R,s)}, we see that for any positive integer $l$, there exists some holomorphic function $u(s)$ near $s=1$ such that
	\begin{equation}\label{eqn: field-counting of abelian extensions in general 3}
		\prod_{\mathfrak{p}\in\mathcal{P}_{\mathfrak{m}}^+}
		\Bigl(
		1+l\mathfrak{Np}^{-s}
		\Bigr)
		=u(s)(s-1)^{-\frac{l}{h_{\mathfrak{m}}}}
	\end{equation}
	near $s=1$,
	and it admits an analytic continuation to the line $\sigma=1$ except at $s=1$.
	By Lemma~\ref{lemma: l(r,s)}, there exist holomorphic functions $v_0(s),\dots,v_{\gamma}(s)$ such that
	\begin{equation*}
		\sum_{\mathfrak{a}\in I^{+,\mu}_{\mathfrak{m}}}
		\prod_{\mathfrak{p}\mid\mathfrak{a}}
		(m-m')\mathfrak{Np}^{-s}
		=\sum_{j=0}^{\gamma}v_j(s)\Bigl(\log\frac{1}{s-1}\Bigr)^j
	\end{equation*}
	near $s=1$,
	and it admits an analytic continuation to the line $\sigma=1$ except at $s=1$.
	Combining these three parts, we see that $\tilde{F}_{C,\mathcal{F}_{q,\gamma}^{\mathfrak{m}}}(s)$ admits an analytic continuation to the line $\sigma=1$ except at the point $s=1$.
	And there exist some holomorphic function $f_0(s),\dots,f_{\gamma}(s)$ near $s=1$ such that
	\begin{equation*}
		\tilde{F}_{C,\mathcal{F}_{q,\gamma}^{\mathfrak{m}}}(s)=\sum_{j=0}^l f_j(s)(s-1)^{-\beta+\frac{m-m'}{h_{\mathfrak{m}}}}\Bigl(\log\frac{1}{s-1}\Bigr)^j.
	\end{equation*}
	Then Theorem~\cite[Appendix II Theorem I]{narkiewicz2014elementary} implies that
	\begin{equation*}
		\sum_{n<X}c_n\ll X\Bigl(\log X\Bigr)^{\beta-1-\frac{m-m'}{h_{\mathfrak{m}}}}\Bigl(\log\frac{1}{s-1}\Bigr)^{\gamma}.
	\end{equation*}
	And we are done.
\end{proof}
Let us conclude this section by proving that the Hypothesis~\ref{hypothesis: comparison between field counting results}(2) is true in this case.
\begin{proof}[Proof of Proposition~\ref{prop: hypothesis abelian}]
	Recall that $m=\lvert G\rvert$.
	The existence of the module $\mathfrak{m}$ is proven by Lemma~\ref{lemma: a suitable module of K}.
	For such a module $\mathfrak{m}$, Proposition~\ref{prop: field-counting of abelian extensions in general} together with Lemma~\ref{lemma: field-counting of abelian extensions in general} says that
	\begin{equation*}
		N_{C}(\mathcal{E}_{p^{l+1},\gamma}^{\mathfrak{m}},X)=o(N_{C}(\mathcal{E},X)),
	\end{equation*}
	for all $\gamma\geq0$.
	So the proposition is true.
\end{proof}

\section{General case}\label{sec: general case}
In this section, we apply our results obtained in Section~\ref{section: estimate of class group} to fields that are not necessarily abelian over $\mathbb{Q}$.
We first give a set-up, including quite a few notations.
Fix $G$ to be a finite transitive permutation group such that $\operatorname{Stab}_G(1)=\{e_G\}$ where $e_G$ is the identity of $G$.
Define $\mathcal{E}:=\mathcal{E}(G;\mathbb{Q})$.
Note that $L\in\mathcal{E}$ if and only if $L$ is a Galois $G$-field by Definition~\ref{def:set of fields}.
Assume that $G$ admits the following short exact sequence of groups
\begin{equation}\label{eqn: definition of G}
	1\to H\to G\to \bar{G}\to1,
\end{equation}
where $H$ is a nontrivial finite abelian group, and $\bar{G}$ is a finite group, and $\gcd(\lvert H\rvert,\lvert\bar{G}\rvert)=1$.
We also treat $H$ as a $\bar{G}$-module according to this sequence.
That is, if $\bar{g}\in\bar{G}$ and $h\in H$, then
\begin{equation*}
	\bar{g}\cdot h=ghg^{-1}
\end{equation*}
where $g$ is a preimage of $\bar{g}$ in $G$ and $h$ is viewed as an element of $G$.
Since $H$ is abelian and is normal in $G$, the definition does not depend on the choice of the preimage $g$.

Let $q$ be a fixed prime such that $q\mid\lvert H\rvert$,
and $\Omega$ be a nonempty subset of $\{g\in H\mid r_g\equiv0\bmod{q}\}$ that is closed under invertible powering and conjugation, where $r_g$ is the order of $g$ in $G$.
Using Definition~\ref{def: specifications in general}, define the notation $\mathcal{E}_{\gamma}:=\mathcal{E}_{\Omega,\gamma}^{\mathcal{P}}$ for each non-negative integer $\gamma$.
It means that a field $L\in\mathcal{E}$ is contained in $\mathcal{E}_{\gamma}$ if and only if there are exactly $\gamma$ tamely ramified primes $p$ such that its inertia generator is contained in $\Omega$.

Fix a positive integer $m$.
Let $C$ be a counting function (Definition~\ref{def: product of ramified primes}) for $L\in\mathcal{E}$ satisfying the following conditions.
The maps $c_G$ and $\{c_p\}_{p}$ are all $\mathbb{Z}_{\geq0}$-valued.
For each $g\in G\backslash\{e_G\}$, we have $c_G(g)\geq m$.
Moreover, $c_G=m$ if and only if $g\in\Omega$.
We present the main results of this section as follows.
\begin{lemma}[Correction of the original Lemma 7.1]\label{lemma of general case}
	Assume that $\bar{G}$ is abelian.
	\begin{enumerate}[(i)]
		\item Let
		\begin{equation*}
			m':=\min_{g\notin\Omega\cup\{e_G\}}c_G(g).
		\end{equation*} 
		For each $0<\varepsilon<1$, we have
		\begin{equation*}
			N_C(\mathcal{E}_0,X)\ll X^{1/(m'-\varepsilon)},
		\end{equation*}
		as $X\to\infty$.
		In particular, $N_C(\mathcal{E}_0,X)=o(X^{1/(m+\varepsilon')})$ for each $0<\varepsilon'<1$.
		\item For each positive integer $\gamma$, we have
		\begin{equation*}
			N_C(\mathcal{E}_{\gamma},X)\asymp\frac{X^{1/m}}{\log X}(\log\log X)^{\gamma-1},
		\end{equation*}
		as $X\to\infty$.
		\item The Hypothesis~3.12(1) holds for $((\mathcal{E},C),\Omega,\mathcal{P})$.
		To be precise, for each non-negative integer $\gamma$, we have that
		\begin{equation*}
			N_C(\mathcal{E}_{\gamma},X)=o(N_C(\mathcal{E}_{\gamma+1},X)).
		\end{equation*}
	\end{enumerate}
\end{lemma}
We'll prove the lemma later.
\begin{theorem}\label{theorem of general case}
	For each non-negative integer $r$, we have
	\begin{equation*}
		\mathbb{P}_{C,\mathcal{E}}(\operatorname{rk}_q\operatorname{Cl}_L\leq r)=0.
	\end{equation*}
	Moreover, we have
	\begin{equation*}
		\mathbb{E}_{C,\mathcal{E}}(\lvert\operatorname{Hom}(\operatorname{Cl}_L,C_q)\rvert)=+\infty.
	\end{equation*}
\end{theorem}
\begin{proof}
	If we admit the above lemma for now, then the theorem just follows from Theorem~\ref{thm: 0-prob and infty-moment}.
\end{proof}
In the rest of this section, we'll prove the lemma by an estimate on field-counting.

\subsection{Algebraic Theory}
We need some algebraic preliminaries.
First of all, Class Field Theory tells us the following:
for a fixed number field $K$, there exists a one-to-one correspondence
\begin{equation*}
	\{L\text{ is an }H\text{-extension over }K\}\leftrightarrow\operatorname{Sur}(\operatorname{C}_K,H).
\end{equation*}
But we also require that $L$ is Galois over $\mathbb{Q}$.
So we need to study the Galois action in the extensions of id{\`e}les.
\begin{lemma}\label{lemma: CFT galois action on idele}
	Let $A$ be a finite group, and $M$ be a finite abelian group.
    Fix an $A$-field $K$.
    Let $L/K$ be an $M$-extension with the Artin reciprocity map
	\begin{equation*}
		\chi:\operatorname{C}_K\to M.
	\end{equation*}
	The field extension $L/\mathbb{Q}$ is also Galois if and only if the kernel $N_L:=\ker\chi$ is an $A$-submodule of $\operatorname{C}_K$.
    In addition, if $M$ is an $A$-module and $\chi$ is an $A$-morphism, then the action of $A$ on $M$ is the same as the one induced by the short exact sequence
    \begin{equation*}
        1\to M\to G(L/\mathbb{Q})\to A\to1,
    \end{equation*}
    i.e., $G(L/\mathbb{Q})$ is an element of $H^2(A,M)$.
\end{lemma}
\begin{proof}
	If $L/\mathbb{Q}$ is Galois, then $N_L=\operatorname{Nm}_{L/K}\operatorname{C}_L$ is a $A$-submodule of $\operatorname{C}_K$.
	This shows that the induced isomorphism $\operatorname{C}_K/N_L\cong M$ is an $A$-morphism, where the action of $A$ on $M$ is induced from its action on $\operatorname{C}_K$.
	
	Conversely, we can assume for contradiction that $L/\mathbb{Q}$ is not Galois.
	Since $K$ is Galois, we see that for any $g$ in the absolute Galois group $G_{\mathbb{Q}}$, we have that
	\begin{equation*}
		gK=K\text{ and }gK\subseteq gL.
	\end{equation*}
	This implies that if $L'$ is any conjugate of $L$, then $K\subseteq L\cap L'$.
	In particular, $N_L$ and $N_{L'}=\operatorname{Nm}_{L'/K}\operatorname{C}_{L'}$ are conjugate with each other in the $A$-module $\operatorname{C}_K$.
	In other words, $N_L$ cannot be $A$-invariant.
	So the contradiction implies that $L/\mathbb{Q}$ must be Galois.
	
	When $L/\mathbb{Q}$ is Galois, then clearly we have the short exact sequence
	\begin{equation*}
		1\to M\to G(L/\mathbb{Q})\to A\to1.
	\end{equation*}
	Since $M$ is abelian, we see that $A$ acts on $M$ by conjugation.
	We need to verify that this action coincides with the $A$-module $M$.
	By Neukirch~\cite[IV.5.8]{neukirch2013algebraic} (functorial behaviour of the reciprocity map), for an automorphism $\sigma$ of $L$, we have the following commutative diagram:
	\begin{equation*}
		\begin{tikzcd}
			G(L/K)\arrow[d,"\sigma^*"]\arrow[r,"\chi"]&\operatorname{C}_K/N_L\arrow[d,"\sigma"]\\
			G(L/K)\arrow[r,"\chi"]&\operatorname{C}_K/N_L
		\end{tikzcd}
	\end{equation*}
	where $\sigma^*$ means the group conjugation, and $\sigma$ on the right means the action induced by $\sigma$ acting on $\operatorname{C}_K$. 
	Note that $G(L/K)\cong\operatorname{C}_K/N_L$ via the Artin reciprocity map $\chi$.
	This commutative diagram shows that all the actions are compatible.
	And we are done.
\end{proof}
Note that there is a similar result for local extensions.
Let us write down the statement.
\begin{lemma}\label{lemma: CFT galois action on idele local case}
	Let $A$ be a finite group, and $M$ be a finite abelian group.
    Fix a finite rational prime $p$ and an $A$-extension $K/\mathbb{Q}_p$.
	Let $L/K$ be a finite abelian extension with the reciprocity map $\chi:K^*\to M$.
	The extension $L/\mathbb{Q}_p$ is Galois if and only if $N_L:=\ker\chi$ is $A$-invariant.
	In addition, if $L/\mathbb{Q}_p$ is Galois, then $G(L/\mathbb{Q})$ is an element of $H^2(A,M)$.
\end{lemma}
The proof is similar to that of Lemma~\ref{lemma: CFT galois action on idele} and we just leave it to the reader.
\begin{lemma}\label{lemma: H^2(C3,C2^2)}
	If $A$ is a finite group and $M$ a finite $A$-module such that $\gcd(\lvert A\rvert,\lvert M\rvert)=1$, then $M$ is cohomologically trivial, that is, for each subgroup $B\subseteq A$ and for each integer $n\in\mathbb{Z}_+$ we have 
	\begin{equation*}
		H^n(B,M)=0.
	\end{equation*}
\end{lemma}
\begin{proof}
	Since multiplying by $\lvert B\rvert$ is an automorphism of $M$ by our condition, by Neukirch and Schmidt and Wingberg~\cite[I.1.6.2]{neukirch2013cohomology}, we know that $H^n(B,M)$ is trivial for each positive integer $n$.
\end{proof}
Note that this implies that $G$ is the only group that fits into the short exact sequence~(\ref{eqn: definition of G}).
To be precise, we have the following.
\begin{proposition}\label{prop: C3-morphism}
	Fix a $\bar{G}$-field $K$.
	There is a one-to-one correspondence between the following two sets
	\begin{equation*}
		\operatorname{Sur}_{\bar{G}}(\operatorname{C}_K,H)\leftrightarrow\{L\in\mathcal{E}\mid K\subseteq L\}.
	\end{equation*}
\end{proposition}
\begin{proof}
	If $K$ is a $\bar{G}$-field and $\chi:\operatorname{C}_K\to H$ is a surjective $\bar{G}$-morphism, then the class field $L/K$ must be Galois by Lemma~\ref{lemma: CFT galois action on idele}.
	In addition, the Galois group $G(L/\mathbb{Q})$ is an element of $H^2(\bar{G},H)$.
	Since the cohomology group is trivial, we see that $G(L/\mathbb{Q})$ must be isomorphic to $G$.
	In other words, the $\bar{G}$-morphism $\chi$ corresponds to a field $L\in\mathcal{E}$.
	
	If $L\in\mathcal{E}$, then $L$ is a $G$-field, so $K:=L^{H}$ is a $\bar{G}$-field.
	This implies that there is an Artin reciprocity map $\chi:\operatorname{C}_K\to H$ that corresponds to $L/K$.
	By Lemma~\ref{lemma: CFT galois action on idele} again, the map $\chi$ must be a $\bar{G}$-morphism, and we are done.
\end{proof}
For a $\bar{G}$-field $K$, we know that $G$-fields above $K$ are given by surjective $\bar{G}$-morphisms $\chi:\operatorname{C}_K\to H$.
And of course, if $L\in\mathcal{E}$, then $L^{H}$ must be a $\bar{G}$-field.
Now we want to show how to compute the value $C(L)$ (see Definition~\ref{def: product of ramified primes}) following the chain of extensions $\mathbb{Q}\subseteq K\subseteq L$ in theory.
After all, the map $c_{p}$ is defined for specifications at $p$.
So we need to recover the information of specifications at $p$ from this chain and the Artin reciprocity map.
We give the statements in steps.
\begin{lemma}\label{lemma: c(Sigma) in general case 1}
	If we choose a splitting map $s:\bar{G}\to G$ and treat $\bar{G}$ also as a subgroup of $G$, 
	then for each subgroup $G'$ of $G$, there exists $g\in G$, and $H'\subseteq H$ and $\bar{G}'\subseteq\bar{G}$ such that $\bar{G}'$ acts on $H'$ via the action of $\bar{G}$ on $H$, and
	\begin{equation*}
		gG'g^{-1}=H'\bar{G}'.
	\end{equation*}
\end{lemma}
\begin{proof}
	Let $\pi:G\to\bar{G}$ be the projection map with the section $s:\bar{G}\to G$.
	Since $G'$ is a subgroup of $G$, there is a short exact sequence induced by $\pi$, i.e.,
	\begin{equation*}
		1\to H'\to G'\to\bar{G}'\to1,
	\end{equation*}
	where $\bar{G}'$ is the image of $G'$ under the map $\pi$, and $H'=G'\cap H$ is the kernel of this map.
	In other words $H'$ is a normal subgroup of $G'$ such that $H'$ and $G'/H'$ have coprime orders.
	By Schur-Zassenhaus Theorem, we see that $H'$ admits a complement in $G'$, i.e., there exists some subgroup $\bar{G}''$ such that $H'\bar{G}''=G'$.
	In particular, the restriction $\pi:\bar{G}''\to\bar{G}'$ induces a group isomorphism.
	This implies that $\bar{G}''\subseteq\pi^{-1}(\bar{G}')=H\bar{G}'$ and it is one of the complements of $H$ in $H\bar{G}'$, for
	\begin{equation*}
		H\cap\bar{G}''=\{e_G\}\quad\text{and}\quad\pi^{-1}(\bar{G}')=HG'=HH'\bar{G}''=H\bar{G}''.
	\end{equation*}
	By Schur-Zassenhaus theorem again, we see that the complements of $H$ in $H\bar{G}'$ are all conjugate to each other.
	So if $g\in H\bar{G}'$ is some element such that $g\bar{G}''g^{-1}=\bar{G}'$, then we have that
	\begin{equation*}
		gG'g^{-1}=gH'g^{-1}g\bar{G}''g^{-1}=H'\bar{G}'.
	\end{equation*}
	And we are done.
\end{proof}
Recall that a specification $\Sigma$ at $p$ is a $G$-structured algebra of the form $\operatorname{Ind}^G_{G'}F$ where $G'$ is a subgroup of $G$, and $F$ is a $G'$-extension of $\mathbb{Q}_p$.
Lemma~\ref{lemma: c(Sigma) in general case 1} actually shows that every specification at $p$ comes from two pieces of data:
1) a $\bar{G}'$-extension $K_{\mathfrak{p}}/\mathbb{Q}_p$; 
2) a $\bar{G}'$-morphism $K_{\mathfrak{p}}^*\to H$.
Because for each subgroup $G'\subseteq G$, the $G'$-extensions $F/\mathbb{Q}_p$ arise in this way (up to conjugation).
Let us make it precise by the following statement.
\begin{lemma}\label{lemma: c(Sigma) in general case 2}
    Fix a finite rational prime $p$ and a $\bar{G}$-field $K$.
	If $\mathfrak{p}$ is a prime of $K$ above $p$, and $G_{\mathfrak{p}}\subseteq\bar{G}$ is the decomposition group, then for each $G_{\mathfrak{p}}$-morphism $\chi_{\mathfrak{p}}:K_{\mathfrak{p}}^*\to H$ with image $H'$, the class field $F$ corresponding to $\chi_{\mathfrak{p}}$ is a $G'$-extension over $\mathbb{Q}_p$ with $G'=H'\rtimes G_{\mathfrak{p}}$.
	If $L\in\mathcal{E}$ is a field extension of $K$ that corresponds to the $\bar{G}$-morphism $\rho:\operatorname{C}_K\to H$, then $\rho_{\mathfrak{p}}$ must be $G_{\mathfrak{p}}$-equivariant.
	Moreover if $\rho_{\mathfrak{p}}=\chi_{\mathfrak{p}}$, then $G_{\mathfrak{P}}\cong G'$ and $L_{\mathfrak{P}}\cong F$ as $G_{\mathfrak{P}}$-extensions over $\mathbb{Q}_p$, and they induce the same local specification at $p$, i.e.,
	\begin{equation*}
		L_p\cong\operatorname{Ind}^G_{G_{\mathfrak{P}}}L_{\mathfrak{P}}\cong\operatorname{Ind}^G_{G_{\mathfrak{P}}}F.
	\end{equation*}
\end{lemma}
\begin{proof}
	By Lemma~\ref{lemma: CFT galois action on idele local case}, since the kernel of the surjective map $\chi_{\mathfrak{p}}:K_{\mathfrak{p}}^*\to H'$ is $G_{\mathfrak{p}}$-invariant, the class field $F$ is normal over $\mathbb{Q}_p$ with Galois group $G'$, defined as in the statement.
	
	Now assume that $L\in\mathcal{E}$ is a $G$-field that is also an extension of $K$ corresponding to the Artin reciprocity map $\rho$.
	By definition, we have
	\begin{equation*}
		L_p=L\otimes\mathbb{Q}_p
		=\prod_{\mathfrak{P}\mid p}L_{\mathfrak{P}}
		=\operatorname{Ind}^G_{G_{\mathfrak{P}}}L_{\mathfrak{P}},
	\end{equation*}
	where the last notation indicates the $G$-structure on $L_p$.
	Assume without loss of generality that $\mathfrak{p}=\mathfrak{P}\cap\mathscr{O}_K$, where $\mathscr{O}_K$ is the ring of integers.
	We see that $L_{\mathfrak{P}}/K_{\mathfrak{p}}$ is a finite abelian extension of which the corresponding Artin map is induced by $\rho_{\mathfrak{p}}:K_{\mathfrak{p}}^*\to H$.
	The map $\rho_{\mathfrak{p}}$ must be consistent with the $G_{\mathfrak{p}}$-action.
	Because $\rho_p:K_p^*\to H$ is $G$-equivariant, and the action of $G_{\mathfrak{p}}$ does not permute the place $\mathfrak{p}$, hence inducing an automorphism on $K_{\mathfrak{p}}^*$.
	
	If in addition, $\rho_{\mathfrak{p}}=\chi_{\mathfrak{p}}$, then $G_{\mathfrak{P}}\cong G'$ and $L_{\mathfrak{P}}\cong F$.
	In particular, $G'$ must be conjugate to $G_{\mathfrak{P}}$ by Lemma~\ref{lemma: c(Sigma) in general case 1}.
	In fact, by our choice of $\mathfrak{P}$, we see that $G_{\mathfrak{p}}=G_{\mathfrak{P}}\vert_{K_{\mathfrak{p}}}$, so we can choose a section $s:G_{\mathfrak{p}}\to G$ such that $H's(G_{\mathfrak{p}})=G_{\mathfrak{P}}$.
	And this shows that $L_{\mathfrak{P}}\cong F$ not only as local fields, but also as $G_{\mathfrak{P}}$-extensions over $\mathbb{Q}_p$.
	So they induce the same $G$-structured $\mathbb{Q}_p$-algebra:
	\begin{equation*}
		L_p\cong\operatorname{Ind}^G_{G_{\mathfrak{P}}}L_{\mathfrak{P}}\cong\operatorname{Ind}^G_{G_{\mathfrak{P}}}F.
	\end{equation*}
\end{proof}
In summary, in this section we've established the correspondence
\begin{equation*}
	\operatorname{Sur}_{\bar{G}}(J_K^S/K^S,H)\leftrightarrow\{L\in\mathcal{E}\mid K\subseteq L\}.
\end{equation*}
We can generalize the counting functions in Definition~\ref{def: product of ramified primes} to the $\bar{G}$-morphisms $\chi:J_K^S\to G$.
To be precise, for each $p$, we start with a continuous $\bar{G}$-morphism $\chi_{p}:\prod_{\mathfrak{p}\mid p}U_{\mathfrak{p}}\to H$.
By picking up an image for a prime element (in particular the identity $e_H$ always works), there is always an extension $\tilde{\chi}_{\mathfrak{p}}:K_{\mathfrak{p}}^*\to H$ that respects the action of $G_{\mathfrak{p}}$.
For each $\mathfrak{p}\mid p$, let $\pi_{\mathfrak{p}}$ be a prime element of $K_{\mathfrak{p}}$.
By the formula
\begin{equation*}
	g\tilde{\chi}_{g^{-1}\mathfrak{p}}(\pi_{g^{-1}\mathfrak{p}})=g\tilde{\chi}_p(\pi_{g^{-1}\mathfrak{p}},0,\dots,0):=\tilde{\chi}_{\mathfrak{p}}(g\pi_{g^{-1}\mathfrak{p}}),
\end{equation*} 
we obtain the extension $\tilde{\chi}_p$ of $\chi_p$, which is also a  $\bar{G}$-morphism.
The map $\tilde{\chi}$ induces a specification at $p$ by the formula
\begin{equation*}
	\Sigma_p:=\operatorname{Ind}^G_{G'}F
\end{equation*}
where $F$ is the class field of $\tilde{\chi}_{\mathfrak{p}}$ for some $\mathfrak{p}\mid p$ and $G'\cong\operatorname{im}\tilde{\chi}_{\mathfrak{p}}\rtimes G_{\mathfrak{p}}$ is the Galois group of $F/\mathbb{Q}_p$ viewed as a subgroup of $G$ by choosing a fixed section $s:\bar{G}\to G$.
Since conjugation does not affect the isometry class of $\Sigma_p$ defined above, we see that this construction is well-defined for a fixed extension $\tilde{\chi}$ of $\chi$, i.e., $\Sigma_p$ is independent of the choice of $\mathfrak{p}$ above $p$.
If $C$ is a counting function for $\mathcal{E}$, then we see that
\begin{equation}\label{eqn: counting function for morphisms}
	C(\chi):=\prod_{p}p^{c(\Sigma_p)},
\end{equation}
and this number does not depend on the choice of $\tilde{\chi}$, for all the extensions of the map $\chi$ share the same tame inertia generator.
And one can see that if $p\nmid\lvert G\rvert\infty$ and $I_{\mathfrak{p}}$ is trivial for some $\mathfrak{p}\mid p$ (i.e., $p$ is unramified in $K/\mathbb{Q}$), then the inertia generator $y_p$ could be simply taken as a generator of $\operatorname{im}\chi_{\mathfrak{p}}\subseteq H$.
Then we start proving the estimate of field-counting in the next section.

\subsection{Estimate of field-counting}
In this subsection, if we obtain the estimate
\begin{equation*}
	N_{C}(\mathcal{E}_0,X)\ll X^{1/(m+\varepsilon)}
	\quad\text{and}\quad
	N_{C}(\mathcal{E}_{\gamma},X)\asymp\frac{X^{1/m}}{\log X}(\log\log X)^{\gamma-1},
\end{equation*}
where $\gamma$ is a positive integer,
then Lemma~\ref{lemma of general case} is proved immediately.
For this purpose (estimate of field-counting), let us prove the following results one-by-one.
\begin{lemma}\label{lemma of the lemma}
	\begin{enumerate}
		\item Fix an abelian $\bar{G}$-field $K/\mathbb{Q}$, and define $\mathcal{E}_K:=\{L\in\mathcal{E}\mid K\subseteq L\}$.
		We have
		\begin{equation*}
			N_C(\mathcal{E}_K,X)\leq\lvert\operatorname{Hom}_{\bar{G}}(\operatorname{Cl}_K,H)\rvert\cdot\#\{\chi\in\operatorname{Hom}_{\bar{G}}(J_K^{S_{\infty}}/K^{S_{\infty}}, H)\mid C(\chi)<X\},
		\end{equation*}
		where $S_{\infty}$ is the set of all infinite primes of $K$.
		\item Let $A$ and $B$ be two finite abelian $p$-groups.
		We have
		\begin{equation*}
			\lvert \operatorname{Hom}(A,B)\rvert\leq\lvert A\rvert^{\operatorname{rk}_p B}.
		\end{equation*}
		\item Fix a finite discrete $\bar{G}$-module $A$.
		There exist positive constant $c$ and $N$ depending only on $\bar{G}$ and $A$ such that for each abelian $\bar{G}$-field $K$, we have
		\begin{equation}\label{eqn: lemma of the lemma 1}
			\lvert\operatorname{Hom}(\operatorname{Cl}_K,A)\rvert\leq c\sqrt{d_K}^{N},
		\end{equation}
		where $\sqrt{d_K}$ is the radical of the absolute discriminant (product of ramified primes).
	\end{enumerate}
\end{lemma}
\begin{proof}
	(1):
	Recall that $S_{\infty}$ is the set of infinite primes of $K$.
	For a finite set $S$ of primes including the ones at infinity, we have the notions of $S$-id{\`e}les $J_K^{S}=\prod_{\mathfrak{p}\notin S}U_{\mathfrak{p}}\times\prod_{\mathfrak{p}\in S}K_{\mathfrak{p}}^*$, and $S$-units $K^{S}=K^*\cap J_K^{S_\infty}$ (inside $J_K$).
	In addition, there is a short exact sequence
	\begin{equation*}
		1\to J_K^{S}/K^S\to\operatorname{C}_K\to\operatorname{Cl}_K^S\to1.
	\end{equation*}
	When $S=S_{\infty}$, we have $\operatorname{Cl}_K^{S_\infty}=\operatorname{Cl}_K$.
	By taking the homomorphisms that are $\bar{G}$-equivariant to $H$, we have an exact sequence
	\begin{equation*}
		0\to\operatorname{Hom}_{\bar{G}}(\operatorname{Cl}_K,H)\to\operatorname{Hom}_{\bar{G}}(\operatorname{C}_K,H)\to\operatorname{Hom}_{\bar{G}}(J_K^{S_{\infty}}/K^{S_{\infty}},H).
	\end{equation*}
	This implies first that if $\chi:J_K^{S_{\infty}}/K^{S_{\infty}}\to H$ admits a preimage, then the number of $\tilde{\chi}:\operatorname{C}_K\to H$ such that its restriction is exactly $\chi$ is given by the size of $\operatorname{Hom}_{\bar{G}}(\operatorname{Cl}_K,H)$.
	In particular, this shows that the original proof of Lemma 7.10 is flawed.
	That is, when $\operatorname{Hom}_{\bar{G}}(\operatorname{Cl}_K,H)$ is non-trivial, there must be multiple $G$-fields $L/\mathbb{Q}$ including $K$ as a subfield with exactly the same local specifications.
	
	On the other hand, we have shown that $\mathcal{E}_K$ and $\operatorname{Sur}_{\bar{G}}(\operatorname{C}_K,H)$ corresponds to each other bijectively (see Proposition 7.6).
	And we see in Section 7.1 that the counting function $C$ could be generalized to $\chi:J_K^{S_{\infty}}\to H$ (that factors through $K^{S_{\infty}})$.
	Putting everything together, we have
	\begin{equation*}
		\begin{aligned}
			N_C(\mathcal{E}_K,X)=&\#\{\tilde{\chi}\in\operatorname{Sur}_{\bar{G}}(\operatorname{C}_K,H)\mid C(\tilde{\chi})<X\}
			\\
			\leq&\lvert\operatorname{Hom}_{\bar{G}}(\operatorname{Cl}_K,H)\rvert\cdot\#\{\chi\in\operatorname{Hom}_{\bar{G}}(J_K^{S_{\infty}}/K^{S_{\infty}}, H)\mid C(\chi)<X\}.
		\end{aligned}
	\end{equation*}
	And we are done for (1).
	
	(2):
	For finite cyclic $p$-groups, we see that
	\begin{equation*}
		\operatorname{Hom}(\mathbb{Z}/p^a\mathbb{Z},\mathbb{Z}/p^b\mathbb{Z})\cong\mathbb{Z}/p^{\min\{a,b\}}\mathbb{Z}.
	\end{equation*}
	In general, write $A\cong\prod_{i=1}^{\operatorname{rk}_p A}\mathbb{Z}/p^{a_i}\mathbb{Z}$ and $B\cong\prod_{i=1}^{\operatorname{rk}_p B}\mathbb{Z}/p^{b_i}\mathbb{Z}$, we have
	\begin{equation*}
		\operatorname{Hom}(A,B)\cong\prod_{i=1}^{\operatorname{rk}_p A}\prod_{j=1}^{\operatorname{rk}_p B}\mathbb{Z}/p^{\min\{a_i,b_j\}}\mathbb{Z}.
	\end{equation*}
	This implies that
	\begin{equation*}
		\lvert\operatorname{Hom}(A,B)\rvert\leq\prod_{j=1}^{\operatorname{rk}_p B}\prod_{i=1}^{\operatorname{rk}_p A}p^{a_i}=\lvert A\rvert^{\operatorname{rk}_p B}.
	\end{equation*}
	And we are done for (2).
	
	(3):
	By (2), we have
	\begin{equation*}
		\begin{aligned}
			\lvert\operatorname{Hom}_{\bar{G}}(\operatorname{Cl}_K,A)\rvert\leq&\lvert\operatorname{Hom}(\operatorname{Cl}_K,A)\rvert\\
			=&\prod_{p}\lvert\operatorname{Hom}(\operatorname{Cl}_K[p^\infty],A[p^\infty])\rvert
			\\
			\leq&\prod_p\lvert\operatorname{Cl}_K[p^\infty]\rvert^{\operatorname{rk}_p A}
			\leq\lvert\operatorname{Cl}_K\rvert^{r_A},
		\end{aligned}
	\end{equation*}
	where $r_A=\sup_p\operatorname{rk}_pA$.
	By Minkowski bound of the size of the class group, for each $\varepsilon>0$, there exists a constant $c=c(\lvert \bar{G}\rvert,\varepsilon)$ such that
	\begin{equation*}
		\lvert\operatorname{Cl}_K\rvert^{r_A}\leq c d_K^{\frac{1+\varepsilon}{2}r_A},
	\end{equation*}
	where $d_K$ is the absolute discriminant of $K/\mathbb{Q}$.
	Since the Galois group $\bar{G}$ is fixed, we see that there exists $N=N(\bar{G},\varepsilon,A)$ such that $(\sqrt{d_K})^N>d_K^{\frac{1+\varepsilon}{2}r_A}$, where $\sqrt{d_K}$ means the radical of the absolute discriminant (or just the product of ramified primes), that is, 
	\begin{equation*}
		\lvert\operatorname{Hom}(\operatorname{Cl}_K,A)\rvert\leq c\sqrt{d_K}^{N},
	\end{equation*}
	and we are done for the proof.
\end{proof}
Let us first give an upper bound.
\begin{lemma}[Correction of the original Lemma 7.10]\label{lemma: estimate for upper bound}
	Assume that $\bar{G}$ is abelian.
	\begin{enumerate}
		\item Let
		\begin{equation*}
			m':=\min_{g\notin\Omega\cup\{e_G\}}c_G(g).
		\end{equation*}
		For each $0<\varepsilon<1$, as $X\to\infty$, we have 
		\begin{equation*}
			N_C(\mathcal{E}_0,X)\ll X^{1/(m'-\varepsilon)}.
		\end{equation*}
		\item For each positive integer $\gamma$, as $X\to\infty$, we have
		\begin{equation*}
			N_{C}(\mathcal{E}_{\gamma},X)\ll\frac{X^{1/m}}{\log X}(\log\log X)^{\gamma-1}.
		\end{equation*}		
	\end{enumerate}
\end{lemma}
\begin{proof}
	Let us prove (i) and (ii) together.
	For each abelian $\bar{G}$-field $K$, define
	\begin{equation*}
		\mathcal{E}_{K,\gamma}:=\{L\in\mathcal{E}_{\gamma}\mid K\subseteq L\}.
	\end{equation*}
	For each $p\nmid\lvert G\rvert\infty$, define
	\begin{equation*}
		\Omega_p^+:=\{\Sigma_p\mid y_p\in\Omega\}
		\quad\text{and}\quad
		\Omega_p^-:=\{\Sigma_p\mid y_p\notin\Omega\}
	\end{equation*}
	where $\Sigma_p$ is a specification at $p$ with respect to $G$, and $y_p$ is an inertia generator (up to conjugate).
	To abbreviate the notations, write $I^{+,\mu}:=I^{+,\mu}_{\mathcal{P}\backslash S}$, where $S:=\{p\in\mathcal{P}\mid p\nmid\lvert G\rvert\infty\}$.
	In other words, $I^{+,\mu}$ is the set of square-free positive integers $d$ such that $p\mid d$ only if $p\nmid\lvert G\rvert$.
	For each integer $d\in I^{+,\mu}$, define
	\begin{equation*}
		\Phi(d):=\prod_{p\mid d}\Omega_p^+\times\prod_{p\nmid d\lvert G\rvert\infty}\Omega_p^-.
	\end{equation*}
	Clearly $\Phi(d)$ is determined by the inertia generators, so for each $\bar{G}$-morphism $\chi:\prod_{\mathfrak{p}\mid p}U_{\mathfrak{p}}\to H$, the notations $\chi\in\Omega_p^{\pm}$ are well-defined. 
	And we can define
	\begin{equation*}
		\tilde{\mathcal{E}}_{K,\gamma}:=
		\bigsqcup_{
			\substack{
				d\in I^{+,\mu}\\
				\omega(d)=\gamma
			}
		}\{\chi\in\operatorname{Hom}_{\bar{G}}(J_K^{S_{\infty}},H)\mid\chi\in\Phi(d)\}.
	\end{equation*}
	The relation between $\tilde{\mathcal{E}}_{K,\gamma}$ and $\mathcal{E}_{K,\gamma}$ is induced by Lemma~\ref{lemma of the lemma}(1) in the sense of statistics, that is,
	\begin{equation*}
		N_C(\mathcal{E}_{K,\gamma},X)\leq\lvert\operatorname{Hom}_{\bar{G}}(\operatorname{Cl}_K,H)\rvert N_C(\tilde{\mathcal{E}}_{K,\gamma},X).
	\end{equation*}
	Let us define a generating series for $\tilde{\mathcal{E}}_{K,\gamma}$ in the form of an Euler product
	\begin{equation*}
		\begin{aligned}
			F_{C,\tilde{\mathcal{E}}_{K,\gamma}}(s):=
			&\Bigl(\sum_{\chi\in\operatorname{Hom}_{\bar{G}}(\prod_{\mathfrak{p}\mid\infty}K_{\mathfrak{p}}^*,H)}1\Bigr)
			\prod_{p\mid\lvert G\rvert}\Bigl(
			\sum_{\chi\in\operatorname{Hom}_{\bar{G}}(\prod_{\mathfrak{p}\mid p}K_{\mathfrak{p}}^*,H)}p^{-c(\chi)s}
			\Bigr)
			\\
			\cdot&\sum_{
				\substack{
					d\in I^{+,\mu}\\
					\omega(d)=\gamma
				}
			}
			\prod_{p\mid d}
			\Bigl(\sum_{
				\substack{
					\chi\in\operatorname{Hom}_{\bar{G}}(\prod_{\mathfrak{p}\mid p}U_{\mathfrak{p}},H)\\
					\chi(y_p)\in\Omega
				}			
			}p^{-c(\chi)s}\Bigr),
		\end{aligned}
	\end{equation*}
	where $y_p$ is a generator of the group of roots of unity in $U_{\mathfrak{p}}$.
	Note that the map $c_G:G\to\mathbb{Z}_{\geq0}$ induces a map $c_{\bar{G}}:\bar{G}\to\mathbb{Z}_{\geq0}$.
	Then we obtain a counting function $C'$ for $\bar{G}$-fields with respect to the maps $c_{\bar{G}}$ and $\{c'_p\}_{p\mid\lvert G\rvert}$, where $c'_p$ is trivial (zero map).
	Note that for each prime $p\mid\lvert G\rvert\infty$, and for each $\bar{G}$-map $\chi:\prod_{\mathfrak{p}\mid p}K_{\mathfrak{p}}^*\to H$, the value $c(\chi)$ is exactly the value of the specification at $p$ that corresponds to $\chi$.
	In particular, it takes into account the ramification of $p$ in the extension $K/\mathbb{Q}$.
	So, for each $\bar{G}$-morphism $\chi:J_K^{S_\infty}\to H$, if there is a surjective $\bar{G}$-morphism $\tilde{\chi}:\operatorname{C}_K\to H$ that induces the same map on $J_K^{S_\infty}$, then we have
	\begin{equation*}
		C(L)=C'(K)C(\chi),
	\end{equation*}
	where $L$ is the $G$-field that corresponds to $\tilde{\chi}$.
	We get a generating series as follows:
	\begin{equation*}
		F_{C,\tilde{\mathcal{E}}_{\gamma}}(s):=\sum_{K\in\mathcal{E}(\bar{G})}\lvert\operatorname{Hom}_{\bar{G}}(\operatorname{Cl}_K,H)\rvert C'(K)^{-s}F_{C,\tilde{\mathcal{E}}_{K,\gamma}}(s).
	\end{equation*}
	We want to reduce it to a single Euler product so that we could deduce the analytic properties of the series.
	Let us first apply Lemma~\ref{lemma of the lemma} here.
	By (\ref{eqn: lemma of the lemma 1}), for each abelian $\bar{G}$-field extension $K/\mathbb{Q}$, there exist positive constants $d$ and $N$ depending only on $\bar{G}$ and $H$ such that 
	\begin{equation*}
		\lvert\operatorname{Hom}_{\bar{G}}(\operatorname{Cl}_K,H)\rvert\leq d\sqrt{d_K}^N.
	\end{equation*}
	We may assume without loss of generality that $N$ is an integer, and obtain the following series
	\begin{equation*}
		E_{C,\tilde{\mathcal{E}}_{\gamma}}(s)=\sum_{K\in\mathcal{E}(\bar{G})}d\sqrt{d_K}^NC'(K)^{-s}F_{C,\tilde{\mathcal{E}}_{K,\gamma}}(s).
	\end{equation*}
	Apply Proposition 7.7 here, and we obtain a positive integer $a$ for each $\bar{G}$-field $K$ and for each $p\in\mathcal{P}$ as an upper bound for $\lvert\operatorname{Hom}_{\bar{G}}(\prod_{\mathfrak{p}\mid p}K^*_{\mathfrak{p}},H)\rvert$,
	and $\lvert\operatorname{Hom}_{\bar{G}}(\prod_{\mathfrak{p}\mid p}U_{\mathfrak{p}},H)\rvert$.
	Because
	\begin{equation*}
		\lvert\operatorname{Hom}_{\bar{G}}(\prod_{\mathfrak{p}\mid p}U_{\mathfrak{p}},H)\rvert\leq\lvert\operatorname{Hom}_{\bar{G}}(\prod_{\mathfrak{p}\mid p}K_{\mathfrak{p}}^*,H)\rvert\leq\#\{\text{Specifications at }p\text{ w.r.t }G\}.
	\end{equation*}
	Or equivalently one can see the existence of such a number $a$ by the structure of $K^*_{\mathfrak{p}}$ as a topological group.
	In particular, the number $a$ could be chosen so that it is uniform for all $\bar{G}$-fields $K$.
	Let $M:=\max_{g\in G}c_G(g)$.
	For each non-negative integer $\gamma$, define
	\begin{equation*}
		D_{C,\tilde{\mathcal{E}}_{\gamma}}(s):=
		d\prod_{p\mid\lvert G\rvert\infty}\Bigl(\sum_{\chi:\mathbb{Z}_p^*\to\bar{G}}1+\sum_{\Sigma_p}p^{-c(\Sigma_p)s}\Bigr)
		\prod_{p\nmid\lvert G\rvert\infty}(1+ap^N\sum_{n=m'}^M p^{-ns})
		\sum_{
			\substack{
				d\in I^{+,\mu}\\
				\omega(d)=\gamma
			}
		}
		\prod_{p\mid d}ap^{-ms},
	\end{equation*}
	where $\Sigma_p$ runs over specifications at $p$ with respect to $G$.
	Write
	\begin{equation*}
		\sum_n b_nn^{-s}:=d\prod_{p\mid\lvert G\rvert\infty}\Bigl(\sum_{\chi:\mathbb{Z}_p^*\to\bar{G}}1\Bigr)\prod_{p\nmid\lvert G\rvert\infty}(1+ap^N\sum_{n=m'}^M p^{-ns}).
	\end{equation*}
	Note that the sum $a\sum_{n=m'}^M p^{-ns}$ in the local factor over counts the number of nontrivial elements in $\operatorname{Hom}(\mathbb{Z}_p^*,\bar{G})$ when $p\nmid\lvert G\rvert\infty$, we have
	\begin{equation*}
		\begin{aligned}
			b_n\geq&\sum_{C'(K)=n}d\sqrt{d_K}^N
			\\
			\geq&\sum_{C'(K)=n}\lvert\operatorname{Hom}_{\bar{G}}(\operatorname{Cl}_K,H)\rvert.
		\end{aligned}
	\end{equation*}
	Similarly if we write
	\begin{equation*}
		\sum_n c_nn^{-s}:=
		\prod_{p\mid\lvert G\rvert\infty}\Bigl(\sum_{\Sigma_p}p^{-c(\Sigma_p)s}\Bigr)
		\sum_{
			\substack{
				d\in I^{+,\mu}\\
				\omega(d)=\gamma
			}
		}
		\prod_{p\mid d}ap^{-ms}
		\quad\text{and}\quad
		\sum_n c_{K,n}n^{-s}:=F_{C,\tilde{\mathcal{E}}_{K,\gamma}}(s),
	\end{equation*}
	then for each $K\in\mathcal{E}(\bar{G})$, we have
	\begin{equation*}
		c_n\geq c_{K,n}.
	\end{equation*}
	So, we have obtained a uniform Euler product to describe an upper bound estimate for field-counting as promised.
	Therefore, it suffices to consider the analytic properties of $\sum_n a_nn^{-s}:=D_{C,\tilde{\mathcal{E}}_{\gamma}}(s)$.
	
	First consider the case when $m'-N>m$ and $\gamma>0$.
	The series $D_{C,\tilde{\mathcal{E}}_{\gamma}}(s)$ defines a holomorphic function in the open half-plane $\sigma>1/m$ by comparing with a suitable power of $\zeta(ms)$.
	By Lemma 5.2 and Proposition 5.3, the function $D_{C,\tilde{\mathcal{E}}_{\gamma}}(s)$ could be extended analytically to the line $\sigma=1/m$, and there exist holomorphic functions $f_0,\dots,f_{\gamma}$ at $s=1/m$ such that
	\begin{equation*}
		D_{C,\tilde{\mathcal{E}}_{\gamma}}(s)=\sum_{i=0}^{\gamma}f_i(s)\Bigl(\log\frac{1}{s-1/m}\Bigr)^{i}.
	\end{equation*}
	The Tauberian Theorem~\cite[Appendix II Theorem I]{narkiewicz2014elementary} implies that
	\begin{equation*}
		\sum_{n<X}a_n\sim A\frac{X^{1/m}}{\log X}(\log\log X)^{\gamma-1},
	\end{equation*}
	where $A$ is a positive constant.
	Fix a small constant $0<\varepsilon<1$.
	If $\gamma=0$, or $m'-N\leq m$, let $l$ be a large enough positive integer such that $0<N/l<\varepsilon$.
	Let $C^l(L):=C(L)^l$, we have
	\begin{equation*}
		N_C(\mathcal{E}_{\gamma},X)=N_{C^l}(\mathcal{E}_{\gamma},X^l).
	\end{equation*}
	The method above could be applied to the case when the counting function $C$ is replaced by $C^l$ in the following sense.
	If $\gamma>0$, then by $m'-1\geq m$, we have
	\begin{equation*}
		lm'-N>l(m'-\varepsilon)>l(m'-1)\geq lm.
	\end{equation*}
	So we come back to the first case and
	\begin{equation*}
		N_C(\mathcal{E}_{\gamma},X)=N_{C^l}(\mathcal{E}_{\gamma},X^l)\ll\frac{X^{1/m}}{\log X}(\log\log X)^{\gamma-1}.
	\end{equation*}
	If $\gamma=0$, then the corresponding series $D_{C^l,\tilde{\mathcal{E}}_{\gamma}}(s)$ is of the form
	\begin{equation*}
		D_{C^l,\tilde{\mathcal{E}}_{\gamma}}(s)=d\prod_{p\mid\lvert G\rvert\infty}\Bigl(\sum_{\chi:\mathbb{Z}_p^*\to\bar{G}}1+\sum_{\Sigma_p}p^{-lc(\Sigma_p)s}\Bigr)
		\prod_{p\nmid\lvert G\rvert\infty}(1+ap^N\sum_{n=lm'}^{lM} p^{-ns}).
	\end{equation*}
	The Proposition 5.3 implies that $D_{C^l,\tilde{\mathcal{E}}_{\gamma}}(s)$ defines a holomorphic function in the open half-plane $\sigma>1/(lm'-N)$, and it could be extended analytically to the line $\sigma=1/(lm'-N)$ with a unique pole at $s=1/(lm'-N)$ of order $a$.
	So the Tauberian Theorem~\cite[Appendix II Theorem I]{narkiewicz2014elementary} implies that
	\begin{equation*}
		N_C(\mathcal{E}_{0},X)=N_{C^l}(\mathcal{E}_{0},X^l)\ll X^{1/(m'-N/l)}(\log X)^{a-1}\ll X^{1/(m'-\varepsilon)}.
	\end{equation*}
	And we are done for the proof.
\end{proof}
Then we prove a lower bound for $N_{C}(\mathcal{E}_{\gamma},X)$.
\begin{lemma}\label{lemma: estimate for lower bound}
	Fix a $\bar{G}$-field $K$, and define
	\begin{equation*}
		\mathcal{E}_K:=\{L\in\mathcal{E}\mid K\subseteq L\}.
	\end{equation*}
	For each positive integer $\gamma$, let $\mathcal{E}_{K,\gamma}:=\mathcal{E}_{\gamma}\cap\mathcal{E}_{K}$.
	Then we have
	\begin{equation*}
		N_C(\mathcal{E}_{K,\gamma},X)\gg \frac{X^{1/m}}{\log X}(\log\log X)^{\gamma-1},
	\end{equation*}
	as $X\to\infty$.
\end{lemma}
\begin{proof}
	\emph{Step 1:}
	Let $S$ be a finite set of primes of $K$ such that
	\begin{enumerate}
		\item $S$ is a $\bar{G}$-set, and there exists a surjective $\bar{G}$-morphism $\varphi:\prod_{\mathfrak{p}\in S}K_{\mathfrak{p}}^*\to H$;
		\item  $J_K^S/K^S\cong\operatorname{C}_K$, i.e., $\operatorname{Cl}_K^S$ is trivial;
		\item for each finite prime $\mathfrak{p}\notin S$, $\mathfrak{p}\nmid d_K\lvert H\rvert\infty$, where $d_K$ means the (absolute) discriminant of $K$.
	\end{enumerate}
	Such a set $S$ exists.
	Because we can first take a set $S'$ such that $J_K^{S'}/K^{S'}\cong\operatorname{C}_K$ (see also Neukirch~\cite[VI.1.4]{neukirch2013algebraic}).
	Then replace $S'$ by a larger set of primes if necessary, there must be a surjective $G$-morphism $\varphi:\prod_{\mathfrak{p}\in S'}K_{\mathfrak{p}}^*\to H$.
	So take $S$ to be a $\bar{G}$-set that includes $S'$ and $\{\mathfrak{p}:\mathfrak{p}\mid d_K\lvert H\rvert\infty\}$, then $S$ will satisfy these three conditions simultaneously.
	
	Since $S$ is a $\bar{G}$-set, we say that a rational prime $p\in S$ if for some $\mathfrak{p}\mid p$ we have $\mathfrak{p}\in S$.
	Clearly this definition does not depend on the choice of $\mathfrak{p}$, i.e., $p\in S$ if and only if $\mathfrak{p}\in S$ for each $\mathfrak{p}\mid p$.	
	
	Recall that for each $\bar{G}$-morphism $\chi_p:\prod_{\mathfrak{p}\mid p}U_{\mathfrak{p}}\to H$, we could extend it to some $\bar{G}$-morphism $\tilde{\chi}_p:\prod_{\mathfrak{p}\mid p}K_{\mathfrak{p}}^*\to H$ (the choice is not unique).
	And such a map $\tilde{\chi}_p$ corresponds bijectively to a specification $\Sigma_p$ at $p$ with respect to $G$.
	Define
	\begin{equation*}
		\Omega_p^+:=\{\Sigma_p\mid y_p\in\Omega\}
	\end{equation*}
	to be the set of specifications at $p$ such that its inertia generator is contained in $\Omega$.
	Since the definition of $\Omega_p^+$ only depends on the inertia generators, the notation $\chi_p\in\Omega_p^+$ is well-defined.
	For a $\bar{G}$-morphism $\chi:J_K^S\to H$, we say that $\chi\in\tilde{\mathcal{E}}_{K,\gamma}$ if there are exactly $\gamma$ primes $p\notin S$ such that $\chi_p\in\Omega_p^+$.
	Recall also that if $\chi\in\tilde{\mathcal{E}}_{K,\gamma}$ is surjective and factors through $J_K^S/K^S$, then the corresponding field $L$ is contained in $\mathcal{E}_{K,\gamma}$.
	In this case we simply write $\chi\in\mathcal{E}_{K,\gamma}$.
	
	\emph{Step 2:}
	Then we follow the idea of Wood~\cite{wood2010probabilities} to construct the generating series and so on.
	Using the structure theorem of finite abelian groups, write $H=\sum_{j=1}^l\mathbb{Z}/n_j\mathbb{Z}$, where $n_j\mid n_{j+1}$ for each $j=1,\dots,l-1$.
	Define $\mathcal{A}:=\prod_{j=1}^l K^S/(K^S)^{n_j}$.
	For each continuous homomorphism $\chi:J_K^S\to H$, define $\chi_j:J_K^S\to\mathbb{Z}/n_j\mathbb{Z}$ to be the projection.
	And for each $\epsilon=(\epsilon_1,\dots,\epsilon_l)\in\mathcal{A}$, define
	\begin{equation*}
		\dot{\chi}(\epsilon):=\prod_{j=1}^{l}\zeta_{n_j}^{\chi_j(\epsilon_j)}.
	\end{equation*}
	For the computation of this notation $\dot{\chi}(\epsilon)$, we refer to Wood~\cite[Lemma 2.16 and 2.17]{wood2010probabilities} for details.
    Another result (or a corollary) we cite from this paper is the following.
    Let $\mathfrak{p}\nmid\lvert G\rvert\infty$ be a prime of $K$.
    If $\chi_{\mathfrak{p}}:K_{\mathfrak{p}}^*\to H$ is a continuous map such that the tame inertia generator of the corresponding $H$-structured $K_{\mathfrak{p}}$-algebra is $g\in H$, then Wood~\cite[Lemma 2.17]{wood2010probabilities} implies that there exists some Hecke character $\theta_{\epsilon^g}$ of $K(\zeta_{r_g})$ such that
    \begin{equation*}
    	\theta_{\epsilon^g}(\mathfrak{P})=\frac{\operatorname{Frob}_{\mathfrak{P}}(\epsilon^g)}{\epsilon^g}=\dot{\chi}_{\mathfrak{p}}(\epsilon)
    \end{equation*}
    where $\mathfrak{P}$ is any prime of $K(\zeta_{r_g})$ above $\mathfrak{p}$, and $\epsilon^g$ is defined in Wood~\cite[Lemma 2.17]{wood2010probabilities}.
    Since the value is independent of the choice of $\mathfrak{P}$, by abuse of notation, let us just write it as $\theta_{\epsilon^g}(\mathfrak{p})$.
    But keep in mind that the Hecke character is defined for $K(\zeta_{r_g})$.
    
	\emph{Step 3:}
	Now let $\varphi:\prod_{\mathfrak{p}\in S}K_{\mathfrak{p}}^*\to H$ be a surjective $\bar{G}$-morphism.
	For each $\gamma\in\mathbb{Z}_+$, define
	\begin{equation*}
		F_{C,\mathcal{E}_{K,\gamma}}(\varphi,s):=
		\sum_{
			\substack{
				\chi\in\mathcal{E}_{K,\gamma}\\
				\chi_{\mathfrak{p}}=\varphi_{\mathfrak{p}},\forall\mathfrak{p}\in S
			}
		}C(\chi)^{-s}.
	\end{equation*}
	For simplicity, it suffices to consider the following specifications and the corresponding generating series.
	For each $p\notin S$, define
	\begin{equation*}
		\Omega_p^0:=\{\Sigma_p\mid y_p=e_G\}
	\end{equation*}
	where $\Sigma_p$ is a specification at $p$ with respect to $G$, and $y_p$ is its inertia generator.
	Recall again that for each $\bar{G}$-morphism $\chi_p:\prod_{\mathfrak{p}\mid p}U_{\mathfrak{p}}\to H$, the notations $\chi_p\in\Omega_p^+$ and $\chi_p\in\Omega_p^0$ are well-defined because $\Omega^+_p$ and $\Omega_p^0$ are defined by inertia generators.
	Let
	\begin{equation*}
		\mathcal{P}_{K/\mathbb{Q}}(e_G):=\{p\text{ rational prime}\mid p\notin S\text{ and }p\text{ totally splits in }K/\mathbb{Q}\},
	\end{equation*}
    where $e_G$ means that the Frobenius is trivial.
	By \v{C}ebotarev density Theorem, we know that this set is infinite.
    To abbreviate the notations, write $I^{+,\mu}_{e_G}:=I^{+,\mu}_{\mathcal{P}_{K/\mathbb{Q}}(e_G)}$.
	For each integer $d\in I^{+,\mu}_{e_G}$, define
	\begin{equation*}
		\Phi(d):=\prod_{p\mid d}\Omega_p^+\times\prod_{p\notin S,p\nmid d}\Omega_p^0.
	\end{equation*}
	Define
	\begin{equation*}
		\begin{aligned}
			\mathcal{E}'_{K,\gamma}:=
			&\bigsqcup_{
				\substack{
					d\in I^{+,\mu}_{e_G}\\
					\omega(d)=\gamma
				}
			}
			\{\chi\in\mathcal{E}_{K,\gamma}\mid\chi\in\Phi(d)\}
			\\
			\tilde{\mathcal{E}}'_{K,\gamma}:=
			&\bigsqcup_{
				\substack{
					d\in I^{+,\mu}_{e_G}\\
					\omega(d)=\gamma
				}
			}\{\chi\in\tilde{\mathcal{E}}_{K,\gamma}\mid\chi\in\Phi(d)\}.
		\end{aligned}		
	\end{equation*}
	Clearly $\mathcal{E}'_{K,\gamma}$ is a subset of $\mathcal{E}_{K,\gamma}$.
	We have the generating series
	\begin{equation*}
		F_{C,\mathcal{E}'_{K,\gamma}}(\varphi,s):=
		\sum_{
			\substack{
				\chi\in\mathcal{E}'_{K,\gamma}\\
				\chi_{\mathfrak{p}}=\varphi_{\mathfrak{p}},\forall\mathfrak{p}\in S
			}
		}C(\chi)^{-s}.
	\end{equation*}
	If we write $F_{C,\mathcal{E}_{K,\gamma}}(\varphi,s)=\sum_n a_n n^{-s}$ and $F_{C,\mathcal{E}'_{K,\gamma}}(\varphi,s)=\sum_n b_n n^{-s}$, then $b_n\leq a_n$ for all positive integer $n$.
	For each $\epsilon\in\mathcal{A}$, we have the following:
	\begin{equation*}
		\begin{aligned}
			F_{C,\tilde{\mathcal{E}}'_{K,\gamma}}(\varphi,\epsilon,s):=
			&\sum_{
			    \substack{
				\chi\in\tilde{\mathcal{E}}'_{K,\gamma}\\
				\chi_{\mathfrak{p}}=\varphi_{\mathfrak{p}},\forall\mathfrak{p}\in S
			    }
		    }\dot{\chi}(\epsilon)C(\chi)^{-s}C(\varphi)^{-s}
			\\
                =&C(\varphi)^{-s}\sum_{
				\substack{
					d\in I^{+,\mu}_{e_G}\\
					\omega(d)=\gamma
				}
			}
			\prod_{p\mid d}\Bigl(\sum_{\chi_p\in\Omega_p^+}\dot{\chi}_p(\epsilon)p^{-ms}\Bigr).
		\end{aligned}
	\end{equation*}
    The relation between $F_{C,\mathcal{E}'_{K,\gamma}}(\varphi,s)$ and $F_{C,\tilde{\mathcal{E}}'_{K,\gamma}}(\varphi,\epsilon,s)$ is given by the expression:
    \begin{equation*}
        F_{C,\mathcal{E}'_{K,\gamma}}(\varphi,s)=\frac{1}{\lvert\mathcal{A}\rvert}\sum_{\epsilon\in\mathcal{A}}F_{C,\tilde{\mathcal{E}}'_{K,\gamma}}(\varphi,\epsilon,s).
    \end{equation*}
    That is, a field $L\in\mathcal{E}'_{K,\gamma}$ is represented by a unique surjective map $\chi\in\tilde{\mathcal{E}}'_{K,\gamma}$ that factors through $K^S$.
	For two functions $f(s)$ and $g(s)$ defined in the open half-plane $\sigma>1/m$, we say that $f(s)\approx_m g(s)$ if the ratio $f/g$ admits an analytic continuation to the closed half-plane $\sigma\geq1/m$.
    For $\gamma=1$, we do the computation as follows.
    \begin{equation*}
        \begin{aligned}
            F_{C,\tilde{\mathcal{E}}'_{K,1}}(\varphi,\epsilon,s)
			=&C(\varphi)^{-s}
            \sum_{p\in\mathcal{P}_{K/\mathbb{Q}}(e_G)}
            \Bigl(\sum_{\chi_p\in\Omega_p^+}\dot{\chi}_p(\epsilon)p^{-ms}\Bigr)
			\\
            =&C(\varphi)^{-s}
            \sum_{p\in\mathcal{P}_{K/\mathbb{Q}}(e_G)}p^{-ms}
            \Bigl(
            \sum_{\chi_p\in\Omega_p^+}
            \prod_{
                \substack{
                    \mathfrak{p}\in\mathcal{P}_K\\
                    \mathfrak{p}\mid p
                }
            }
            \dot{\chi}_{\mathfrak{p}}(\epsilon)
            \Bigr)
            \\
            \approx_m&
            \sum_{g\in\Omega}
            \sum_{
                \substack{
                    p\in\mathcal{P}_{K/\mathbb{Q}}(e_G)\\
                    p\equiv1\bmod{r_g}
                }
            }p^{-ms}
            \prod_{x\in\bar{G}}\theta_{\epsilon^{g_{x\mathfrak{p}}}}(x\mathfrak{p}),
        \end{aligned}
    \end{equation*}
    where $\mathfrak{p}$ is some prime of $K$ above $p$, and $g_{x\mathfrak{p}}$ is the inertia generator at $x\mathfrak{p}$.
    We use a trick to express this product in the last row into a single Hecke character.
    For each $p\in\mathcal{P}_{K/\mathbb{Q}}(e_G)$, we know that $K_{\mathfrak{p}}\cong\mathbb{Q}_p$.
    Therefore we may treat the inertia generator $y_{\mathfrak{p}}$ as an element of $K_{\mathfrak{p}}^*$ such that $y_{\mathfrak{p}}\equiv\zeta_{p-1}\bmod{\mathfrak{p}}$, where $\zeta_{p-1}\in(\mathbb{Z}/p\mathbb{Z})^*$ is a fixed generator.
    Note that $x:K_{\mathfrak{p}}\to K_{x\mathfrak{p}}$ is a $\mathbb{Q}_p$-isomorphism.
    We see that $xy_{\mathfrak{p}}$ also satisfies $xy_{\mathfrak{p}}\equiv\zeta_{p-1}\bmod{\mathfrak{p}}$, for $\zeta_{p-1}$ could be represented by an integer which is fixed by any Galois action, that is, $xy_{\mathfrak{p}}=y_{x\mathfrak{p}}$.
    So, the product of the last row also defines a Hecke character of $K(\zeta_{r_g})$ by
    \begin{equation*}
        \psi_{\epsilon^g}:=\prod_{x\in\bar{G}}\theta_{\epsilon^{xg}}\circ\tilde{x},
    \end{equation*}
    where $\tilde{x}$ is a preimage of $x$ in $G(K(\zeta_{r_g})/\mathbb{Q})$.
    Since the choice of $\mathfrak{P}/\mathfrak{p}$ in the extension $K(\zeta_{r_g})/K$ does not affect the value of $\theta_{\epsilon^g}$, we see that this is a well-defined product of Hecke characters, hence still a Hecke character.
    Note also that $p$ totally splits in $\mathbb{Q}(\zeta_{r_g})/\mathbb{Q}$ if and only if $p\equiv1\bmod{r_g}$.
    So, we can write 
    \begin{equation*}
        F_{C,\tilde{\mathcal{E}}'_{K,1}}(\varphi,\epsilon,s)\approx_m
        \sum_{g\in\Omega}
            \sum_{p\in\mathcal{P}_{K(\zeta_{r_g})/\mathbb{Q}}(e)}
            \sum_{
                \substack{
                    \mathfrak{P}\in\mathcal{P}_{K(\zeta_{r_g})}\\
                    \mathfrak{P}\mid p
                }
            }p^{-ms}\psi_{\epsilon^g}(\mathfrak{P}).
    \end{equation*}
    Then let us deal with the condition that $p\in\mathcal{P}_{K/\mathbb{Q}}(e_G)$ so that everything could be rearranged as a combination of $L$-functions.
    Note that an unramified finite prime $p$ splits completely in $K/\mathbb{Q}$ if and only if the Frobenius at $p$ is trivial.
    So, by orthogonality of characters, we have
    \begin{equation*}
        \begin{aligned}
            &\sum_{p\in\mathcal{P}_{K(\zeta_{r_g})/\mathbb{Q}}(e)}p^{-ms}\sum_{\mathfrak{P}\mid p}\psi_{\epsilon^g}(\mathfrak{P})
            \\
            =&\sum_{p\text{ unr}}p^{-ms}\sum_{\mathfrak{P}\mid p}\psi_{\epsilon^g}(\mathfrak{P})
            \Bigl(
            \frac{1}{\lvert C_{\bar{G}}(e_{\bar{G}})\rvert}\sum_{\chi\in\hat{\bar{G}}}\bar{\chi}(\operatorname{Frob}_p)
            \Bigr)
            \Bigl(
            \frac{1}{\phi(r_g)}\sum_{\psi\in\widehat{(\mathbb{Z}/r_g)^*}}\bar{\psi}(p)
            \Bigr)
            \\
            =&\frac{1}{\phi(r_g)\lvert \bar{G}\rvert}
            \sum_{
                \substack{
                    \chi\in\hat{\bar{G}}\\
                    \psi\in\widehat{(\mathbb{Z}/r_g)^*}
                }
            }
            \sum_{p\text{ unr}}
            \bar{\chi}(\operatorname{Frob}_p)\bar{\psi}(p)p^{-ms}\sum_{\mathfrak{P}\mid p}\psi_{\epsilon^g}(\mathfrak{P}).
        \end{aligned}
    \end{equation*}
    Note that $\bar{\chi}(\operatorname{Frob}_p)$ is also a Dirichlet character.
    To be precise, there exists a Dirichlet character $\psi_{\chi}$ such that $\psi_{\chi}(p)=\chi(\operatorname{Frob}_p)$ for all unramified $p$.
    Lift the characters to $K(\zeta_{r_g})$, that is, for each $(\chi,\psi)\in\widehat{\bar{G}}\times\widehat{(\mathbb{Z}/r_g)^*}$, we define a character $\rho_{g,\chi,\psi}$ by 
    \begin{equation*}
        \rho_{g,\chi,\psi}:=(\psi_{\chi}\cdot\psi)\circ\operatorname{Nm}_{K(\zeta_{r_g})/\mathbb{Q}}.
    \end{equation*}
        Recall that a Hecke $L$-series $L(\psi,s)$ admits a zero-free region (Coleman~\cite{Coleman_1990}).
        So, we have 
    \begin{equation*}
      F_{C,\tilde{\mathcal{E}}'_{K,1}}(\varphi,\epsilon,s)
      \approx_m
      \sum_{g\in\Omega}
      \sum_{
      	\substack{
      		\chi\in\hat{\bar{G}}\\
      		\psi\in\widehat{(\mathbb{Z}/r_g)^*}
      		}
        }
        \log L(\rho_{g,\chi,\psi}\cdot\psi_{\epsilon^g},ms).
    \end{equation*}
    The relation $\approx_m$ holds because the comparison of the series is mainly determined by the totally split primes, that is, they have the same coefficient for each $p^{-ms}$.
    Note that $\widehat{\bar{G}}$ and $\widehat{(\mathbb{Z}/r_g)^*}$ admit a trivial character.
    As for $\psi_{\epsilon^g}$, i.e., $\theta_{\epsilon^g}$, if the element $\epsilon^g$ defined as in Wood~\cite[Lemma 2.17]{wood2010probabilities} is contained in $K(\zeta_{r_g})$, then $\theta_{\epsilon^g}$ is trivial.
    And this condition is automatically satisfied when $\epsilon$ is taken to be the identity of $\mathcal{A}$.
    Recall that if a Hecke character $\psi$ is nontrivial, then $L(\psi,ms)$ is holomorphic in the closed half-plane $\sigma\geq1/m$.
    Else if $\psi$ is principal, then $L(\psi,s)$ and $\zeta_{K(\zeta_{r_g})}(s)$ differ by finitely many local factors, and $L(\psi,ms)$ admits a holomorphic continuation to the line $s=1/m$ except at $s=1/m$
    In particular, near the point $s=1/m$, we have
    \begin{equation*}
        \lim_{s\to1}\frac{\log L(\psi,ms)}{\log\bigl(\frac{1}{s-1/m}\bigr)}=1,
    \end{equation*}
    which is exactly the order of the pole of $L(\psi,ms)$ at $s=1/m$.
	This shows that for each $\epsilon\in\mathcal{A}$, the series $F_{C,\tilde{\mathcal{E}}'_{K,1}}(\varphi,\epsilon,s)$ is a holomorphic function in the open half-plane $\sigma>1/m$,
    and it admits a holomorphic continuation to the line $\sigma=1/m$ except at $s=1/m$.
	In addition, there exist holomorphic functions $f_{1,1}(s),f_{1,0}(s)$ at $s=1/m$ such that $f_{1,1}(1/m)$ is positive real and such that
	\begin{equation*}
		F_{C,\mathcal{E}'_{K,1}}(\varphi,s)=f_{1,1}(s)\log\Bigl(\frac{1}{s-1/m}\Bigr)+f_{1,0}(s)
	\end{equation*}
	near $s=1/m$.
        By Tauberian Theorem~\cite[Appendix II Theorem I]{narkiewicz2014elementary} we see that
	\begin{equation*}
		N_C(\mathcal{E}_{K,1},X)\geq N_C(\mathcal{E}'_{K,1},X)\sim b_0\frac{X^{1/m}}{\log X}
	\end{equation*}
	where $b_0$ is a positive real number.
	This proves the statement when $\gamma=1$.
	The case for general $\gamma$ follows from Lemma~\ref{lemma: l(r,s)}, for if we write
	\begin{equation*}
		F_{C,\tilde{\mathcal{E}}'_{K,1}}(\varphi,\epsilon,s)=C(\varphi)^{-s}\sum_{p}b_{p,\epsilon}p^{-ms},
	\end{equation*}
	then
	\begin{equation*}
		F_{C,\tilde{\mathcal{E}}'_{K,\gamma}}(\varphi,\epsilon,s)=
		C(\varphi)^{-s}\sum_{
			\substack{
				d\in I^{+,\mu}_{e_G}\\
				\omega(d)=\gamma
			}
		}
        \Bigl(\prod_{p\mid d}b_{p,\epsilon}p^{-ms}\Bigr).
	\end{equation*}
	So Lemma~\ref{lemma: l(r,s)} implies the analytic continuation of $F_{C,\tilde{\mathcal{E}}'_{K,\gamma}}(\varphi,\epsilon,s)$, hence also of $F_{C,\mathcal{E}'_{K,\gamma}}(\varphi,s)$.
    In particular, the case when $\epsilon=e_{\mathcal{A}}$ shows that $F_{C,\mathcal{E}'_{K,\gamma}}(\varphi,s)$ could not be extended analytically to $s=1/m$, and gives the desired analytic behaviour near $s=1/m$.
    To be precise, there exist holomorphic functions $f_{\gamma,0},f_{\gamma,1},\dots,f_{\gamma,\gamma}$ at $s=1/m$ such that $f_{\gamma,\gamma}(1/m)$ is positive real and that
    \begin{equation*}
        F_{C,\tilde{\mathcal{E}}'_{K,\gamma}}(\varphi,s)=\sum_{j=0}^{\gamma}f_{\gamma,i}(s)\log\Bigl(\frac{1}{s-1/m}\Bigr)^{j}.
    \end{equation*}
	Then Tauberian Theorem~\cite[Appendix II Theorem I]{narkiewicz2014elementary} gives the desired estimate similar to the case when $\gamma=1$:
	\begin{equation*}
		N_C(\mathcal{E}_{K,\gamma},X)\geq N_C(\mathcal{E}'_{K,\gamma},X)\sim b_{\gamma}\frac{X^{1/m}}{\log X}(\log\log X)^{\gamma-1},
	\end{equation*}
    where $b_{\gamma}$ is a positive real number.
\end{proof}
We are finished with the proof of all the main statements of this section.
In summary, Lemma~\ref{lemma: estimate for upper bound} shows an upper bound of $N_C(\mathcal{E}_{\gamma},X)$.
while Lemma~\ref{lemma: estimate for lower bound} shows an lower bound by
\begin{equation*}
    N_C(\mathcal{E}_{\gamma},C)\geq N_C(\mathcal{E}_{K,\gamma},X).
\end{equation*}
So, we've proven that the Hypothesis~\ref{hypothesis: comparison between field counting results}(1) is true for $((\mathcal{E},C),\Omega,\mathcal{P})$.
To be precise, for each non-negative integer $\gamma$ we have
\begin{equation*}
    N_C(\mathcal{E}_{\gamma},X)=o(N_C(\mathcal{E}_{\gamma+1},X)).
\end{equation*}

\subsection{Examples}
In this section let us discuss some applications of the main results of this section via examples, including the proof of Theorem~\ref{thm: S1 dist of class group sextic case} and~\ref{thm: S1 dist of class group cubic case}.
\begin{example}[Sextic fields with $A_4$-closure]
	This is the prototype of Section~\ref{sec: general case} and even the whole paper.
	Let $A_4\subseteq S_4$ be the alternating group of order $4$.
	Consider the map $f:A_4\hookrightarrow S_6$ induced by $(12)(34)\mapsto(1)(2)(34)(56)$, and $(123)\mapsto(135)(246)$, and let $G$ be the image of $A_4$.
	It is a transitive permutation group in $S_6$.
	The stabilizer $\operatorname{Stab}(1)$ in $G$ of $1$ is a cyclic group of order $2$, generated by the image of $(12)(34)$.
	Let $\mathcal{E}:=\mathcal{E}(G)$ be the set of sextic fields with $A_4$-closure.
	Clearly $G$ admits the following short exact sequence
	\begin{equation*}
		1\to V_4\to G\to C_3\to1,
	\end{equation*}
	where $C_3$ acts on $V_4$ by permuting the generators.
	Define
        \begin{equation*}
            \Omega:=\{g\in G\mid r_g=2\},
        \end{equation*}
    which is equivalent to the image of $V_4\backslash\{e_G\}$ in $G$.
	For any non-negative integer $\gamma$, write $\mathcal{E}_{\gamma}:=\mathcal{E}_{\Omega,\gamma}^{\mathcal{P}}$.
	In other words, $L\in\mathcal{E}_{\gamma}$ means that there are exactly $\gamma$ primes $p\nmid6$ such that its inertia generator is contained in $\Omega$.
	
	Though it seems that we cannot apply the results of this section to it directly.
	Note that every field $L\in\mathcal{E}$ admits a Galois closure $M$ which is a Galois $A_4$-field.
    The study of $\mathcal{E}$ and the study of $A_4$-fields are essentially the same once we give the definition of the counting function $C$.
	The (absolute) discriminant of the sextic field $L\in\mathcal{E}$ could be written as
	\begin{equation*}
		d_L=d_K^2\operatorname{Nm}_{K/\mathbb{Q}}(\mathfrak{d}_{L/K})
	\end{equation*}
	where $K$ is the Galois cubic subfield of $L$, and $\mathfrak{d}_{L/K}$ is the relative discriminant.
        
    Let $p\geq5$ be a rational prime that is ramified in $L/\mathbb{Q}$.
    Its inertia subgroup $I_p$ is cyclic, so it is either isomorphic to $C_2$ or $C_3$.
    If $I_p\cong C_3$, then the normalizer of $I_p$ is itself.
    So, $I_p=G_p$, and $L/K$ is unramified.
    This implies that $p^2\| d_K$ and $p\nmid\operatorname{Nm}_{K/\mathbb{Q}}(\mathfrak{d}_{L/K})$.
        
    If $I_p\cong C_2$, then the normalizer of $I_p$ is isomorphic to $V_4$.
    However, whether we have $G_p\cong V_4$ or $G_p=I_p$, the prime $p\mathscr{O}_K=\mathfrak{p}_1\mathfrak{p}_2\mathfrak{p}_3$ must split in $K/\mathbb{Q}$.
    Because of the action of $C_3$ on $V_4$, there are exactly $2$ of the primes, say $\mathfrak{p}_1$ and $\mathfrak{p}_2$, ramified in $L/K$.
    Therefore, $p\nmid d_K$ and $p^2\|\operatorname{Nm}_{K/\mathbb{Q}}(\mathfrak{d}_{L/K})$.

    This shows that $d_L$ could be realized as a counting function.
    To be precise, let $\mathcal{F}$ be the set of Galois $A_4$-fields, and define $\mathcal{F}_{\gamma}:=\mathcal{F}_{\Omega,\gamma}^{\mathcal{P}}$.
    The map $c_G:G\to\mathbb{Z}_{\geq0}$ is defined by $c_G(g)=2$ for $g\in\Omega$ and $c_G(g)=4$ for $g\in G\backslash(\Omega\cup\{e_G\})$.
    This induces a counting function $C$ for $\mathcal{F}$.
    And for each $M\in\mathcal{F}$ with a subfield $L\in\mathcal{E}$, we see that $C(M)=d_L$ up to wildly ramified primes.    
	So, we have a nontrivial example that satisfies the set-up of the main results of this section, and Theorem~\ref{thm: S1 dist of class group sextic case} is proved immediately.
	Note that Definition~\ref{def:set of fields} defines the set of fields up to isomorphism classes instead of fields in a fixed algebraic closure.
    So, by Lemma~\ref{lemma of general case}, for each $0<\varepsilon<1$, we have
        \begin{equation*}
            N_d(\mathcal{E}_0,X)=N_C(\mathcal{F}_0,X)\ll X^{1/(3-\varepsilon)}.
        \end{equation*}
    For each positive integer $\gamma$, we have
        \begin{equation*}
            N_d(\mathcal{E}_{\gamma},X)=N_C(\mathcal{F}_{\gamma},X)\asymp\frac{X^{1/2}}{\log X}(\log\log X)^{\gamma-1}.
        \end{equation*}
    This implies that Hypothesis~\ref{hypothesis: comparison between field counting results}(1) is true for $((\mathcal{E},d),\Omega,\mathcal{P})$.
    So, Theorem~\ref{thm: estimate of class group} and~\ref{thm: 0-prob and infty-moment} imply that Theorem~\ref{thm: S1 dist of class group sextic case} is true.
\end{example}
Finally, let us apply the method of this section to the non-Galois cubic fields, i.e., the proof of Theorem~\ref{thm: S1 dist of class group cubic case}.
\begin{example}[Cubic fields]
	Let $G=S_3$, and $\mathcal{E}:=\mathcal{E}(G)$ be the set of non-Galois cubic fields.
	The group $G$ admits the short exact sequence
	\begin{equation*}
		1\to C_3\to G\to C_2\to1
	\end{equation*}
	where $C_2$ acts on $C_3$ in a nontrivial way.
	Let $\Omega:=C_3\backslash\{e_G\}$.
	For each non-negative integer $\gamma$, write $\mathcal{E}_{\gamma}:=\mathcal{E}_{\Omega,\gamma}^{\mathcal{P}}$.
	In other words, a cubic field $K_3\in\mathcal{E}_{\gamma}$ if there are exactly $\gamma$ primes $p\nmid6$ such that its tame inertia generator is contained in $\Omega$.
    The condition simply means that $p$ is totally ramified in $K_3$.
	The counting functions discriminant and product of ramified primes do not satisfy our conditions.
	So we need to define a special one so that the method of this section works.
	
	For each non-Galois cubic field $K_3$, let $K_6$ be its Galois closure and $K_2:=K_6^{C_3}$ be the associated quadratic number field.
	Note that if we write $d_{K_3}=d_{K_2}f^2$, where $p\mid f$ implies that $p$ is totally ramified in $K_3$.
	Or equivalently $f^4=\operatorname{Nm}_{K_2/\mathbb{Q}}(\mathfrak{d}_{K_6/K_2})$ where $\mathfrak{d}$ means the relative discriminant.
	Define
	\begin{equation*}
		C(K_3):= d_{K_3}d_{K_2}^2= d_{K_2}^3f^2.
	\end{equation*}
	One can check that this counting function could also be realized by Definition~\ref{def: product of ramified primes} in the following sense.
	Let $c_G:G\to\mathbb{Z}$ be the map induced by $c_G((12))=3$ and $c_G((123))=2$.
	For $p\mid6$, let $c_p$ be a $\mathbb{Z}_{\geq0}$-valued map on the set of specifications at $p$.
	Then we obtain a counting function $C'$ for $S_3$-fields induced by $c_G$ and $c_2,c_3$.
	Since $K_3=K_6^{\langle(23)\rangle}$, the function $C'$ induces a counting function for $K_3\in\mathcal{E}$, still denoted by $C'$.
	In particular whenever $K_3$ is tamely ramified, $C'(K_3)=C(K_3)$.
	In other words, they differ only at $2$ and $3$.
	
	Let $\mathcal{F}$ be the set of Galois $S_3$-fields.
        Now we can apply Lemma~\ref{lemma of general case} and obtain for each positive integer $\gamma$ the following estimate
	\begin{equation*}
		N_C(\mathcal{E}_{\gamma},X)=N_C(\mathcal{F}_{\gamma},X)\asymp\frac{X^{1/2}}{\log X}(\log\log X)^{\gamma-1}.
	\end{equation*}
        Also, the work of Davenport and Heilbronn~\cite{Davenport1971Cubic} implies that there exists some real number $c_0>0$ such that
        \begin{equation*}
            N_C(\mathcal{E}_0,X)\sim c_0X^{1/3},
        \end{equation*}
        for there is no totally ramified primes and for each $K\in\mathcal{E}_0$ we have $C(K)=d_K^3$, where $d$ is the usual discriminant.
        This also shows that in general the lower bound estimate from Lemma~\ref{lemma of general case} could not be improved in the sense of the main term, unless we have more conditions on the counting function $C$.
	By Theorem~\ref{thm: estimate of class group}, there exists some constant $c=c(6)$, for each $K_3\in\mathcal{E}$, we have
	\begin{equation*}
		\operatorname{rk}_3\operatorname{Cl}_{K_3}\geq\#\{p\in\mathcal{P}\mid p\text{ is totally ramified}\}-c=\#\mathcal{P}_{\mathbb{Q}}(\Omega,K_3)-c.
	\end{equation*}
	By Theorem~\ref{thm: 0-prob and infty-moment}, we obtain the following:
	for each non-negative integer $r$, we have that
	\begin{equation*}
		\mathbb{P}_C(\operatorname{rk}_3\operatorname{Cl}_{K_3}\leq r)=0,
	\end{equation*}
	and 
	\begin{equation*}
		\mathbb{E}_C(\lvert\operatorname{Hom}(\operatorname{Cl}_{K_3},C_3)\rvert)=+\infty.
	\end{equation*}
	Moreover, we can see that the statement remains true for all counting functions
	\begin{equation*}
		C(K_3):=d_{K_3}d_{K_2}^n,
	\end{equation*}
	where the integer $n\geq2$.
	This proves Theorem~\ref{thm: S1 dist of class group cubic case}.	
\end{example}

\section*{Acknowledgement}
The author would like to thank Yuan Liu for many useful conversations, 
and thank Melanie Wood and an anonymous referee for many helpful comments，
The author is supported by Jiangsu Funding Program for Excellent Postdoctoral Talent 2023ZB064.

\section*{Conflict of Interest Statement}

The author declares that there is no potential conflict.

\section*{Data Availability Statement}
The author claims that the data supporting the study of this paper are available within the article.

\bibliographystyle{plain}
\bibliography{references}
\end{document}